\documentclass{article}

\usepackage{amsrefs,amssymb,amsmath,amsthm,amsfonts,epsfig,graphicx}
\usepackage[utf8]{inputenc}
\usepackage[T1]{fontenc}
\usepackage[english]{babel}
\usepackage{hyperref}
\hypersetup{
    colorlinks,
    citecolor=black,
    filecolor=black,
    linkcolor=black,
    urlcolor=black
}

\theoremstyle{plain}
\newtheorem{teo}{Theorem}[subsection]
\newtheorem{thm}[teo]{Theorem}
\newtheorem*{thm*}{Theorem}
\newtheorem{cor}[teo]{Corollary}
\newtheorem{lem}[teo]{Lemma}
\newtheorem{prop}[teo]{Proposition}		
\newtheorem{prob}[teo]{Problem}
\theoremstyle{definition}
\newtheorem{df}[teo]{Definition}
\newtheorem{exa}[teo]{Example}
\newtheorem{rmk}[teo]{Remark}

\DeclareMathOperator{\arc}{arc}

\DeclareMathOperator{\parts}{{\mathcal P}}

\DeclareMathOperator{\diam}{diam}
\DeclareMathOperator{\interior}{int}
\DeclareMathOperator{\dist}{dist}

\DeclareMathOperator{\card}{card}

\DeclareMathOperator{\cc}{Comp}

\newcommand{\azul}[1]{\textcolor{black}{#1}}

\newcommand{\clos}[1]{\overline{#1}}
\newcommand{\cwfin}{$\text{cw}_{\text F}$}

\newcommand{\U}{{\cal U}}
\newcommand{\A}{{\cal A}}

\newcommand{\R}    {\mathbb R}

\newcommand{\Z}  {\mathbb Z}

\newcommand{\continua}{{\mathcal C}}

\newcommand{\W}{\mathcal F}

\renewcommand{\epsilon}{\varepsilon}

\author{Alfonso Artigue}
\title{Dendritations of Surfaces}
\date{\today}
\begin{document}

\maketitle

\begin{abstract}
In this paper we develop a generalization of foliated manifolds in the context of metric spaces. 
In particular we study dendritations of surfaces that are defined as maximal atlases of compatible upper 
semicontinuous local decompositions into dendrites.
Applications are given in modeling stable and unstable sets of topological dynamical systems. 
For this purpose new forms of expansivity are defined.
\end{abstract}

\setcounter{tocdepth}{3}

\section{Introduction}
In this paper we develop a \emph{Theory of Foliations} from a viewpoint 
of \emph{Continuum Theory}. We define and study what we call \emph{continuumwise foliations} 
or simply \emph{cw-foliations} and as a special case \emph{dendritations}. 
These are generalizations of foliations of smooth manifolds, laminations, singular pseudo-Anosov foliations and 
the generalized foliations used by Hiraide to study expansive homeomorphisms. 
The idea is to consider monotone upper semicontinuous decompositions as \emph{local charts}.
We will not assume that the \emph{plaques} are distributed as a product structure as in standard foliation theory.
Moreover, two plaques in a common local chart have not even to be homeomorphic.

As we will show in Theorem \ref{thmCovCwExpFol}, cw-foliations are a conceptual framework to understand the 
distribution of local stable and unstable continua in the dynamical systems that we will consider. 
A motivation is to classify cw-expansive surface homeomorphisms. We say that $f$ is \emph{cw-expansive} (\emph{continuum-wise expansive}) \cite{Ka93} if 
there is $\delta>0$ such that if $\diam(f^n(A))<\delta$ for all $n\in\Z$ and $A$ is connected then 
$\card(A)=1$, i.e., $A$ is a singleton.
Important examples of cw-expansive dynamics are Anosov diffeomorphisms of a compact manifold of arbitrary dimension and 
pseudo-Anosov maps of compact surfaces with singular points and 1-prongs. 

For an Anosov diffeomorphism
local stable sets form foliated charts (at least $C^0$, see \cite{HPS}). 
If we consider an expansive homeomorphism of a compact surface, via \cites{L,Hi}, we know that local stable sets form 
a singular foliation. 
Recall that a homeomorphism $f$ is \emph{expansive} if there is $\delta>0$ such that if $\dist(f^n(x),f^n(y))<\delta$ for all $n\in\Z$ then $y=x$.
In fact, Hiraide and Lewowicz independently proved that expansive homeomorphisms of compact surfaces are conjugate to pseudo-Anosov diffeomorphisms. 
If we consider cw-expansivity on a compact surface we have that local stable sets may be far 
from determining foliations in the standard sense. 
Moreover, a local stable set may not even be a finite union of arcs.
For example, in \cite{ArAnomalous}, it is constructed a cw-expansive homeomorphism on a compact surface with a fixed point 
whose local stable set is connected but not locally connected, see \S \ref{secAnomalous}. 
This example suggests that the goal of classifying all the cw-expansive surface homeomorphisms requires new technology. 

In \S \ref{secCwNexp}, we introduce some definitions located between expansivity and cw-expansivity 
that will be called \emph{cwN-expansivity}. 
In a sense, a version of $N$-expansivity from a viewpoint of continuum theory. 
Recall that, given $N\geq 1$, we say that $f$ is \emph{N-expansive} \cite{Mo12} if there is $\delta>0$ such that 
if $A\subset X$ (an arbitrary subset) and $\diam(f^n(A))\leq\delta$ for all $n\in\Z$ then $\card(A)\leq N$ (i.e., $A$ has at most $N$ points).
We will say, see Definition \ref{defcwN}, that a homeomorphism is \emph{cwN-expansive} if there 
is $\delta>0$ such that if $A,B\subset X$ are continua, $\diam(f^n(A))<\delta$ for all $n\geq 0$ and 
$\diam(f^n(B))<\delta$ for all $n\leq 0$ then $\card(A\cap B)\leq N$.
Also, a homeomorphism is \emph{\cwfin-expansive} (where the $F$ means \emph{finite}) if there 
is $\delta>0$ such that if $A,B\subset X$ are continua, $\diam(f^n(A))<\delta$ for all $n\geq 0$ and 
$\diam(f^n(B))<\delta$ for all $n\leq 0$ then $A\cap B$ is a finite set.
These new forms of expansivity allow us to prove the local connection of stable sets and to conclude that 
they are dendrites. See Theorem \ref{thmCwfSuperficies}.
Recall that a \emph{dendrite} is a Peano continuum containing no simple closed curve, a \emph{continuum} is a compact connected metric space 
and a \emph{Peano continuum} is a locally connected continuum. 
It is known, see Theorem \ref{thmExConNT}, that such dendrites (the stable and 
unstable local sets mentioned above) have a uniform size, i.e., there is $\delta>0$ such that 
for all $x$ in the surface the stable and the unstable dendrites of $x$ meet the boundary of the disc 
centered at $x$ with radius $\delta$. 
In this way we arrive naturally to the concept of dendritic decomposition, see \S \ref{secDendDec}. 

The topic of topological decompositions has a long literature, see for example \cites{Sm,Ro,Da,Nadler2,ChaCha}. 
For applications in dynamical systems the reader is referred to \cites{ArAnomalous,CaPa,PX,BM,KTT}. 
A celebrated result proved by Moore in \cite{Mo25} states that 
the quotient space of an upper semicontinuous decomposition of the two-dimensional sphere 
in non-separating continua is again the sphere (assuming that the decomposition has at least two elements, in other case the quotient is a point). 
In some sense, Moore's decompositions are a generalization of zero-dimensional foliations, i.e., the decomposition in singletons. 
A standard foliated chart of $[0,1]\times[0,1]$ by horizontal lines has two properties: 
each plaque is an arc and the quotient space is an arc. 
By Proposition \ref{propLocChartDendritation}, a local chart of a dendritation has the following properties: each plaque is a dendrite and 
the quotient space is a dendrite. 
The mentioned properties of a standard foliated chart do not characterize the foliated chart, 
see Example \ref{exaCuasiBox}. 
In \S \ref{secFolBox} we give necessary and sufficient conditions for a decomposition to be 
a standard foliated chart.
In Corollary \ref{corCharFol} we conclude that a continuous and $\continua$-smooth dendritation is a standard foliation.
The terms \emph{continuous} and $\continua$-\emph{smooth} have a special meaning in this paper, see Definitions 
\ref{dfContDend} and \ref{dfCsmooth}.

In \cite{Hi89} Hiraide considered \emph{generalized foliations} 
for the study of expansivity from a topological viewpoint (see \S \ref{secConAtl}). 
However, this definition seems to be useful jointly with the pseudo orbit tracing property, 
and not in our case. 
For example, pseudo-Anosov singular foliations are not generalized foliations in the sense of Hiraide 
and they are our main examples of dendritations.
Our generalization of a foliation is designed to accompany cw-expansivity.

The concept of cwN-expansivity appears naturally in the deep study of the articles \cites{L,Hi}. 
In these papers expansive homeomorphisms of surfaces are classified. 
A careful reading reveals that several arguments can be done 
assuming cw1-expansivity instead of expansivity. 
The concept of cw1-expansivity was previously considered in \cite{Samba} 
where some topological properties of stable sets were proved.
In Theorem \ref{teoEquivPANo1prong} we show that every cw1-expansive homeomorphism of a compact surface is expansive. 
The key of this proof is Theorem \ref{thmCw1exp} where we give sufficient conditions for a cw1-expansive homeomorphism 
of a compact metric space to be expansive.
In Example \ref{exaCw1VsExp} we show that cw1-expansivity does not imply expansivity on arbitrary compact metric spaces 
(this space is not locally connected). 
In \S \ref{secPAS2} we show that there is a cw2-expansive homeomorphism of the two-dimensional sphere that is not 2-expansive.
It is a pseudo-Anosov with 1-prongs.

We obtained some general results on surface dendritations.
In Theorem \ref{thmGenCwfSup} we show that 
for every dendritation of a closed surface there is a residual set 
of points without ramifications.
In particular, generic leaves are one-dimensional submanifolds. 
This is a consequence of another result by Moore \cite{Mo28}, 
saying that at most a countable number of disjoint triods can be embedded in a plane.
In Theorem \ref{thmCwfSuperficies} we consider a \cwfin-expansive homeomorphism of a compact surface. 
We show that stable and unstable continua form dendritations. 
We also prove that: no leaf is a Peano-continuum, generic leaves are non-compact one-dimensional manifolds 
and in a dense subset of the surface stable and unstable leaves are topologically transverse.
One can feel that this result gives a nice \emph{generic picture} of what is a \cwfin-expansive homeomorphism of a compact surface. 
However, we think that the goal of classifying surface \cwfin-expansive homeomorphism is far from the present paper, not to mention cw-expansivity.

A brief sketch of the paper is as follows.
In \S \ref{secCwFolPeano} we introduce new forms of expansivity and we study the topology of stable and unstable sets. 
In \S \ref{secConThDec} we recall the main results from continuum theory that will be needed and 
study decompositions, the local charts of our foliations.
In \S \ref{secCwFoliations} we define and study cw-foliations on metric spaces. 
They are defined via atlases of upper semicontinuous decompositions. 
We study the induced partition of the space into leaves.
The contents of \S \ref{secCwFolPeano} and \S \ref{secCwFoliations} are independent.
In \S \ref{secStUnsCwFol} the results of the previous sections are joined to study the stable and the unstable 
cw-foliations defined by a cw-expansive homeomorphism of a metric space.
In \S \ref{secDendritations} we study some special cases of dendritations of surfaces. 
We give sufficient conditions to prove that they are (singular) foliations. 
Also, some properties of the stable and the unstable dendritations of a homeomorphism with 
some kind of expansivity are derived.
Throughout the paper several open problems are given.

The author thanks the referee for the careful reading and several suggestions that helped to improve the article.

\section{Variations of expansivity}
\label{secCwFolPeano}

We start presenting general results of cw-expansive homeomorphisms on Peano continua. 
In \S \ref{secCwNexp} we introduce cwN-expansivity and we summarize the main variations of 
expansivity that will be considered in this paper.
In \S \ref{secExamples} we present some examples 
that will be used to illustrate our results in subsequent sections.
In particular, in \S \ref{secPAS2} we give an example of a cw2-expansive homeomorphism on the two-dimensional sphere that is not 
2-expansive.
In \S \ref{secTopStaSet} basic topological properties of stable sets are stated. 
We recall the Invariant Continuum Theorem for such dynamics.
In \S \ref{secCapacitor} we introduce capacitors as a tool to understand 
what happens with unstable continua between two close stable plates. 
In \S \ref{secPartExp} we give another form of expansivity, called partial expansivity, 
that generalizes partial hyperbolicity of diffeomorphisms. We do not develop this concept in the 
paper, however we think that it could be of interest. 
In \S \ref{secRelExp} we consider the relationship between stable sets and stable continua. 
Also, we give sufficient conditions for a cwN-expansive homeomorphism to be N-expansive. 
For these purposes we introduce expansivity modulo an equivalence relation.

\subsection{CwN-expansivity}
\label{secCwNexp}
Let $f\colon X\to X$ be a homeomorphism of a compact metric space $(X,\dist)$. 
We say that $A\subset X$ is a \emph{subcontinuum} if $A$ is compact and connected.
Denote by $\continua(X)$ the space of subcontinua of $X$ 
$$\continua(X)=\{A\subset X: A\text{ is a non-empty continuum}\}.$$
In $\continua(X)$ we consider the Hausdorff metric. 
It is usually called \emph{hyperspace} of $X$ and has the remarkable properties of being compact (if $X$ is compact) 
and arc-connected (provided that $X$ is connected), see \cites{Nadler,Nadler2}.
For $\delta>0$ define the sets
\[
  \continua_\delta=\{A\in\continua(X):\diam(A)\leq\delta\},
\]
\[
  \continua^s_\delta=\{A\in \continua(X):f^n(A)\in\continua_\delta\text{ for all }n\geq 0\},
\]
\[
  \continua^u_\delta=\{A\in \continua(X):f^n(A)\in\continua_\delta\text{ for all }n\leq 0\}.
\]
The continua in $\continua^s_\delta$ are called $\delta$-\emph{stable} 
and those in $\continua^u_\delta$ are $\delta$-\emph{unstable}.
We also consider the sets
\begin{equation}
 \label{defCsCu}
 \begin{array}{l}
  \continua^s=\{A^s\in \continua(X):\lim_{n\to +\infty}\diam(f^n(A^s))= 0\},\\
  \continua^u=\{A^u\in \continua(X):\lim_{n\to -\infty}\diam(f^n(A^u))= 0\}.  
 \end{array}
\end{equation}
The continua in $\continua^s$ are called \emph{stable} 
and those in $\continua^u$ are \emph{unstable}.

\begin{rmk}
 The sets $\continua_\delta,\continua^s_\delta$ and $\continua^u_\delta$ are closed in $\continua(X)$.
\end{rmk}

\begin{df}
\label{defcwN}
  Given $N\geq 1$ we say that $f$ is \emph{cwN-expansive} if there is $\delta>0$ such that 
  if $A^s\in \continua^s_\delta$ and $A^u\in \continua^u_\delta$ then 
  $\card(A^s\cap A^u)\leq N$.
  In this case $\delta$ is a \emph{cwN-expansivity constant}.
\label{dfCwZ}
  We say that $f$ is \emph{\cwfin-expansive} if there is $\delta>0$ such that 
  if $A^s\in \continua^s_\delta$ and $A^u\in \continua^u_\delta$ then 
  $A^s\cap A^u$ is a finite set.
\end{df}

In Table \ref{tablaExp} the main variations of expansivity considered in this paper are summarized.

\begin{table}[ht]
\[
\begin{array}{ccccccc}
\hbox{expansivity}&\Leftrightarrow &\hbox{1-exp} & \Rightarrow & \hbox{{cw1-exp}} &&\\
&&\Downarrow && \Downarrow\\
&&\hbox{{2-exp}} & \Rightarrow & \hbox{{cw2-exp}}\\
&&\Downarrow && \Downarrow\\
&&\hbox{{\dots}} &  & \hbox{{\dots}}\\
&&\Downarrow && \Downarrow\\
&&\hbox{{N-exp}} & \Rightarrow & \hbox{{cwN-exp}}& & \\
&&&& \Downarrow\\
&& &  & \hbox{{\cwfin-exp}}&\Rightarrow & \hbox{{cw-expansivity}}\\
\end{array}
\]
\caption{Hierarchy of some generalizations of expansivity. 
The implications indicated by the arrows are easy to prove and hold 
for homeomorphisms on metric spaces.}
\label{tablaExp}
\end{table}

\subsection{Examples}
\label{secExamples}

In this section we explain the examples that we had in mind while developing this paper. 
We start with three classical families for which there are well established theories 
for modeling stable and unstable sets:
\begin{enumerate}
\item \emph{Anosov diffeomorphisms}. These diffeomorphisms are defined on smooth manifolds 
and are characterized by the uniform hyperbolicity on the tangent bundle.
Considering the distribution of stable and unstable sets,
they are the most regular kind of expansive homeomorphisms 
because they form continuous foliations (see \cite{HPS}). 
\item \emph{Smale spaces}. These are expansive homeomorphisms of compact metric spaces with local product structure (equivalently, canonical coordinates 
 or pseudo-orbit tracing property). 
 The category includes Smale's basic sets of Axiom A diffeomorphisms. 
 Stable and unstable sets can be modeled with \emph{Hiraide's generalized foliations} (see \S \ref{secConAtl}).
 \item \emph{Pseudo-Anosov diffeomorphisms}. These diffeomorphisms are defined on compact surfaces. 
 They have local product structure except for a finite number of points called \emph{singular}. 
 Stable and unstable sets form \emph{singular foliations} (see for example \cite{Hi}). 
\end{enumerate}
Next we describe other examples that will be essential in the development of the paper. 
Except for the first one, we assume that the reader is familiar with Smale's \emph{derived from Anosov diffeomorphisms} 
(which in particular is an interesting example of a Smale space).

\subsubsection{A pseudo-Anosov with 1-prongs}
\label{secPAS2}
The dynamics of pseudo-Anosov diffeomorphisms is not simple, at least from author's viewpoint. 
In this section we wish to discuss in detail some properties of a special example, 
a pseudo-Anosov with 1-prongs of the sphere. 
\azul{This example has been considered several times in the literature. 
It has some properties that may not be easy to predict at first sight.
In \cite{WaltersET}*{Example 1, p. 140} Walters considered it to show that 
a factor of an expansive homeomorphism may not be expansive.
In \cite{PaPuVi}*{\S 2.4} it is proved that 
the local stable set of some points is not locally connected.
In \cite{PaVi} it is shown that it is not entropy expansive, in fact they show that 
there are arbitrarily small horseshoes.
We will show that it is cw2-expansive but not $N$-expansive, for all $N\geq 1$.}

\azul{The example is as follows.} Let $T^2=S^1\times S^1$ be the two dimensional torus where $S^1=\R/\Z$. 
 Consider the equivalence relation $p\sim -p$ for $p\in T^2$. 
 The quotient space is a two-dimensional sphere $S^2=T^2/\sim$. 
 Denote by $\pi\colon T^2\to S^2$ the canonical projection.
 On the torus consider the Anosov diffeomorphism $\tilde f\colon T^2\to T^2$ defined by $\tilde f(x,y)=(2x+y,x+y)$. 
 Define $f\colon S^2\to S^2$ by $f(\pi(p))=\pi(\tilde f(p))$ for all $p\in T^2$.
\azul{For a more detailed construction the reader is referred to the works mentioned above.}
\begin{prop}
\azul{The homeomorphism $f$ is cw2-expansive.}
\end{prop}

\begin{proof}
Denote by $\W^s$ and $\W^u$ the stable and the unstable singular foliations of $f$, respectively. 
These are transverse foliations except at the singularities. 
Singular points are 1-prongs and the foliations looks as in Figure~\ref{figPAS2Cantor}. 
Then, a small arc of the stable foliation intersects in at most two points 
an unstable arc. 
Thus, the proof is reduced to show that every stable continuum is contained in a stable leaf.
Arguing by contradiction, let $C\subset S^2$ be a stable continuum that it is not contained in a stable leaf. 
Then there is a hyperbolic periodic point $p\in S^2$ such that $\W^s(p)\cap C\neq\emptyset$, \azul{where $\W^s(p)$ denotes the stable leaf trough $p$}. 
Since $p$ is hyperbolic we see that $\diam(f^n(C))$ cannot go to zero as $n\to+\infty$ because 
$C$ is not contained in $\W^s(p)$. This contradiction proves that every stable continuum is contained in a stable leaf. 
\end{proof}

As usual we define \emph{local stable} and \emph{unstable sets} as 
\begin{eqnarray*}
W^s_\epsilon(p)=\{x\in S^2:\dist(f^n(p),f^n(x))\leq\epsilon\text{ for all }n\geq 0\}\\
W^u_\epsilon(p)=\{x\in S^2:\dist(f^n(p),f^n(x))\leq\epsilon\text{ for all }n\leq 0\}
\end{eqnarray*}
respectively. 
The next result can be derived from \cite{PaVi}, 
however, since we think that more details can be given, a proof is included. 
The author learned this proof from J. Vieitez and J. Lewowicz. 
We will use the notation $\clos{A}$ for the closure of $A$.

\begin{prop}
\label{propPAS2Cantor}
For all $\epsilon>0$ there is a Cantor set $C$ such that $\diam(f^n(C))\leq\epsilon$ for all $n\in\Z$,
in particular 
$f$ is not $N$-expansive for all $N\geq 1$. 
\end{prop}

\begin{proof}
Let $q\in S^2$ be a 1-prong of $f$. 
Take $p\in S^2$ such that the orbit $\{f^n(p):n\geq 0\}$ is dense in $S^2$. 
A point with this property will be called \emph{transitive}.
Consider $\epsilon>0$. 
We will show that $W^s_\epsilon(p)$ contains a Cantor set contained in the unstable arc of $p$. 
Consider $n\geq 0$ such that there is $x_1\neq f^n(p)$ with $x_1\in W^s_{\epsilon/2}(f^n(p))\cap W^u_{\epsilon/2}(f^n(p))$. 
We have that $p_1=f^{-n}(x_1)\in W^s_{\epsilon/2}(p)\cap W^u_{\epsilon/2}(p)$. 
Also we can assume that $p_1$ is in the unstable arc of $p$, see Figure~\ref{figPAS2Cantor}. 
Since $f^n(p)$ and $f^n(p_1)$ are in a stable arc and 
$p$ is transitive we have that $p_1$ is transitive too. 
\begin{figure}[ht]
\center
  \includegraphics{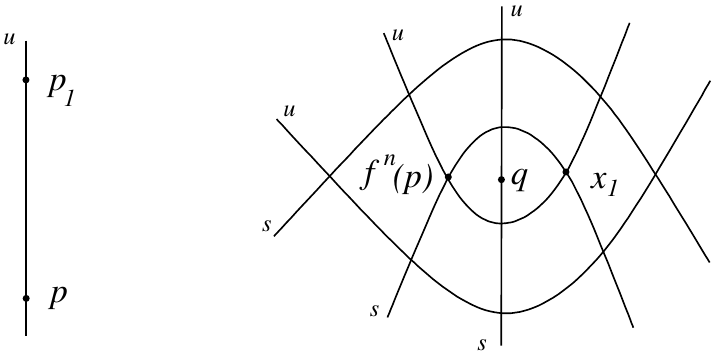} 
  \caption{The unstable arc of $p$ (left) and a neighborhood of the 1-prong $q$ (right). 
  Stable and unstable arcs are indicated with $s$ and $u$ respectively.}
  \label{figPAS2Cantor}
\end{figure}

Denote by $C_1=\{p,p_1\}$. 
For each $x\in C_1$ take a transitive point $y$ in the unstable arc of $p$ with 
$y\in W^s_{\epsilon/4}(x)$. 
Define $C_2$ as the set containing $C_1$ and these two new points. 
Again, for each $x\in C_2$ take a transitive point $y$ in the unstable arc of $p$ with 
$y\in W^s_{\epsilon/8}(x)$. Inductively we define a sequence of sets $\{C_n:n\geq 1\}$ 
with $|C_n|=2^{n+1}$. By construction each $C_n$ is contained in $W^s_\epsilon(p)$. 
Denote by $C=\clos{\cup_{n\geq 1}C_n}$. 
Since $W^s_\epsilon(p)$ is closed we have that $C\subset W^s_\epsilon(p)$. 
Also, $C$ has no isolated point and it cannot contain an arc because $f$ is cw-expansive (in fact cw2-expansive as we proved). 
Then, $C$ is a Cantor set contained in the unstable arc of $p$ and in $W^s_\epsilon(p)$.
\end{proof}

\subsubsection{\azul{Quasi-Anosov diffeomorphisms}}
\label{secQAnosov}

By definition, a \emph{quasi-Anosov} diffeomorphism is a $C^1$ diffeomorphism $f\colon M\to M$ 
of a compact manifold such that $\{\|df^n(v)\|:n\in\Z\}$ is unbounded for every non-vanishing tangent vector $v$, 
where $df^n$ is the differential of $f^n$ and $\|\cdot\|$ is the norm induced by a Riemannian metric on $M$.
They were characterized by Mañé \cite{Ma75} as Axiom A diffeomorphisms with quasi-transverse stable and unstable spaces.
Recall that \emph{Axiom A} means that the non-wandering set is hyperbolic and periodic points are dense in the non-wandering set. 
For an Axiom A diffeomorphism, at every point $x\in M$ the tangent space contains a contracting subspace $E^s_x$ and an expanding subspace $E^u_x$. 
The \emph{quasi-transversality} condition means that $E^s_x\cap E^u_x=0$ for all $x\in M$.

We proceed to sketch the construction of a particular quasi-Anosov diffeomorphism 
that is not Anosov \cite{FR}.
Consider two derived from Anosov diffeomorphisms 
$f_i\colon M_i\to M_i$, $i=1,2$, where $M_i$ is an $n$-torus,
$f_1$ is conjugate to $f_2^{-1}$ and
$f_1$ presents a codimension one shrinking repeller and a sink fixed point $p_1$.
Consequently, $f_2$ presents a codimension one expanding attractor and a source fixed point $p_2$.
Let $B_i$, $i=1,2$, be an open ball around $p_i$ such that $\clos{B_1}\subset f_1^{-1}(B_1)$ and 
$\clos{B_2}\subset f_2(B_2)$. 
Consider the manifolds with boundary $N_i=M_i\setminus B_i$
and a diffeomorphism $\varphi\colon 
\clos{f_1^{-1}(B_1)}\setminus B_1\to \clos{f_2(B_2)}\setminus B_2$
such that $M=N_1\cup N_2/x\sim \varphi(x)$ is a closed manifold and 
there is a diffeomorphism $f\colon M\to M$ extending the dynamics of $f_1$ and $f_2$.

In \cite{FR} it is proved that for $n=3$, there is a diffeomorphism $\varphi$ making $f$ quasi-Anosov.
As we said, the non-wandering set is hyperbolic. 
At wandering points stable and unstable manifolds are one-dimensional and its tangent lines are quasi-transverse.
Since $M$ is 3-dimensional, they are not transverse and $f$ is not Anosov.

\subsubsection{\azul{Q{$^r$}-Anosov diffeomorphisms}}
\label{secQrAnosov}
In \cite{ArRobNexp} the construction of \S \ref{secQAnosov}
was considered for the simpler case $n=2$. 
On a surface there is not enough space to construct a quasi-Anosov diffeomorphism (not being Anosov). 
However, the construction can be performed.
In this case the map $\varphi$ will introduce tangencies between stable and unstable manifolds at wandering points. 
Assuming that $f$ is of class $C^r$, if these tangencies are of order at most $r\geq 2$ we say that $f$ is \emph{Q}{$^r$}-\emph{Anosov}. 
In \cite{ArRobNexp} it is shown that the set of Q$^r$-Anosov diffeomorphisms 
of a closed surface is an open set (in the $C^r$ topology) of $r$-expansive diffeomorphisms, 
where $r$-\emph{expansive} means $N$-expansive with $N=r$.
These examples are Axiom A and show that $(N+1)$-expansivity does not imply $N$-expansivity.

\subsubsection{\azul{Anomalous cw-expansivity}}
\label{secAnomalous}

In \cite{ArAnomalous} it was considered another variation of the quasi-Anosov. 
On a surface, as in \S \ref{secQrAnosov}, 
it is defined a cw-expansive homeomorphism with what we called an \emph{anomalous saddle}. 
It is a fixed point whose stable set is connected but not locally connected. 
Let us briefly give a description. 
Let $E\subset \R^2$ be the set $(\{1\}\cup\{1+\frac1n\}_{n\geq 2})\times[0,1]$.
It is a countable union of vertical segments.
Consider the auxiliary map $T\colon\R^2\to\R^2$ defined as $T(p)=p/2$. 
Define $F^s=(\R\times\{0\})\cup_{n\geq 0} T^n(E)$ and $F^u=\{0\}\times \R$. 
The anomalous saddle in this case is the origin. 
In \cite{ArAnomalous} it is constructed a homeomorphism around $(0,0)$ such that 
its local stable set is $F^s$ and the unstable set is $F^u$ (intersected with a neighborhood of the origin). 
This anomalous saddle is \emph{inserted} in a $Q^r$-Anosov diffeomorphism of the previous section.
In the example, the intersection of local stable and unstable sets is totally disconnected, thus implying cw-expansivity.

%

\subsection{Invariant continua}
\label{secTopStaSet}
Our interest is centered at dynamical systems on surfaces, but some fundamental results on cw-expansivity hold for homeomorphisms on Peano continua, 
i.e., a connected, locally connected and compact metric space.
In this setting, Theorem \ref{thmExConNT} plays the role of the Invariant Manifold Theorem \cite{HPS} for hyperbolic diffeomorphisms on smooth manifolds.

The following result is a characterization of cw-expansivity in terms of stable and unstable continua. 
\azul{The direct part is known \cite{Ka93}*{\S 2}.}
The converse may be new, however it is quite direct from the definitions.

\begin{prop}
\label{propExConNT}
A homeomorphism $f\colon X\to X$ of a compact metric space is cw-expansive if and only if 
 there is $\epsilon^*>0$ such that 
  \begin{enumerate}
   \item $\continua^\sigma_{\epsilon^*}\subset \continua^\sigma$ and 
   \item for all $\epsilon>0$ there is $\delta>0$ such that 
   $\continua^\sigma\cap \continua_\delta\subset \continua^\sigma_\epsilon$
   \end{enumerate}
   for $\sigma=s$ and $\sigma=u$.
\end{prop}

\begin{proof}
 Direct. Let $\epsilon^*$ be such that if $A\subset X$ is connected and 
 $\diam(f^n(A))\leq\epsilon^*$ for all $n\in\Z$ then $A$ is a singleton. 
 Suppose that $\sigma=s$ and take $A\in \continua^s_{\epsilon^*}$, that is, 
 $\diam(f^n(A))\leq \epsilon^*$ for all $n\geq 0$. 
 If this diameter does not converge to 0 then there are $r>0$ and $n_k\to +\infty$ such that 
 $\diam(f^{n_k}(A))>r$ for all $k\geq 0$. 
 Since $\continua(X)$ is compact we can assume that $f^{n_k}(A)\to C$ in the Hausdorff metric. 
 We will show that 
 \begin{equation}
  \label{ecuEstCont}
  \diam(f^n(C))\leq\epsilon^*\text{ for all }n\in\Z.
 \end{equation}
 Since $n_k\to +\infty$, given $n\in\Z$ there is $k_0$ such that $n_k+n\geq 0$ for all $k\geq k_0$. 
 Then, $\diam(f^{n_k+n}(A))\leq \epsilon ^*$ for all $k\geq k_0$. 
 Since $f^{n_k+n}(A)=f^n(f^{n_k}(A))$ and $f^{n_k}(A)\to C$ we conclude (\ref{ecuEstCont}).
 Since $\diam (C)\geq r>0$ we have that $C$ is not a singleton. 
 This contradicts the cw-expansivity of $f$ and proves that $\continua^s_{\epsilon^*}\subset \continua^\sigma$.
 
 To prove the next part we argue by contradiction. 
 Suppose that there is $\epsilon>0$ and a sequence $A_k\in\continua^s$ such that 
 $\diam(A_k)\to 0$ and $\diam(f^{m_k}(A_k))>\epsilon$ for all $k\geq 0$ where $m_k\geq 0$. 
 Note that $m_k\to+\infty$ and that we can assume that $\epsilon<\epsilon^*$.
 Since $A_k\in\continua^s$ there is $n_k\geq m_k$ such that $\diam(f^n(A_k))\leq\epsilon$ for all $n\geq n_k$.
 We can take a subcontinuum $C_k\subset A_k$ and another divergent sequence $m'_k$ such that 
 $\diam(f^{m'_k}(C_k))=\epsilon$ and $\diam(f^n(C_k))\leq\epsilon$ for all $n\geq 0$. 
 \azul{Then, arguing as in the proof of (\ref{ecuEstCont}), we conclude that 
 a limit continuum of $f^{m'_k}(C_k)$ contradicts the cw-expansivity}. 

 Converse. Let us show that $\epsilon^*$ is a cw-expansivity constant. 
 Suppose that $\diam(f^n(A))\leq\epsilon^*$ for all $n\in\Z$. 
 Assume, by contradiction, that 
 there is $\epsilon\in (0,\diam(A)).$ 
 For this value of $\epsilon$ there is $\delta>0$ such that $\continua^s\cap \continua_\delta\subset \continua^s_\epsilon$. 
 Since $A\in \continua^u_{\epsilon^*}$ and $\continua^u_{\epsilon^*}\subset \continua^u$ 
 there is $m\leq 0$ such that $\diam(f^m(A))<\delta$. 
 Then $f^m(A)\in \continua^s\cap \continua_\delta$. 
 As $\diam(A)>\epsilon$ we have that $f^m(A)\notin\continua^s_\epsilon$. 
 This is a contradiction that finishes the proof.
\end{proof}

\azul{The next theorem states the existence of non-trivial stable and unstable continua through each point of the space. 
Moreover, these continua have diameter bounded away from zero. 
Their invariance is in fact given by Proposition \ref{propExConNT}.
The result was first proved independently by Hiraide and Lewowicz in \cites{Hi,L} for 
the classification of expansive homeomorphisms on compact surfaces. 
It was generalized by Kato in \cite{Ka2}*{Theorem 1.6} for cw-expansivity. 
A partial result was previously given by Mañé in \cite{Ma} to prove that minimal expansive homeomorphisms can only exist 
on totally disconnected spaces.}

\begin{rmk}
For the study of cw-expansive homeomorphisms it is essential to assume some kind of connection of the space. 
To illustrate this point consider that every homeomorphism of a Cantor set is cw-expansive. 
It turns that local connection is a good property to exploit the cw-expansivity. 
There is a minor loss of generality if, in addition, we assume that the space is a Peano continuum. 
Indeed, if $X$ is a locally connected compact metric space then it has a finite number of components. 
Therefore, a homeomorphism $f$ of $X$ will permute this components and 
taking a power of $f$ we will have a finite number of homeomorphisms of Peano continua.
\end{rmk}

For a set $\continua^*\subset\continua(X)$, as $\continua_\delta,\continua^\sigma$ and $\continua^\sigma_\epsilon$, 
define $\continua^*(x)=\{A\in\continua^*:x\in A\}$.

\begin{thm}[Invariant Continuum Theorem]
\label{thmExConNT}
  If $f$ is a cw-expansive homeomorphism of a Peano continuum $X$
  then for all $\epsilon>0$ there is $\delta>0$ such that 
  \begin{equation}
   \label{ecuICT}
   \continua^\sigma_\epsilon(x)\setminus\continua_\delta(x)\neq\emptyset
  \end{equation}
 for all $x\in X$ and $\sigma=s,u$.
%
%
\end{thm}

This theorem has the following direct consequences.

\begin{rmk}[Uniform size of stable continua]
 Note that (\ref{ecuICT}) implies that for all $x\in X$ there are stable and unstable continua through $x$ 
 \azul{of diameter greater than $\delta$. 
 Consequently, these continua meet} the boundary of 
 the ball $B_{\delta/2}(x)$. 
\end{rmk}

\begin{rmk}[No stable points]
\label{rmkNoEstablePoints}
 From the Invariant Continuum Theorem we see that if $f$ is a cw-expansive homeomorphism of a Peano continuum $X$
 then neither stable nor unstable continua have interior points. 
 This is because if $A\subset X$ is a stable set with an interior point $x$ then we can take an unstable continuum contained in 
 $A$ that contradicts the cw-expansivity. 
 In particular there are no Lyapunov stable trajectories.
\end{rmk}

\begin{rmk}[Surfaces with boundary]
\label{rmkNoboundary}
Let us explain why we do not consider surfaces with boundary. 
Suppose that $f\colon S\to S$ is a cw-expansive homeomorphism 
of a compact surface with boundary. 
By Brouwer's Theorem on the Invariance of Domain \cite{HW} 
we know that $\partial S$ 
is invariant by $f$.
Then, the restriction $f\colon \partial S\to\partial S$ is cw-expansive.
This gives a contradiction because, 
on one hand there are non-trivial stable continua in $\partial S$, 
and on the other hand every non-trivial continuum of $\partial S$
has interior points (relative to $\partial S$). 
This is the argument to prove that the circle admits no expansive homeomorphisms 
that the author learned from Lewowicz.
\end{rmk}

\begin{prob}
 Do cw-expansive homeomorphisms of compact manifolds with non-empty boundary exist? 
For example, does there exist a cw-expansive homeomorphism of the 3-ball?
Does the 3-sphere admit expansive or cw-expansive homeomorphisms?
\end{prob}

We say that $C\subset X$ \azul{\emph{separates} $X$} if $X\setminus C$ is not connected.
We will show 
in Theorem \ref{nosepara}, under a \emph{natural} assumption on the Peano space $X$, that no stable set separates. 
Previously we prove a topological lemma.

\begin{lem}
\label{lemThmNosepara}
If $X$ is a continuum and every point has arbitrarily 
small neighborhoods with connected boundary then 
for all $\epsilon>0$ there is $\delta>0$ such that 
if $A\subset X$ is a closed set that separates $X$ and $\diam(A)<\delta$ 
then there is a component $V$ of $X\setminus A$ with $\diam(V)<\epsilon$.
\end{lem}

\begin{proof}
  Given $\epsilon>0$ consider a finite open cover $\mathcal U$ of $X$ 
  such that $\diam(U)<\epsilon$ for all $U\in\mathcal U$ and $\partial U$ is connected. 
  Take $\delta>0$ such that if $\diam(A)<\delta$ then there is $U\in\mathcal U$ such that $A\subset U$. 
  Suppose that $A\subset X$ is closed, separates $X$ and $\diam(A)<\delta$. 
Take $U_*\in\mathcal U$ such that $A\subset U_*$. 
  
Let us show that $C=X\setminus U_*$ is connected. 
By contradiction, suppose that $C=A\cup B$ a union of disjoint, closed, non-empty sets. 
Since $\partial U_*$ is connected, we can assume that $\partial U_*\subset A$. 
Then, $B$ and $A\cup\clos{U_*}$ disconnects $X$, a contradiction that proves that $X\setminus U_*$ is connected.

Let $V_*$ be the component of $X\setminus A$ containing $X\setminus U_*$. 
Since $A$ separates $X$, there is at least another component. 
This component is contained in $U_*$ and has diameter smaller than $\epsilon$.
\end{proof}

\begin{teo}
\label{nosepara}
If $f$ is a cw-expansive homeomorphism of the Peano continuum $X$ and every 
point of $X$ has arbitrarily 
small neighborhoods with connected boundary
then no stable closed set separates $X$.
\end{teo}

\begin{proof}
Arguing by contradiction assume that $A\subset X$ is a closed stable set separating $X$. 
\azul{As $f$ is a homeomorphism, every iterate of $A$ separates $X$.}
\azul{Since it is stable, taking a positive iterate,} we can suppose that it is as small as we want. 
By Lemma \ref{lemThmNosepara} there is a small component $U$ of its complement. 
By Theorem \ref{thmExConNT} each point of $U$ has a stable continuum meeting $\partial U$. 
\azul{Since $\partial U\subset A$ we have that $\partial U$ is stable, therefore, $\clos{U}$ is a stable set.}
Since $\clos{U}$ has interior points we have a contradiction with Remark \ref{rmkNoEstablePoints}.
\end{proof}

\begin{rmk}
If a Peano continuum has no locally separating points then 
every point has arbitrarily small neighborhoods with connected boundary. 
See \cite{Jo} for a proof.
\end{rmk}

\begin{rmk}
 Theorem \ref{nosepara} holds if $X$ is a compact manifold with or without boundary. 
 In the one-dimensional case, intervals \azul{do not have} connected boundary, but since there are not cw-expansive 
 homeomorphisms on one-dimensional manifolds the theorem holds true. 
\end{rmk}

The following example shows that in Theorem \ref{nosepara} we need to assume 
that every point has arbitrarily small neighborhoods with connected boundary.

\begin{exa}
  Consider two copies of an Anosov diffeomorphism of the two-dimensional 
  torus identifying two fixed points. 
  The gluing point has not arbitrarily small neighborhoods with connected boundary. 
  Also, this point forms a stable set and it (locally and globally) separates. 
\end{exa}

\subsection{Capacitors}
\label{secCapacitor}

The following technical definition is based on the arguments of the proof of \cite{L}*{Lemma 2.3}.

\begin{df}
Given $x\in X$, $\epsilon,r>0$, an $(\epsilon,r,x)$-\emph{capacitor} is a triple $(A,G,B)$ such that: 
\begin{enumerate}
 \item $A,B,G\subset X$, $A$ and $B$ are disjoint continua, $x\in \clos G$, $G$ is open,
 \item $(\partial G)\cap B_r(x)\subset A\cup B$,
 \item there is a continuum $\gamma\subset \clos{G}\cap B_{r/2}(x)$ meeting $A$ and $B$,
 \item $G\subset B_\epsilon(A)$ or $G\subset B_\epsilon (B)$.
\end{enumerate}
In this case $A$ and $B$ are the \emph{plates} of the capacitor, 
$\epsilon$ is the \emph{separation} of the plates and $r$ is the \emph{radius}.
See Figure~\ref{figCap}.
\begin{figure}[ht]
\center
  \includegraphics{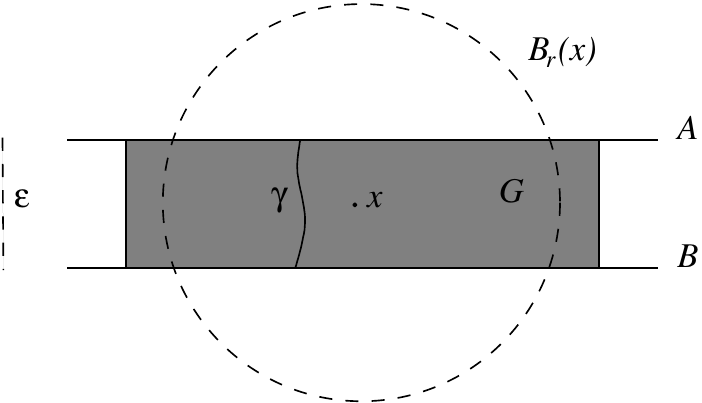} 
  \caption{A capacitor of separation $\epsilon$ and radius $r$ in the plane.}
  \label{figCap}
\end{figure}
We say that a capacitor has \emph{stable plates} if $A,B$ are stable sets for the homeomorphism $f\colon X\to X$.
\end{df}

In the next result we study unstable continua between two close stable plates.
It generalizes \cite{L}*{Lemma 2.3}
and it will be applied in Theorem \ref{thmCwfSuperficies} for the study of \cwfin-expansivity on surfaces. 
\begin{teo}
\label{teoCap}
Assume that $f\colon X\to X$ is a cw-expansive homeomorphism of a Peano continuum $X$.
Then, for all $r>0$ small there is $\epsilon>0$ such that if 
 $(A,G,B)$ is an $(\epsilon,r,x)$-capacitor 
\azul{with stable plates} then: 
 \begin{enumerate}
  \item for all $y\in B_{r/2}(x)\cap\clos{G}$ there is an unstable continuum from $y$ to $A\cup B$ 
  contained in $\clos{G}\cap B_r(x)$,
  \item there is an unstable continuum $C\subset\clos{G}\cap B_r(x)$ meeting $A$ and $B$. 
 \end{enumerate}
\end{teo}

\begin{proof}
Arguing by contradiction assume that there are 
$r>0$, a sequence of $(1/n,r,x_n)$-capacitors $(A_n,G_n,B_n)$ with stable plates
%
and $y_n\in B_{r/2}(x_n)\cap\clos{G_n}$ with no unstable continuum from $y_n$ to $A_n\cup B_n$ 
contained in $\clos{G_n}\cap B_r(x_n)$.
Since $X$ is a Peano continuum, by Theorem \ref{thmExConNT} 
for each $n$ there is an unstable continuum $C_n\subset B_r(x_n)$ 
containing $y_n$ and intersecting the boundary of $B_r(x_n)$. 
Since $(\partial G_n)\cap B_r(x_n)\subset A_n\cup B_n$  
we conclude that $C_n\subset G_n$. 
Taking subsequences, we can assume that $A_n\to A$ and $C_n\to C$ in the Hausdorff metric. 
By definition of capacitor we have that $G_n\subset B_{1/n}(A_n)$ 
which implies that $C\subset A$. 
Therefore $C$ is a continuum that is stable and unstable.  
Since $\diam(C)>0$ we have a contradiction with the 
cw-expansivity of $f$. 

Given $r>0$ consider $\epsilon>0$ satisfying the first item. 
Let $(A,G,B)$ be an $(\epsilon,r,x)$-capacitor. 
By definition, there is a continuum $\gamma\subset \clos{B_{r/2}(x)}$ 
\azul{and points $p^A,p^B$}
with $p^A\in A\cap \gamma$ and $p^B\in B\cap\gamma$.
Define 
\[
\begin{array}{l}
\gamma^A=\{y\in\gamma:\exists \hbox{ an unstable continuum from } y \hbox{ to }A \hbox{ contained in }\clos G\cap B_r(x)\},\\
\gamma^B=\{y\in\gamma:\exists \hbox{ an unstable continuum from } y \hbox{ to }B \hbox{ contained in }\clos G\cap B_r(x)\}.  
\end{array}
\]
The sets $\gamma^A$ and $\gamma^B$ are non-empty because they contain $p^A$ and $p^B$ respectively. 
They are closed sets and, as we have shown, they cover $\gamma$. 
Since $\gamma$ is connected, 
there is a point $z\in \gamma^A\cap\gamma^B$. 
\end{proof}

\begin{rmk}
\label{rmkAcordeon}
Let $f\colon M\to M$ be an expansive homeomorphism of a compact three-manifold.
In \cite{Vi2002} it is shown that local stable sets are 
locally connected 
if $f$ is smooth and without wandering points.
By Theorem \ref{teoCap} we know that a stable set of $f$ 
cannot be homeomorphic to the closure of the set
$$C=\{(x,y,z):-1\leq x\leq 1, 0<z\leq 1, y=\sin(1/z)\}.$$
\end{rmk}

\subsection{\azul{Partial expansivity}}
\label{secPartExp}

Consider a homeomorphism $f\colon X \to X$ of a compact metric space.
We will use some concepts of topological dimension. We refer the reader to \cite{HW} for the definitions and basic properties.

\begin{df}
\label{dePartExp}
Given an integer $d\geq -1$, we say that $f$ is \emph{partially expansive} 
with \emph{central dimension} $d$ and \emph{expansivity constant}
$\epsilon>0$ 
if for every non-trivial compact set $C\subset X$ with $\dim(C)>d$ 
there is $k\in\Z$ such that 
$\diam(f^k(C))\geq\epsilon$.
\end{df}

 As usual, we say that $f$ 
 is \emph{sensitive to initial conditions}
 if there is $\rho>0$ such that for all $x\in M$ and for all $r>0$ there are $y\in B_r(x)$ and 
 $n\in\Z$ such that $\dist(f^n(y),f^n(x))>\rho$.

\begin{prop}
\label{propPartExpMenosUno}
For a homeomorphism $f\colon X\to X$ the following hold:
\begin{enumerate}
 \item $f$ is expansive if and only if it is partially expansive with $d=-1$,
 \item $f$ is cw-expansive if and only if it is partially expansive with $d=0$,
\item if in addition $X$ is a compact manifold of dimension $n$ 
then $f$ is sensitive to initial conditions if and only if 
$f$ is partially expansive with central dimension $d=n-1$.
\end{enumerate}
\end{prop}

\begin{proof}[Sketch of the proof]
Since the arguments are quite direct we only give the details that we consider more relevant.
To prove the first part, note that by definition 
(see \cite{HW}) the condition $\dim(C)>-1$ means $C\neq\emptyset$. 
Then, $C\subset X$ with $\dim(C)>-1$ is non-trivial if and only if it has at least 
two points. To conclude the stated equivalence one has to note that 
$\diam(\{x,y\})=\dist(x,y)$.

The statement related to cw-expansivity follows because positive dimension is equivalent 
to contain a non-trivial continuum. 

For the last part we recall \cite{HW}*{Corollary 1, p. 46}
that if $X$ is a compact $n$-dimensional manifold 
and $C\subset X$ then $\dim(C)=n$ if and only if $C$ has non-empty interior.
\end{proof}

%
%

Let $\phi\colon\R\times X\to X$ be a continuous flow. 
We consider the following weak form of expansivity.
We say that a flow $\phi$ is \emph{separating} \cite{ArKinExp}
if there is $\delta>0$ (a \emph{separating constant}) such that if 
$\dist(\phi_t(x),\phi_t(y))<\delta$ for all $t\in\R$ then $y=\phi_s(x)$ 
for some $s\in\R$. 
Examples of separating flows are expansive flows in the sense of 
Bowen-Walters \cite{BW} and $k^*$-expansive flows as 
defined by Komuro \cite{Ko84}. 

\begin{prop}
If $\phi$ is a separating flow then
for all $T\neq 0$ the homeomorphism 
$\phi_T\colon X\to X$ is partially expansive with $d=1$.
\end{prop}

\begin{proof}
Let $\sigma>0$ be a separating constant 
and take $\delta>0$ such that 
if $C\subset X$ is a continuum and $\diam(C)<\delta$ then $\diam(\phi_s(C))<\sigma$ 
for all $s\in[-T,T]$. 
Then, if $\diam(\phi_{nT}(C))<\delta$ for all $n\in\Z$ we have that 
$\diam(\phi_s(C))<\sigma$ for all $s\in\R$. 
Therefore, $C$ is an orbit segment and consequently $\dim(C)\leq 1$. 
\end{proof}


\subsection{Relative expansivity}
\label{secRelExp}
%

Given a homeomorphism $f\colon X\to X$ 
it is usual to define the \emph{stable set} of $x\in X$ 
as 
\[
 W^s(x)=\{y\in X: \lim_{n\to +\infty}\dist(f^n(x),f^n(y))= 0\}.
\]
and the \emph{unstable set} 
\[
 W^u(x)=\{y\in X: \lim_{n\to -\infty}\dist(f^n(x),f^n(y))= 0\}.
\]
Consider the equivalence relation $x\sim_s y$, $x,y\in X$, if there is a continuum $C\subset X$ such that 
$x,y\in C$ and $\diam(f^n(C))\to 0$ as $n\to +\infty$. 
Similarly we define $\sim_u$ (taking $n\to-\infty$). 
The equivalence class of $x$ will be denoted as $\W^s(x)$ (and $\W^u(x)$). 
We will give conditions that allow us to prove that $W^s=\W^s$ and $W^u=\W^u$.

\begin{df}
Let $\W$ be a partition of $X$.
We say that $f$ is \emph{separating mod} $\W$ if 
 there is $\delta>0$ such that if $\dist(f^n(x),f^n(y))<\delta$ for all $n\geq 0$ then $y\in \W(x)$
\end{df}

\begin{prop}
 If $f$ is separating mod $\W^s$ then $W^s=\W^s$.
\end{prop}

\begin{proof}
Assume that $f$ is separating mod $\W^s$ with $\delta$ as in the definition.
Suppose that $\dist(f^n(x),f^n(y))\to 0$ as $n\to+\infty$. 
 Take $k$ such that for all $n\geq k$ it holds that $\dist(f^n(x),f^n(y))<\delta$. 
 Then there is a stable continuum $C$ containing $f^k(x),f^k(y)$. 
 This implies that $x,y$ are in the stable continuum $f^{-k}(C)$. 
 Then $W^s=\W^s$.
\end{proof}

\begin{prob}
 Assuming that $f$ is cw-expansive does the condition $\W^s=W^s$ imply 
 that it is separating mod $\W^s$?
\end{prob}

\begin{prop}
 The pseudo-Anosov with 1-prongs on the two-dimensional sphere 
 given in \S \ref{secPAS2} is not separating mod $\W^s$.
\end{prop}

\begin{proof}
\azul{By Proposition \ref{propPAS2Cantor}}, for all $\epsilon>0$ 
 there is a Cantor set $C$ contained in an unstable arc and contained 
 in $W^s_\epsilon(p)$ for some $p$ in this unstable arc. 
 Since $\W^s(p)$ cuts the unstable arc in a countable set, 
 there is $a\in C$ that is not in $\W^s(p)$. 
 This proves that $f$ is not separating mod $\W^s$.
\end{proof}

\begin{prob}
Does the example in \S \ref{secPAS2} satisfy $W^s=\W^s$? 
The solution could be simple but we were not able to solve it.
\end{prob}

\begin{df}
\label{dfExpModFsigma}
 We say that $f$ is \emph{expansive mod} $\W^s$ if for all $\epsilon>0$ there is $\delta>0$ such that 
 if $\dist(f^n(x),f^n(y))<\delta$ for all $n\geq 0$ then $x$ and $y$ are in a common $\epsilon$-stable continuum. 
 \end{df}

 Note that every homeomorphism expansive mod $\W^s$ is separating mod $\W^s$.

\begin{prop}
\label{propCwnNexpExpMod}
 If $f\colon X\to X$ is a cwN-expansive homeomorphism of a Peano continuum 
 and $f$ is expansive mod $\W^s$ and $f^{-1}$ is expansive mod $\W^u$ then $f$ is N-expansive.
\end{prop}

\begin{proof}
Let $\alpha>0$ be a cwN-expansivity constant for $f$ and take $\epsilon=\alpha/2$. 
Consider $\delta$ from Definition \ref{dfExpModFsigma} and 
suppose that $\diam(f^n(C))<\delta$ for all $n\in\Z$. 
For $x,y\in C$ we have that $\dist(f^n(x),f^n(y))<\delta$ for all 
$n\geq 0$. 
Then, there is an $\epsilon$-stable continuum containing $x,y$. 
We have that there is a $2\epsilon$-stable continuum containing $C$. 
Similarly, there is a $2\epsilon$-unstable continuum containing $C$. 
Since $2\epsilon=\alpha$ is a cwN-expansivity constant for $f$ we have that 
$\card(C)\leq N$. This proves that $\delta$ is an N-expansivity constant for $f$.
\end{proof}

\begin{prob}
\label{probExpModFs}
 It seems that pseudo-Anosov diffeomorphisms of surfaces without 1-prongs are 
 expansive mod $\W^s$. 
 We know that the pseudo-Anosov with 1-prongs of \S \ref{secPAS2} is not expansive mod $\W^s$.
 Are the examples in \S \ref{secAnomalous} and \S \ref{secQrAnosov} expansive mod $\W^s$? 
 It would be interesting to understand which cw-expansive surface homeomorphisms are
 expansive mod $\W^s$.
\end{prob}

\section{Continuum theory and decompositions}
\label{secConThDec}

In this section we review some results from continuum theory that will be used throughout the article. 
Also, we study decompositions, that will play the role of \emph{foliated charts} in the next section.

\subsection{\azul{Background on Continuum Theory}}

We recall that a \emph{continuum} is a compact connected metric space. 
General references for continuum theory are \cites{Nadler2,Kur1,Kur}.

\subsubsection{Partitions and monotone restrictions}
\label{secPartandMonRes}
Let $(X,\dist)$ be a compact metric space and denote by $\parts(X)$ the set of subsets of $X$.
A \emph{partition} of $X$ is a function 
$Q\colon X\to\parts(X)$ such that:
\begin{enumerate}
  \item $x\in Q(x)$ for all $x\in X$,
  \item $y\in Q(x)$ if and only if $x\in Q(y)$ and 
  \item $x\in Q(y)$ and $y\in Q(z)$ implies $x\in Q(z)$.
\end{enumerate}
A partition $Q$ is \emph{monotone} if each $Q(x)$ is connected. 
Given $Y\subset X$ and $x\in Y$ denote by $\cc_x(Y)$ the component of $Y$ containing $x$.
For a partition $Q\colon X\to \parts(X)$ 
and $Y\subset X$ define the \emph{monotone restriction} $Q|^m_Y\colon Y\to\parts(Y)$ as 
\begin{equation}
 \label{ecuRest}
 Q|^m_Y(x)=\cc_x(Q(x)\cap Y)
\end{equation}
for all $x\in Y$.
\begin{prop}
\label{propRestRest}
 If $Q$ is a partition of $X$ and $Z\subset Y\subset X$ then $$(Q|^m_Y)|^m_Z=Q|^m_Z.$$ 
\end{prop}

\begin{proof}
From the definition (\ref{ecuRest}) we see that we have to show that 
\[
 \cc_x(\cc_x(Q(x)\cap Y)\cap Z)=\cc_x(Q(x)\cap Z)
\]
for all $x\in Z$. 
To prove the inclusion $\supset$ consider a connected set $C\subset Q(x)\cap Z$ such that $x\in C$. 
Since $Z\subset Y$ we have that $C\subset Q(x)\cap Y$. 
Then, $C\subset \cc_x(Q(x)\cap Y)$ because $x\in C$.
Given that $C\subset Z$ we conclude that $C\subset \cc_x(\cc_x(Q(x)\cap Y)\cap Z)$. 
Then $\cc_x(Q(x)\cap Z)\subset\cc_x(\cc_x(Q(x)\cap Y)\cap Z)$.
The converse inclusion is easier to prove. It follows from the fact that $\cc_x(Q(x)\cap Y)\subset Q(x)$.
\end{proof}

We say that a partition $Q$ is \emph{upper semicontinuous} if for all $x\in X$ and every open set $U$ containing $Q(x)$ 
there is a neighborhood $V$ of $x$ such that if $y\in V$ then 
$Q(y)\subset U$. 

\begin{rmk}
 \label{rmkUpSCont}
 $Q$ is upper semicontinuous if and only if given $x_n\to x$ such that $Q(x_n)\to C$ in the Hausdorff metric 
then $C\subset Q(x)$. 
The upper semicontinuity of $Q$ implies that each $Q(x)$ is a closed subset of $X$.
\end{rmk}

We say that a partition $Q$ is \emph{continuous at} $x\in X$ 
if for every $x_n\to x$ we have that $Q(x_n)\to Q(x)$ in the Hausdorff metric. 
We say that $Q$ is \emph{continuous} if it is continuous at every 
$x\in X$.
A set $G\subset X$ is \emph{residual} if it is a 
countable intersection of open and dense subsets of $X$.

\begin{prop}
\label{propCoc1}
If $Q$ is an upper semicontinuous partition of $X$ then: 
\begin{enumerate}
 \item \cite{Kur}*{p. 70-71} there is a residual subset $G\subset X$ such that $Q$ is continuous at every $x\in G$, 
 \item \cite{Nadler2}*{Theorem 3.9} if in addition $X$ is a continuum then $X/Q$ with its quotient topology is a continuum.\footnote{It is clear that $X/Q$ is compact and connected since it is the quotient of the continuum $X$. 
 In \cite{Nadler2}*{Theorem 3.9} it is shown that $X/Q$ is metrizable.}
 \end{enumerate}
\end{prop}

\subsubsection{Local connection}
\label{secLocConBack}
A continuum is \emph{hereditarily locally connected} if every subcontinuum is locally connected. 
A \emph{convergence continuum} $A$ of a compact metric space $X$ is a 
non-trivial subcontinuum of $X$ for which
there is a sequence of continua $A_i\subset X$ such that $A_i\to A$ in the Hausdorff metric, 
$A_i\cap A=\emptyset$ and $A_i\cap A_j=\emptyset$ for all $i\neq j$. 

\begin{teo}[\cite{Nadler2}*{Theorem 10.4}]
\label{teoLocConConvCont}
A continuum $X$ is hereditarily locally connected if and only if $X$ contains no 
convergence continuum. 
\end{teo}

\begin{teo}[\cite{HY}*{Theorem 3-17}]
\label{teoLocConArc}
Every connected, locally connected, complete metric space is arc-connected. 
\end{teo}

\begin{teo}[Sierpi\'nski's Theorem \cite{Kur1}*{p. 218}] 
\label{teoSierpinski} A continuum $C$ is Peano
if and only if for all $\epsilon>0$ there is a finite cover $C_1,\dots,C_n$ 
of $C$ by connected sets of diameter less than $\epsilon$. 
\end{teo}

\subsubsection{Unicoherence}
\label{secUnicoh}
A continuum $C$ is \emph{unicoherent} 
if given subcontinua $A,B\subset C$ such that $A\cup B=C$ then $A\cap B$ 
is connected. 
A continuum is \emph{hereditarily unicoherent} if every subcontinuum is 
unicoherent. 
To show the difference between these concepts and for future reference let us recall the following result.

\begin{thm}[Janisewski's Theorem \cite{Kur}*{p. 506}] 
\label{teoJanis}
 The union of two subcontinua $A,B$ of the 2-sphere disconnects the sphere if and only if $A\cap B$ is disconnected.
\end{thm}

By Janisewski's Theorem we see that the 2-sphere is a unicoherent continuum. 
It is not hereditarily unicoherent because it contains a circle, which is not unicoherent.

\subsubsection{Dendrites}
\label{secBackDend}
A \emph{dendrite} is a Peano continuum $X$ containing no simple closed curve.
The points of a dendrite are classified as: 
\emph{end point} if its complement is connected,
\emph{ramification} or \emph{branch} point if its complement has at least three components,
\emph{regular point} if its complement has two components.
The following results from \cite{Nadler2}*{\S 10} summarizes several properties of dendrites.

\begin{teo}
\label{teoSubContDend}
Every subcontinuum of a dendrite is a dendrite.
\end{teo}

\begin{teo}
\label{teoDenRamNum}
The set of all the ramification points of a dendrite is countable.
\end{teo}

\begin{teo}
\label{teoCharDend}
If $X$ is a compact metric space then the following statements are equivalent: 
\begin{enumerate}
 \item $X$ is a dendrite,
 \item $X$ is a hereditarily unicoherent Peano continuum,
 \item $X$ is connected and any two points of the continuum are 
separated by a third point.
\end{enumerate}
\end{teo}

\begin{teo}
\label{teoWaz}
Every dendrite can be embedded in the plane. 
Moreover, the Wazewski's universal dendrite is a dendrite in $\R^2$ which contains a 
topological copy of any dendrite.
\end{teo}


\subsection{Decompositions}
\label{secDecomp}
The standard theory of foliations is based on a special kind of local partitions in plaques, i.e., \emph{foliated charts}. 
On a surface $S$, a $C^0$ foliated chart is a homeomorphism $\varphi\colon U\subset S\to (0,1)\times (0,1)$ where $U\subset S$ 
is an open subset.
If $\varphi=(\varphi_1,\varphi_2)$ then the plaques of $U$ are $\varphi_2^{-1}(y)$ for all $y\in (0,1)$. 
We can define a map $Q\colon U\to\continua(S)$ as $Q(p)=\varphi_2^{-1}(\varphi(p))$.
This map $Q$ is a continuous monotone partition of $U$.

For a pseudo-Anosov diffeomorphism, the stable and unstable sets form \emph{singular foliations}. 
At a singular point, an $n$-prong with $n>2$, it is not possible to define a local chart as above where the plaques are the local stable sets. 
The partition in local stable sets in a neighborhood of a singularity is monotone and upper semicontinuous. 
The (full) continuity is lost at the singularity. This is the idea we have in mind for the next definition of \emph{decomposition}, which is a standard concept in continuum theory.

Given a compact metric space $X$, we will consider a compact subset $Y\subset X$. 
We can think that $Y$ is the closure of the open set $U$ in the surface considered above. 
However, the theory is developed in such a way that $Y$ may not be connected and may have empty interior.

\begin{df}
\label{dfDec}
A \emph{decomposition} of $Y$ is a monotone upper semicontinuous partition $Q\colon Y\to\continua(Y)$. 
The sets $Q(x)$ are the \emph{plaques} of the decomposition.
\end{df}

\begin{exa}[Extremal examples]
\label{exaExtremal}
For a compact metric space $Y$ define $Q_{\min}(x)=\{x\}$ and $Q_{\max}(x)=\cc_x(Y)$. 
It holds that $Q_{\min}$ and $Q_{\max}$ are decompositions of $Y$, a proof can be found in \cite{Mo25}*{Theorem 24}.
For every decomposition $Q$ of $Y$ we have that $Q_{\min}(x)\subset Q(x)\subset Q_{\max}(x)$ for all $x\in Y$.
\end{exa}

For standard foliations, the restriction of a foliated chart to an arbitrary subset may not be a foliated chart. 
This is because the product structure may be lost. 
Next we show that the restriction of a decomposition is a decomposition.

\begin{prop}
\label{propRestDecEsDec}
 If $Q$ is a decomposition of the compact set $Y$
 and $Z\subset Y$ is compact then $Q|^m_Z$ is a decomposition.
\end{prop}

\begin{proof}
By definition (\ref{ecuRest}) $Q|^m_Z$ is monotone. 
Suppose that $x_n\to x$ and $Q|^m_Z(x_n)\to C$. 
Taking a subsequence we can also assume that $Q(x_n)\to D$ for some subcontinuum $D\subset X$.
Since $Q|^m_Z(x_n)\subset Q(x_n)$ and $Q$ is upper semicontinuous 
we have that $C\subset D\subset Q(x)$ 
and, as $C$ is a continuum contained in $Z\cap Q(x)$, 
we conclude that $C\subset Q|^m_Z(x)$.
This proves that $Q|^m_Z$ is upper semicontinuous.
\end{proof}

\begin{exa}
\label{exaDecProd}
 Let $P$ be a continuum and consider a compact metric space $Z$. 
 Define $Y=P\times Z$ and $Q\colon Y\to\continua(Y)$ as 
 $Q(p,z)=P\times \{z\}$. 
If we consider the product topology on $Y$ then
$Q$ is a continuous decomposition of $Y$.
\end{exa}

\begin{df}
\label{dfProSt}
Given two decompositions $Q_i\colon Y_i\to \continua(Y_i)$, $i=1,2$, 
we say that $Q_1$ and $Q_2$ are \emph{equivalent} if there is a 
homeomorphism $h\colon Y_1\to Y_2$ such that 
$h(Q_1(p))=Q_2(h(p))$ for all $p\in Y_1$. 
We say that a decomposition is a \emph{product structure} 
if it is equivalent to a decomposition as in Example \ref{exaDecProd}.
\end{df}

\begin{prop}
 A decomposition $Q_1$ of a continuum $Y$ is a product structure if and only if 
 there is a decomposition $Q_2$ of $Y$ 
 such that 
 \begin{equation}
  \label{ecuCardUno}
  \card(Q_1(x)\cap Q_2(y))=1
 \end{equation}
 for all $x,y\in Y$.
\end{prop}

\begin{proof}
If $Q_1$ is a product structure, 
we can assume that $Y=P\times Z$. 
Since $Y$ is a continuum we have that $P$ and $Z$ are continua. 
If $Q_1(p,z)=P\times \{z\}$ we can define $Q_2(p,z)=\{p\}\times Z$. 
Then 
$$Q_1(p_1,z_1)\cap Q_2(p_2,z_2)=\{(p_2,z_1)\}$$
for all $(p_1,z_1),(p_2,z_2)\in P\times Z$.

To show the converse note that by Proposition \ref{propCoc1}
we know that $Y/Q_1$ and $Y/Q_2$ are continua. 
Consider $h\colon Y\to Y/Q_1\times Y/Q_2$ 
defined by $$h(x)=(Q_1(x),Q_2(x)).$$ 
The condition (\ref{ecuCardUno}) implies that $h$ is bijective.
\azul{Since we consider the quotient topology, 
the projections onto $Y/Q_1$ and $Y/Q_2$ are continuous
and $h$ is continuous.}
Since $Y$ is compact we have that $h^{-1}$ is continuous.
Given that $h(Q_1(x))=\{Q_1(x)\}\times Y/Q_2$ the proof ends.
\end{proof}

\begin{df}
We say that a compact subset $K\subset Y$ is a \emph{representative set} of the decomposition $Q$
if $Q(x)\cap K\neq\emptyset$ for all $x\in Y$.
If $K$ is a representative set we say that $Q$ is a decomposition of $(Y,K)$.
\end{df}

We will consider decompositions of $(\clos{U},\partial U)$, with $U$ an open set, 
because for a cw-expansive homeomorphism 
we know that the diameter of the stable and the unstable sets are bounded away from zero. 
Therefore, if $U$ is small, each stable plaque in $U$ meets the boundary of $U$.

\begin{prop}
 If $Q$ is a decomposition of $(Y,K)$,
 $U\subset Y$ is open and $U\cap K=\emptyset$ then 
 $Q|^m_{\clos{U}}$ is a decomposition of $(\clos{U},\partial U)$.
\end{prop}

\begin{proof}
From Proposition \ref{propRestDecEsDec} we know that $Q'=Q|^m_{\clos{U}}$ is a decomposition.
Given $x\in U$ the fact that $Q'(x)\cap\partial U\neq\emptyset$ follows by \cite{HY}*{Theorem 2-16} 
or \cite{Engel}*{6.1.25}.
\end{proof}

In the following subsections we study some particular types of decompositions. 
Applications will be given in \S \ref{secDendritations} for the study of decompositions of two-dimensional discs.

\subsubsection{Dendritic decompositions}
\label{secDendDec}
Let $Q\colon Y\to \continua(Y)$ be a decomposition of the continuum $Y$. 

\begin{df}
\label{dfDendCoDend}
 We say that $Q$ is \emph{dendritic} if $Q(x)$ is a dendrite for all $x\in Y$. 
 We say that $Q$ is \emph{codendritic} if $Y/Q$ is a dendrite with the quotient topology.
\end{df}


Dendritic decompositions are a generalization of one-dimensional foliated charts. 
Codendritic decompositions generalizes codimension one foliated charts.
In \S \ref{CwFCwExp} we will study codendritic 
and dendritic decompositions of a two-dimensional disc. Some general properties can be derived in the generality of a continuum $Y$.

\begin{prop}
\label{propCoDendGen}
 If $Q\colon Y\to\continua(Y)$ is a codendritic decomposition 
 whose plaques have empty interior 
 then there is a residual subset $G\subset Y$ such that 
 $Y\setminus Q(x)$ has 1 or 2 components for all $x\in G$.
\end{prop}

\begin{proof}
From Theorem \ref{teoDenRamNum}
we have that the set of ramification points 
of the quotient dendrite $Y/Q$ is at most countable.
Let $Q_1,Q_2,\dots$ be the ramification points of $Y/Q$. 
We have that $Y\setminus Q_i$ is open in $Y$. 
It is dense in $Y$ because no $Q_i$ has interior points.
Define $G=Y\setminus \cup Q_i$. 
Since for every point $x$ of the residual set $G$ 
we know that $Q(x)$ is not a ramification point, 
we conclude that $Y\setminus Q(x)$ has 1 or 2 components.
\end{proof}

In the hypothesis of Proposition \ref{propCoDendGen} we cannot conclude that $Y\setminus Q(x)$ has 2 components for all $x$ in a residual subset of $Y$ 
(as could seem natural). See Example \ref{exaEspinaGenerica}.

With respect to the next proposition, if we think that $Q^1$ represents stable continua and $Q^2$ unstable continua, then the condition 
$Q^1(z)\cap Q^2(z)=\{z\}$ is related with cw1-expansivity (recall that 
for cw1-expansivity we require that local stable and unstable continua intersects 
in at most one point).

\begin{prop}
\label{propDendCoDend}
 Let $Q^1, Q^2$ be decompositions of the continuum $Y$ such that 
 $Q^1(z)\cap Q^2(z)=\{z\}$ for all $z\in Y$. 
 If $Q^1$ is codendritic then $Q^2$ is dendritic.
\end{prop}

\begin{proof}
Given $x\in Y$ consider $f\colon Q^2(x)\to Y/Q^1$ defined as $f(y)=Q^1(y)$. 
With the quotient topology on $Y/Q^1$ we have that $f$ is continuous because it is the restriction of the decomposition map 
$Q^1$. The condition $Q^1(z)\cap Q^2(z)=\{z\}$ for all $z\in Y$, gives us that $f$ is injective. 
Since $Q^2(x)$ is compact, $f$ is a homeomorphism onto its image. 
The image of $f$ is a subcontinuum of the quotient dendrite $Y/Q^1$. 
Then, the result follows by Theorem \ref{teoSubContDend} (every subcontinuum of a dendrite is a dendrite).
\end{proof}

\begin{rmk}
 The converse of Proposition \ref{propDendCoDend} is not true. Take two one-dimensional foliations on a three-dimensional ball. 
 Both foliations are dendritic but no one is codendritic.
\end{rmk}



\subsubsection{Smooth and unicoherent decompositions}
\label{secSmoothUniDec}

Our motivation for including the following kind of decompositions is to include 
the example \S \ref{secAnomalous} (the anomalous saddle) in the theory. 
The example presents local stable sets that are not locally connected but hereditarily unicoherent.

\begin{df}
A decomposition is 
 \emph{hereditarily unicoherent} 
 if each plaque is hereditarily unicoherent. 
\end{df}

 If $P$ is an hereditarily unicoherent continuum then
 given $x,y\in P$ there is a unique minimal continuum containing $x$ and $y$. 
 If $Q$ is a hereditarily unicoherent decomposition and $y\in Q(x)$ this minimal continuum will be denoted by $Q(x,y)$. 

 We remark that in continuum theory the term \emph{smooth} has a particular meaning that is not related with any class of differentiability 
 (or at least the author cannot see any connection). To avoid confusions we will call it $\continua$-\emph{smooth}.
 
\begin{df}
\label{dfCsmooth}
A decomposition $Q$ of a continuum $Y$ is $\continua$-\emph{smooth}
if it is hereditarily unicoherent 
and if $x_n\to x$, $y_n\in Q(x_n)$ and $y_n\to y$ then 
$Q(x_n,y_n)\to Q(x,y)$ in the Hausdorff metric.  
\end{df}

\begin{prop}
\label{propHuacImpDend}
 Every $\continua$-smooth decomposition is dendritic. 
\end{prop}

\begin{proof}
By Theorem \ref{teoCharDend}
it is sufficient to prove that 
each plaque is locally connected.
If $P$ is a non-locally connected plaque then there is 
$x\in P$ without arbitrarily small and connected neighborhoods. 
Let $\epsilon>0$ be such that $\cc_x (B_\epsilon(x)\cap P)$ is not 
a neighborhood of $x$ in $P$. 
Then we can take $x_n\in P$ such that 
$x_n\to x$ and $\cc_{x_n}(B_\epsilon(x)\cap P)$ is disjoint from 
$\cc_x(B_\epsilon(x)\cap P)$. 
Then, $\diam(Q(x,x_n))$ is bounded away from zero. 
But $x_n\to x$ and $Q(x,x)=x$. 
This contradicts that $Q$ is $\continua$-smooth and finishes the proof.
\end{proof}

\begin{prop}
\label{propRestSmooth}
 If $Q$ is a $\continua$-smooth decomposition of a continuum $Y$ and $Z\subset Y$ is compact then 
 $Q|^m_Z$ is $\continua$-smooth.
\end{prop}

\begin{proof}
Since each plaque of $Q$ is hereditarily unicoherent we have that each plaque 
of $Q|^m_Z$ is hereditarily unicoherent. 
Note that if $Q(x,y)\subset Z$ then $Q(x,y)=Q|^m_Z(x,y)$.
Suppose that $x_n\to x$, $y_n\in Q|^m_Z(x_n)$ and $y_n\to y$. 
We know that $Q(x_n,y_n)\to Q(x,y)$. 
Since $Q(x_n,y_n)=Q|^m_Z(x_n,y_n)$ and $Q(x,y)=Q|^m_Z(x,y)$ we have that $Q|^m_Z$ is $\continua$-smooth.
\end{proof}

\section{Foliations from continuum theory}
\label{secCwFoliations}
In this section we present a detailed exposition of
our \emph{foliations} from the viewpoint of continuum theory. 
Since we were not able to find this concept in the literature 
we develop the theory from the most basic concepts. 

In the \emph{standard theory of foliations}\footnote{For precision, fix the meaning of this expression to: the theory developed in \cite{CaCo}.}
there are at least two main viewpoints for the definition of foliation:
\begin{enumerate}
\item A partition of the manifold into leaves (satisfying certain conditions). 
This approach seems to be preferred in dynamical systems as it appears in \cite{BS}*{\S 5.13} and \cite{AH}*{\S 6.7}.
\item A maximal atlas of foliated charts. 
This definition is usual in 
texts on foliations as \cite{CaCo,CaNe}.
\end{enumerate}

\noindent{\bf Leaves-atlas equivalence}. (\cite{CaCo}*{Theorem 1.2.18})
Both definitions of foliation coincide.

Roughly speaking, the proof is as follows.
Given an atlas, we construct chains of plaques and then, the leaves. 
For the converse implication, 
one needs to recover the plaques from the leaves. 
A plaque is defined as the component of the intersection of a leaf with a local chart. 

In this section we introduce some levels of generalizations of foliations. 
The following table summarizes the definitions and their implications.
\begin{table}[ht]
\[
\begin{array}{ccc}
\text{dendritation}&\Rightarrow& \text{Peano foliation}\\
\Downarrow&&\Downarrow \\
\text{cw-foliation}&\Rightarrow&\continua\text{-foliation}\\
\end{array}
\]
\caption{Hierarchy of $\continua$-foliations. 
Dendritations, see \S \ref{secDendritations}, are Peano cw-foliations of surfaces.}
\label{tablaFol}
\end{table}

As we will see in \S \ref{exaBackGammon}, the $\continua$-foliations do not satisfy the 
leaves-atlas equivalence. 
This depends on the local connection of plaques. 
Then, we introduce the Peano foliations as $\continua$-foliations with locally connected plaques. 
In Theorem \ref{teoUnaPLC2}
we show that Peano foliations satisfy the leaves-atlas equivalence, see Equation (\ref{ecuFiel}).
For the study of cw-expansive homeomorphisms we introduce cw-foliations. 

\subsection{$\continua$-foliations}

Let $(X,\dist)$ be a compact metric space and denote by $\tau$ the topology of $X$.
Given two closed sets $Y_1,Y_2\subset X$ and we say that two decompositions 
$Q_i$ of $Y_i$, $i=1,2$, are \emph{compatible} if $$Q_1|^m_{Y_1\cap Y_2}=Q_2|^m_{Y_1\cap Y_2}.$$
Recall that $\clos{U}$ denotes the closure of $U$.
\begin{df}
 An \emph{atlas} is a collection of compatible decompositions 
 $$\A=\{Q_{\clos{U}}\colon \clos{U}\to\continua (\clos{U})\}_{U\in\U}$$
 where $\U$ is an open cover of $X$. 
 In this case we say that $\A$ is an atlas \emph{over} $\U$.
\end{df}

The atlases can be ordered by inclusion 
and by Zorn's Lemma we have that every atlas is contained in a maximal atlas.

\begin{df}
 A $\continua$-\emph{foliation} is a maximal atlas. 
\end{df}

A basis $\U$ of the topology of $X$ is \emph{complete} if $V\in\U$ and $U\subset V$ is open then $U\in\U$.

\begin{prop}
\label{propAtlasViaCover}
Every $\continua$-foliation is defined over a complete basis.
\end{prop}

\begin{proof}
Let $\W$ be a $\continua$-foliation over $\U$.
Let $U$ be an open set contained in $V\in\U$. 
Define $Q_{\clos{U}}=Q_{\clos{V}}|^m_{\clos{U}}$.
To show that $U\in\U$ we will show that $Q_{\clos{U}}$ is compatible with the decompositions of 
$\W$. If $W\in\U$ then 
\[
 Q_{\clos{U}}|^m_{\clos{U}\cap\clos{W}}=
 (Q_{\clos{V}}|^m_{\clos{U}})|^m_{\clos{U}\cap\clos{W}}=
 Q_{\clos{V}}|^m_{\clos{U}\cap\clos{W}}=
 Q_{\clos{W}}|^m_{\clos{U}\cap\clos{W}}
\]
where the second equality follows 
by Proposition \ref{propRestRest}
and the last one is the compatibility of the decompositions of $V$ and $W$.
\end{proof}

Notice that the decompositions $Q_{\clos{U}}$ are defined on the closures of the open sets of $\U$. 
We will consider $Q_U=Q_{\clos{U}}|^m_U$ for $U\in\U$.

\begin{df}
An \emph{open plaque} 
is a set of the form
$Q_U(x)$
for $x\in U\in\U$.
\end{df}

\begin{prop}
\label{propTopoFina}
If $\A$ is an atlas over the complete basis 
$\U$ then the set of open plaques 
$\{Q_U(x):x\in U\in\U\}$
is a basis of a topology $\tau_\A$ of $X$. 
If $\A_1\subset \A_2$ are two of such atlases then $\tau_{\A_1}=\tau_{\A_2}\subset \tau$.
\end{prop}

\begin{proof}
Since $\U$ is a cover of $X$, the open plaques cover $X$. 
Given $x\in Q_U(y)\cap Q_V(z)$, $U,V\in\U$, 
notice that $Q_U(y)\cap Q_V(z)=Q_U(x)\cap Q_V(x)$ and 
$Q_{U\cap V}(x)\subset Q_U(x)\cap Q_V(x)$. 
Since $\U$ is complete $U\cap V\in\U$ and $Q_{U\cap V}(x)$ is an open plaque.
This proves that the open plaques form a basis of a topology.
Since $U=\cup_{x\in U} Q_U(x)$ for all $U\in\U$, it holds that $\tau_\A\subset \tau$.

Suppose that $\A_1\subset \A_2$ are atlases over $\U_1\subset \U_2$, respectively.
It is clear that $\tau_{\A_1}\subset \tau_{\A_2}$.
Given $x\in U\in \U_2$ take $V\in \U_1$ such that $x\in V\subset U$. 
Then, $Q_V(x)\subset Q_U(x)$ and $\tau_{\A_2}\subset \tau_{\A_1}$.
\end{proof}

\begin{df}
The topology $\tau_\A$ will be called \emph{the plaque topology}. 
Since it does not depend on the atlas $\A$ (over a complete basis) representing the $\continua$-foliation $\W$ it will also 
be denoted as $\tau_\W$.
\end{df}


\subsubsection{Chains of plaques}

Let $\A$ be an atlas over a complete basis $\U$ of the compact metric space $(X,\dist)$. 
A \emph{chain of} $\A$-\emph{plaques} is a sequence $P_1,\dots,P_n$ 
of open plaques of $\A$ such that $P_i\cap P_{i+1}\neq\emptyset$ for all $i=1,\dots n-1$.
Given $x\in P_1$ and $y\in P_n$ we say that the chain goes from $x$ to $y$.

\begin{df}For a point $x\in X$ the $\A$-\emph{leaf} of 
$x$ is the set $\A(x)$ defined by: $y\in \A(x)$ if there is a chain of open plaques 
from $x$ to $y$.
\end{df}

\begin{prop}
\label{propLeafVsComp}
Every component of $(X,\tau_\A)$ is contained in a $\A$-leaf.
\end{prop}

\begin{proof}
It follows because the $\A$-leaves form a partition of $X$ in $\tau_\A$-open sets.
\end{proof}

We will define a metric on $X$ defining the topology $\tau_\A$. 
Define $\dist_\A\colon X\times X\to \R$ as 
\[
 \dist_\A(x,y)=\inf \sum_{i=1}^n\diam(P_i),
\]
where the infimum is taken over all the chains of open plaques $P_1,\dots,P_n$, $n\geq 1$, with $x\in P_1$ and $y\in P_n$. 
If $x$ and $y$ are in different leaves we set
$\dist_\A(x,y)=+\infty$.
It is standard to prove that $\dist_\A$ is a metric on $X$. 

\begin{lem}
\label{lemPlacasChicas} 
If $U\in \U$, $P_1,\dots,P_n$ is a chain of open plaques contained in $U$ and 
 $x\in P_1$ then each $P_i\subset Q_U(x)$.
\end{lem}

\begin{proof}
 Suppose that $P_i= Q_{U_i}(x_i)$ with $U_i\in\U$. 
 Define $V_i=U\cap U_i$.
 Since $P_i\subset U$, by Proposition \ref{propRestRest} we have that
 $P_i=Q_{V_i}(x_i)$ and  
 $P_1\subset Q_U(x)$. 
 As $P_1\cap P_2\neq\emptyset$ we conclude that $P_2\subset Q_U(x)$. 
 By induction, each $P_i\subset Q_U(x)$.
\end{proof}

\begin{prop}
\label{propTauAMetric}
 The topology defined by $\dist_\A$ is $\tau_\A$.
\end{prop}

\begin{proof}
Given $x\in U\in\U$, consider $\delta>0$ such that 
$B_\delta(x)\subset U$.
Take $y\in X$ such that $\dist_\A(x,y)<\delta$. 
We will show that $y\in Q_U(x)$. 
Since $\dist_\A(x,y)<\delta$ there is a chain of plaques $P_i$, $i=1,\dots,n$,
such that $x\in P_1$, $y\in P_n$ and $\sum_{i=1}^n\diam(P_i)<\delta$.
Then $\diam(\cup_{i=1}^nP_i)<\delta$ and $\cup_{i=1}^nP_i\subset U$. 
By Lemma \ref{lemPlacasChicas} we conclude that 
$\cup_{i=1}^nP_i\subset Q_U(x)$. 
Consequently $y\in Q_U(x)$.
To prove the converse consider a ball $B=\{y\in X:\dist_\A(x,y)<r\}$, $r>0$ and $x\in X$. 
Taking $U\in\U$ with $x\in U$ and $\diam(U)<r$ 
we see that $Q_U(x)\subset B$ and $B\in\tau_\A$.
\end{proof}

\begin{prop}
\label{propCompleto}
Every leaf and $X$ itself are complete (as metric spaces) with respect to $\dist_\A$. 
\end{prop}

\begin{proof}
 Let $x_n$ be a Cauchy sequence with respect to $\dist_A$. 
 The definition of the metric allows us to assume that the sequence is contained in a leaf. 
 Also, its limit (as we will prove that it exists) must be in this leaf. 
 Thus, the proof is reduced to show that $X$ is complete. 
 
 Since $\dist\leq\dist_\A$, 
 $x_n$ is a Cauchy sequence with respect to $\dist$. 
 Then, there is $x\in X$ such that $\dist(x_n,x)\to 0$. 
 Take $\epsilon>0$ such that $U=B_\epsilon(x)\in\U$.
 Take $l\geq 1$ such that $\dist_\A(x_n,x_l)<\epsilon/2$ and $\dist(x_n,x)<\epsilon/2$ for all $n\geq l$. 
 Consider a sequence of open plaques $P_n$ such that $x_l,x_n\in P_n$ and $\diam(P_n)<\epsilon/2$. 
 Take $U_n\in \U$ such that $U_n\subset U$ and $Q_{U_n}(x_l)=P_n$. 
 We know that $Q_{U_n}(x_l)\subset Q_U(x_l)$. 
 Since $x_n\in Q_{U_n}(x_l)$ and $x_n\to x$ (in $\dist$) we conclude that $x\in Q_U(x_l)$. 
 Therefore, $x_n\in Q_U(x)$ for all $n\geq l$. 
 This proves that $\dist_\A(x_n,x)\to 0$.
\end{proof}

\subsection{Peano foliations}


Let $X$ be a compact metric space with topology $\tau$.

\begin{df}
An atlas $\A$ is a \emph{Peano atlas} if its open plaques are locally connected in $\tau$. 
A \emph{Peano foliation} is a maximal Peano atlas.
\end{df}

By \emph{maximal Peano atlas} we mean maximal among all Peano atlases, see Example \ref{exaSolenoide} 
for a maximal Peano atlas that is not maximal among all the atlases. 
Every Peano atlas is contained in a Peano foliation and 
every Peano foliations is defined over a complete basis.
Given a topology $\tau_*$ on $X$ and $Y\subset X$ denote by 
$\tau_*|_Y$ the relative topology on $Y$ induced by $\tau_*$.

\begin{teo}
\label{teoUnaPLC2}
If $\W$ is a Peano foliation over $\U$ then: 
\begin{enumerate}
\item for every open plaque $P$ it holds that $\tau|_P=\tau_\W|_P$,
\item every open plaque is connected in $\tau_\W$,
  \item $(X,\tau_\W)$ is locally connected, 
  \item every leaf is arc-connected in both topologies $\tau_\W$ and $\tau$,
  \item every leaf is a component of $(X,\tau_\W)$,
  \item for every leaf $L$, $U\in\U$ and $x\in L\cap U$ it holds that 
  \begin{equation}
   \label{ecuFiel}
   \cc_x(L\cap U)=Q_U(x). 
  \end{equation}
\end{enumerate}
\end{teo}

\begin{proof}
Let $P$ be an open plaque. By Proposition \ref{propTopoFina} we know that $\tau\subset\tau_\W$ 
and then $\tau|_P\subset\tau_\W|_P$.
To show that $\tau_\W|_P\subset\tau|_P$ it is sufficient to prove that if $P'$ is another open plaque then 
$P\cap P'\in \tau|_P$.
Take $x\in P\cap P'$ and suppose that $P=Q_U(x)$ and $P'=Q_V(x)$ with $U,V\in\U$. 
Since $\U$ is a complete basis of $\tau$ there is $W\in\U$ such that 
$x\in W\subset U\cap V$ and $P\cap W$ is connected ($P$ is locally connected at $x$). 
Then $P\cap W=\cc_x(P\cap W)=\cc_x(Q_U(x)\cap W)=Q_W(x)$.
Since $W\subset U\cap V$ we can apply Proposition \ref{propRestRest} to conclude that $Q_W(x)\subset P\cap P'$. 
Given that $Q_W(x)=P\cap W\in\tau|_P$ we conclude that $\tau_\W|_P\subset\tau|_P$. Then $\tau|_P=\tau_\W|_P$. 

By definition, every open plaque is connected in $\tau$ and then 
we have that $P$ is connected in $\tau_\W$. 
Given that open plaques form a basis for $\tau_{\W}$, we conclude that $(X,\tau_\W)$ is locally connected.

To prove that the leaves are arc-connected we first show that 
they are $\tau_\W$-connected.
Take $x,y$ in a leaf $L$. 
We know that there is a chain of plaques $P_1,\dots,P_n$ from $x$ to $y$. 
Since each plaque is 
$\tau_\W$-connected, 
$\cup_{i=1,\dots,n}P_i$ is $\tau_\W$-connected. 
Fixing $x$ and varying $y\in L$ we conclude that $L$ is $\tau_\W$-connected.
Then, we know that each leaf $L$ is connected and locally connected in $\tau_\W$. 
In addition, by Proposition \ref{propCompleto}, we know that $(L,\dist_{\W})$ is a complete metric space. 
Then, we can apply 
Theorem \ref{teoLocConArc}
to conclude the arc-connection of $L$.
This implies that every leaf is contained in a component of $(X,\tau_\W)$. 
The converse inclusion follows by Proposition \ref{propLeafVsComp}.

\azul{To prove (\ref{ecuFiel}) take a leaf $L$, $U\in\U$, $x\in L\cap U$
and define $C=\cc_x(L\cap U)$.
By the definition of leaf, it is clear that $Q_U(x)\subset C$.
We will show the converse inclusion. 
For this purpose, since $C$ is a union of open plaques, we can write 
$$C=\cup_{j\in J}P_j$$ 
with $P_j=Q_{U_j}(x_j)$ and some index set $J$.}
By Proposition \ref{propRestRest} we can suppose that $U_j\subset U$ for all $j\in J$.
Define $$J_1=\{j\in J: P_j\cap Q_U(x)\neq\emptyset\}$$ and $J_2=J\setminus J_1$.
Since $Q_{U_j}=Q_U|^m_{U_j}$, if $j\in J_1$ then $P_j\subset Q_U(x)\subset C$. 
Define $V_1=\cup_{j\in J_1} P_j$ and $V_2=\cup_{j\in J_2} P_j$. 
Define $V'_1=V_1\cap Q_U(x)$ and $V'_2=V_2\cap Q_U(x)$.
Then $Q_U(x)\subset V'_1\cup V'_2$. 
Since $Q_U(x)$ is locally connected with respect to $\tau$, 
we conclude that 
$V'_1,V'_2\in \tau|_{Q_U(x)}$. 
Since $V'_1\cap V'_2=\emptyset$, $V'_1\neq\emptyset$ and $Q_U(x)$ is connected in $\tau$ we conclude 
that $V'_2=\emptyset$.
Then $Q_U(x)\subset \cup_{j\in J_1}Q_{U_j}(x_j)$. 
Since $Q_{U_j}(x_j)\subset Q_U(x)$ for all $j\in J_1$
we conclude that $Q_U(x)= \cup_{j\in J_1}Q_{U_j}(x_j)$. That is, $Q_U(x)=C=\cc_x(L\cap U)$.
\end{proof}

We will give two examples of $\continua$-foliations with particular properties 
to show the necessity of the hypothesis of Theorem \ref{teoUnaPLC2}.  
First we show that a leaf may not be connected in the plaque topology (for a non-Peano foliation).

\begin{exa}
\label{exaExSinTopo}
Consider the plane continuum 
$X=X_1\cup X_2$ where $X_1=\{0\}\times[-1,1]$ and $X_2=\{(x,\sin(1/x)):x\in (0,1]\}$. 
On $X$ consider the complete basis $\U=\tau$ (the relative open subsets of the plane with its usual topology).
Let $\W$ be the $\continua$-foliation defined as $Q_{\clos{U}}(x)=\cc_x({\clos{U}})$ for every $U\in\tau$.
We have only one leaf $L=X$. 
We see that this leaf is not connected in $\tau_\W$ since it has two components $X_1$ and $X_2$.
Note that $X$ itself is an open plaque that is not locally connected. Then $\W$ is not a Peano foliation. 
It is an exercise for the interested reader to check which conclusions of Theorem \ref{teoUnaPLC2} are satisfied in this example.
In Example \ref{exaDecSen} we give a decomposition of a square with a leaf like the set $X$.
\end{exa}

In the following variation of the previous example we show that $(X,\tau_A)$ may not be locally connected.

\begin{exa}
\label{exaExSinTopo2}
For $n$ a non-negative integer define
\[
X_n=\{(x,(3+\sin(1/x))/4^n)\in\R^2: x\in (0,1]\}  
\]
and $X=(\{0\}\times [0,1])\cup ([0,1]\times\{0\})\cup\bigcup_{n\geq 0} X_n$. 
See Figure~\ref{figSenos}. 
As in the previous example consider the $\continua$-foliation 
defined by $Q_{\clos{U}}(x)=\cc_x({\clos{U}})$ for every $U\in\tau$.
We have that $(0,0)$ has not connected neighborhoods in the plaque topology.
\begin{figure}[ht]
 \center
 \includegraphics[scale=1]{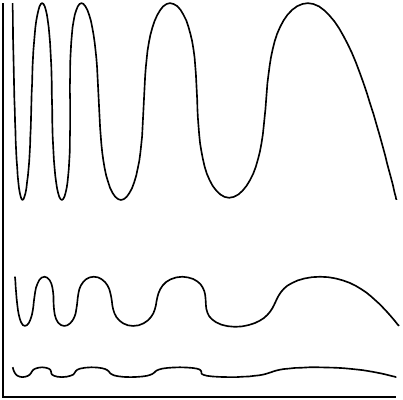}
 \caption{In this continuum the topology $\tau_\A$ is not locally connected around the origin.}
 \label{figSenos}
 \end{figure} 
\end{exa}

We think that the following result is interesting because it allows us to characterize Peano foliations in terms 
of the plaque topology.

\begin{cor}
For an atlas $\A$ over a complete basis
the following statements are equivalent:
\begin{enumerate}
 \item $\A$ is a Peano atlas,
 \item $\tau|_P=\tau_\A|_P$ for every open plaque $P$.
\end{enumerate}
\end{cor}

\begin{proof}
The direct part follows by Theorem \ref{teoUnaPLC2}.
To prove the converse, let $P$ be an open plaque. 
We have to prove that $P$ is locally connected in $\tau$. 
Given $x\in P$ we can take a small neighborhood $U\in\U$ of $x$. 
By definition, the plaque $Q_U(x)$ is connected in $\tau$. 
Also, $Q_U(x)$ is a neighborhood of $x$ in the plaque topology. 
Since, by hypothesis, $\tau|_P=\tau_\A|_P$, 
we conclude that $Q_U(x)$ is a (connected) neighborhood of $x$ in $\tau|P$. 
Then, $P$ is locally connected with respect to $\tau$.
\end{proof}

We say that a leaf is \emph{plaque-compact} if it is compact in the plaque topology.

\begin{thm}
\label{thmPCompIffPeano}
Let $\W$ be a Peano foliation.
An $\W$-leaf is plaque-compact if and only if it is a Peano continuum in the relative topology of $\tau$.
\end{thm}

\begin{proof}
Assume that $L$ is a plaque-compact leaf. 
Since the plaque topology is finer than the relative topology, Proposition \ref{propTopoFina}, we have that 
$L$ is compact in the relative topology.
Given $\epsilon>0$ we define the complete basis 
\begin{equation}
\label{ecuUdeltaa}
 \U_\epsilon=\{U\in\tau:\diam(U)<\epsilon\}.
\end{equation}
We assume that $\epsilon$ is so small that $\U_\epsilon\subset\U$, where $\U$ is the complete basis of $\W$.
Consider the following cover of $L$ $$\U_L=\{Q_U(x):x\in L\cap U, U\in \U_\epsilon\}.$$
Since $L$ is plaque-compact, there is a finite subcover. 
The plaques of this subcover have diameter smaller than $\epsilon$. 
Applying Sierpi\'nski's Theorem \ref{teoSierpinski} 
we conclude the local connection of $L$ in the relative topology.

Converse. 
Assume that the leaf $L$ is a Peano continuum in $\tau$.
Take a cover $\U_L=\{P_i\}_{i\in I}$ of $L$ by open plaques. 
Given $x\in L$ take $U_x\in\U$ such that $x\in P_i=Q_{U_x}(x)$. 
Since $L$ is locally connected and $\U$ is a complete basis 
there is $V_x\in\U$ such that $x\in V_x\subset U_x$
and $V_x\cap L$ is connected. 
Then, by Theorem \ref{teoUnaPLC2},
$Q_{V_x}(x)=V_x\cap L$.
Since $\{V_x\cap L:x\in L\}$ is a cover of $L$ by relative open sets of $L$ and 
$L$ is compact we can take $x_1,\dots,x_n$ such that $V_{x_1},\dots,V_{x_n}$ cover $L$. 
Since $V_x\subset U_x$, we have that 
$\{Q_{U_{x_1}}(x_1),\dots,Q_{U_{x_n}}(x_n)\}$ is a finite subcover of $\U_L$. 
This proves that $L$ is compact in the plaque topology.
\end{proof}

\begin{rmk}
In Example \ref{exaExSinTopo} we see that a compact leaf (in the relative topology) 
may not be plaque-compact (it fails to be locally connected). 
\end{rmk}

\begin{rmk}
A locally connected leaf may not be compact. 
Take a plane flow (with $X$ a compact invariant annulus) with limit cycles. 
An orbit converging to a cycle is locally connected but it is not compact. 
\end{rmk}

Note that the examples given above (Examples \ref{exaExSinTopo} and \ref{exaExSinTopo2}) are not Peano foliations but they satisfy (\ref{ecuFiel}) in Theorem \ref{teoUnaPLC2}. 
The next subsection is devoted to describe an example of a $\continua$-foliation not satisfying (\ref{ecuFiel}).

\subsubsection{Devil's backgammon}
\label{exaBackGammon}
In Theorem \ref{teoUnaPLC2} we proved that for Peano foliations, if 
$L$ is a leaf then for every $U\in\U$ and $x\in L\cap U$ it holds that 
$\cc_x(L\cap U)$ is the plaque of $x$ in $U$. 
That is, knowing the leaves and the basis $\U$ we can recover the plaques. 
This is the \emph{leaves-atlas equivalence} mentioned in the beginning of this section. 
The next example shows that a maximal atlas (with non-locally connected plaques) can contain strictly more information than 
the leaves and the basis, i.e., does not satisfy (\ref{ecuFiel}).

We start constructing a continuum $X\subset [0,1]\times [0,1]$. 
Let $K\subset [0,1]$ be the ternary Cantor set. 
A point $x\in K$ can be expressed as $x=\sum_i2/3^{n_i}$ for some increasing sequence $n_i$ of integers. 
Consider the function $g\colon K\to[0,1]$ 
such that if $x=\sum_i2/3^{n_i}\in K$ then $g(x)=\sum_i1/2^{n_i}$.
For each $y\in[0,1]$ there are 1 or 2 preimages by $g$. 
For $x\in K$ let $I_x$ be the closed segment line from $(x,0)$ to $(g(x),1)$. 
Define 
\[
 X=([0,1]\times\{0\})\cup\bigcup_{x\in K} I_x.
\]
We have that $X$ is an arc-connected continuum. Note that it is not locally connected.
It can also be obtained by deleting open triangles from the square, as shown in Figure~\ref{figDelTri}.
 \begin{figure}[ht]
 \center
 \includegraphics[scale=1]{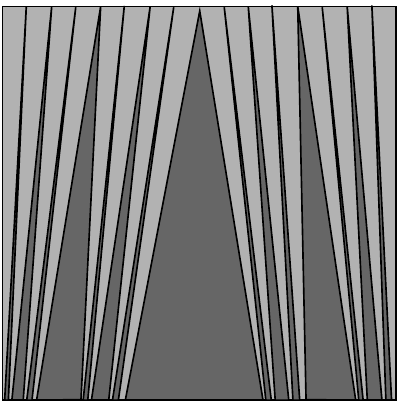}
 \caption{The Devil's Backgammon continuum is obtained from the square by deleting the dark open triangles. 
 Each triangle is determined by $I_x$ and $I_y$ if $g(x)=g(y)$ and $x\neq y$.}
 \label{figDelTri}
 \end{figure}

 Consider the open sets 
 \[ 
  \begin{array}{l}
U=\{(x,y)\in X:y<2/3\},\\
V=\{(x,y)\in X:y>1/3\}.   
  \end{array}
 \] 
 Define $Q_U(x,y)=U$ for all $(x,y)\in U$. 
 Given $p\in V$ suppose that $p\in I_a$. 
 Define 
 \[
  J_p=\left\{
  \begin{array}{ll}
   I_a & \text{ if } g^{-1}(g(a))=\{a\},\\
   I_a\cup I_b & \text{ if } g^{-1}(g(a))=\{a,b\}
  \end{array}
  \right.
 \]
and $Q_V(p)=\cc_p (J_p\cap V)$. 
The decompositions $Q_U,Q_V$ define a $\continua$-foliation $\W$.
It is clear that it is not a Peano foliation since it has non-locally connected open plaques.
Note that there is only one leaf, namely $L=X$. 

In this example we see that
for $p\in V$ we have that $\cc_p(L\cap V)=L\cap V$ and it is not $Q_V(p)$.
As we said, this means that the plaques cannot be recovered by the leaves and the basis.

%
%

\subsection{Cw-foliations}
\label{secCwFol}
It must be recalled that we have in mind the study of the distribution of stable and unstable 
continua of cw-expansive homeomorphisms on surfaces and Peano continua. 
This guides us in selecting the properties that we require in the following definition. 

Recall that a subset of $X$ is \emph{meagre} if it can be expressed as 
the union of countably many nowhere dense subsets of $X$. 
A set is \emph{nowhere dense} if its closure has empty interior.


\begin{df}
An atlas $\A$ is a \emph{cw-atlas} if its $\A$-leaves:
\begin{enumerate}
 \item are a countable union of open plaques,
 \item are meagre sets and 
 \item have diameter bounded away from zero.
\end{enumerate}
A \emph{cw-foliation} 
is a $\continua$-foliation with a cw-atlas.
\end{df}

More observations about this definition must be given.

\begin{rmk}
Note that requiring that a leaf is a countable union of plaques does not imply that the leaf has a countable basis. 
The plaques in this countable union may not be small.
\end{rmk}

\begin{rmk}
A foliation of a compact manifold $M$ in the standard sense 
is a cw-foliation if and only if the dimension of the leaves is positive and 
less than $\dim(M)$. On a surface it means \emph{one-dimensional}.
\end{rmk}

In the standard theory of foliations the leaves are immersed submanifolds 
and have a countable basis for its leaf topology. Also, the intersection of a leaf 
with a local chart is a countable number of plaques. 
For Peano cw-foliations these properties can be recovered.
A \emph{Peano cw-foliation} is a Peano foliation with a cw-atlas.

\begin{teo}
\label{teoBaseNumLeaf}
If $\W$ is a Peano cw-foliation
then: 
\begin{enumerate}
 \item $\tau_\W|_L$ has a countable basis for every leaf $L$,
 \item for every leaf $L$ and every $U\in\tau$ the set $L\cap U$ has a countable number 
of components with respect to both topologies $\tau$ and $\tau_\W$.
\end{enumerate}
\end{teo}

\begin{proof}
Since $X$ is a compact metric space it has a countable basis. 
By Theorem \ref{teoUnaPLC2} we know that $\tau|_P=\tau_\W|_P$, for every open plaque $P$.
We conclude that each plaque has a countable basis.
Consequently, as every leaf is a countable union of plaques, we have that each leaf has a countable basis. 

Consider a leaf $L$ and $U\in\tau$.
We know that $L\in\tau_\W$. 
From Proposition \ref{propTopoFina} we have that $U\in\tau_\W$. 
Then $U\cap L\in\tau_\W$.
We know that $\tau_\W|_L$ has a countable basis.
Therefore, $\tau_\W|_{U\cap L}$ has a countable basis. 
By Theorem \ref{teoUnaPLC2}, knowing that $U\cap L\in\tau_\W$, 
we conclude that $U\cap L$ is locally connected with respect to $\tau_\W$.
Then\footnote{Every locally connected topological space with a countable basis has 
 at most a countable number of components. 
 This can be proved as follows. 
 Since the space is locally connected, its components are open 
 and given that the space has a countable basis each component contains an open set of the countable basis. 
 As the components are disjoint, there is at most a countable number of components.} 
 $U\cap L$ has a countable number of components 
with respect to $\tau_\W$.
Since $\tau\subset\tau_\W$ the same holds for $\tau$.
\end{proof}

For future reference (Proposition \ref{propCwFolDec}) we give one more result.
Recall from (\ref{ecuUdeltaa}) that $\U_\delta$ is the complete basis containing 
all the open sets of diameter smaller than $\delta$.

\begin{prop}
\label{propCwFolDec1}
 If $\A$ is a cw-atlas over the complete basis $\U$ of $X$ then 
 there is $\delta>0$ such that for all $x\in U\in \U_\delta$ 
 it holds that: 
 \begin{enumerate}
  \item $Q_{\clos{U}}(x)\cap\partial U\neq\emptyset$ and 
  \item $x$ is not a $\tau$-interior point of $Q_U(x)$.
 \end{enumerate}
\end{prop}

\begin{proof}
From the definition of cw-atlas we know that 
there is $\delta>0$ such that every leaf has diameter greater than $\delta$. 
Take $x\in U\in \U_\delta$. 
Since the leaf of $x$ has diameter greater than $\delta$ there is $y\notin \clos{U}$ in the leaf of $x$. 
Consider a chain of plaques $P_1,\dots,P_n$ from $x$ to $y$. 
Let $P_m$ be the first of these plaques that is not contained in $U$. 
By Lemma \ref{lemPlacasChicas} the plaques $P_1,\dots,P_{m-1}$ are contained in $Q_U(x)$.
Take $z\in P_{m-1}\cap P_m$. 
Since $P_m$ is connected we have that $\cc_z(P_m\cap \clos U)$ cuts $\partial U$.
Given that $\cc_z(P_m\cap \clos U)\subset Q_{\clos{U}}(x)$ we have that $Q_{\clos{U}}(x)\cap\partial U\neq\emptyset$.

Since each leaf is a meagre set, no plaque has interior points. 
\end{proof}
The next example shows that the converse of Proposition \ref{propCwFolDec1} is not true.

\begin{exa}
\label{exaCantor3D}
 Let $K\subset [0,1]$ be a Cantor set. 
 Define $X=K\times[0,1]\times[0,1]\cup [0,1]\times[0,1]\times\{0\}$. 
 Define the projection $\pi\colon X\subset\R^3\to\R$ given by $\pi(x,y,z)=y$ and the 
 decomposition $\W$ as $\W(x,y,z)=\pi^{-1}(y)$. 
 For $\delta\in(0,1/2)$ and the Euclidean metric on $X$, define $Q_U(p)=\W|^m_U(p)$. 
 It defines an atlas over $\U_\delta$. 
 No leaf has interior points and every leaf has diameter greater than $\delta$. 
 But this atlas does not define a cw-foliation because the leaves are not a countable union plaques.
\end{exa}

\subsection{Continuous atlas}
\label{secConAtl}
We say that an atlas $\A$ over a complete basis $\U$ is \emph{continuous} if the set 
\[
 \{U\in\U: Q_{\clos{U}}\text{ is continuous}\}
\]
is a basis of the topology of $X$. The continuity of $Q_{\clos{U}}$ means the continuity of the 
map $Q_{\clos{U}}\colon \clos{U}\to\continua(Q_{\clos{U}})$ with respect to the Hausdorff metric.

\begin{rmk}
The stable foliation of a pseudo-Anosov diffeomorphism 
with 1-prongs (as in \S \ref{secPAS2})
defines a continuous atlas. 
In Theorem \ref{teoClassRegCwFol} we will show that continuous Peano cw-foliations of compact surfaces (i.e., continuous dendritations),
are in fact foliations (in the standard sense) with a finite number of 1-prongs. 
\end{rmk}

\begin{exa}[Hiraide's generalized foliations]
Let $(X,\dist)$ be a Peano continuum. 
A partition $\W$ of $X$ is a \emph{generalized foliation} \cites{Hi89, AH} if for all 
$x\in X$ there are non-trivial \azul{arc}-connected subsets $D,K\subset X$ 
with $D\cap K=\{x\}$, a connected open neighborhood $N$ of $x$ in $X$ 
and a homeomorphism $\phi_x\colon D\times K\to N$
such that: 
\begin{enumerate}
 \item $\phi_x(y,x)=y$ for all $y\in D$, $\phi_x(x,z)=\azul{z}$ for all $z\in K$,
 \item for every leaf $L$ there is an at most countable set $B\subset K$ such that $N\cap L=\phi_x(D\times B)$. 
\end{enumerate}
Generalized foliations are a main concept for the study of expansive homeomorphisms with canonical coordinates.
Notice that the stable foliation of the pseudo-Anosov diffeomorphism of \S \ref{secPAS2}
is not a generalized foliation.
We can show that every generalized foliation defines a continuous atlas as follows.
Given $x\in X$ consider a local coordinate $\phi\colon D\times K\to N$ around $x$. 
Take $D'$ and $K'$ compact neighborhoods of $x$ in $D$ and $K$ respectively.
Let $U$ be an open set containing $x$ such that $\clos{U}= \phi(D'\times K')$. 
Since on $D\times K$ we consider the product topology, 
we have that $\W|^m_{\clos{U}}$ is continuous. 
\end{exa}

\section{Expansivity and cw-foliations}
\label{secStUnsCwFol}

The results that we have obtained will be now applied in the study of the 
stable and unstable cw-foliations determined by a cw-expansive homeomorphism. 
Also, we will give sufficient conditions for cw1-expansivity to imply expansivity.

\subsection{Stable and unstable cw-foliations}
\label{secSUCwfols}

Let $f\colon X\to X$ be a homeomorphism of a compact metric space. 
Recall the definitions of stable and unstable continua given in (\ref{defCsCu}).

\begin{df}
A decomposition of a closed subset of $X$ is \emph{stable} if every plaque is stable.
An atlas $\A$ is \emph{stable} if each decomposition is stable. 
A $\continua^s$-\emph{foliation} is a maximal stable atlas.
The unstable versions of these definitions are analogous. 
\end{df}

As before, \emph{maximal stable atlas} means maximal among all the stable atlases, see Example \ref{exaSolenoide}.
For a closed subset $Y\subset X$ define 
$Q^s_Y\colon Y\to \parts(Y)$ 
by $y\in Q^s(x)$ if there is a stable continuum 
$C\subset Y$ such that $x,y\in C$. 
For every homeomorphism $f$ and closed subset $Y\subset X$ 
we have that $Q^s_Y$ is a monotone partition of $Y$. 
Similarly we define $Q^u_Y$ considering unstable continua.
Define:
\begin{equation}
\label{ecuUdelta}
\begin{array}{l}
 \U^s=\{U\in\tau: Q^s_{\clos{U}}\text{ is a stable decomposition}\},\\
 \A^s=\{Q^s_{\clos{U}}:U\in\U^s\}.\\ 
\end{array}
\end{equation}
Note that $\U^s$ and $\A^s$ depend on $f$. 
The sets $\U^u$ and $\A^u$ are defined similarly considering the inverse of $f$.

We will show that if 
$f$ is cw-expansive then
$\U^\sigma$ is a cover of $X$ 
and that in this case 
$\A^\sigma$ is a stable atlas of a cw-foliation.
By definition it is clear that $\A^\sigma$ is a maximal stable (unstable) atlas. 
We define the \emph{stable} and \emph{unstable cw-foliations} of $f$ as $\W^s$, $\W^u$ respectively.

An atlas $\A$ over $\U$ is \emph{invariant} if:
\begin{itemize}
 \item $f^n(U)\in \U$ and 
\item $Q_{f^n(\clos{U})}(f^n(x))=f^n(Q_{\clos{U}}(x))$ for all $x\in \clos{U}$
\end{itemize}
for all $n\in\Z$ and all $U\in\U$. 
The purpose of this section is to prove the following result.

\begin{thm}
\label{thmCovCwExpFol}
 If $f$ is a cw-expansive homeomorphism of a Peano continuum then 
 $\W^s,\W^u$ are invariant cw-foliations without plaque-compact leaves.
\end{thm}

The proof is developed in some lemmas that do not assume that $f$ is cw-expansive.

\begin{lem}
\label{lemInvStDec}
If $Q^\sigma_Y$ is a stable (unstable) decomposition, $\sigma=s,u$, and $Z\subset Y$ is closed then $$Q^\sigma_Z=Q^\sigma_Y|^m_Z.$$
 \end{lem}

\begin{proof}
%
Take $x\in Z$ and $y\in Q^s_Z(x)$. 
Then there is a stable continuum $C$ such that $x,y\in C\subset Z$. 
Since $Z\subset Y$, we have that $C\subset Q^s_Y(x)$. 
Since $x\in C$ and $C$ is connected we have that $C\subset Q^s_Y|^m_Z(x)$. 
Consequently $y\in Q^s_Y|^m_Z(x)$ and $Q^s_Z(x)\subset Q^s_Y|^m_Z(x)$.

Let us show the converse inclusion. 
Take $x\in Z$ and $y\in Q^s_Y|^m_Z(x)$. 
Recall that $Q^s_Y|^m_Z(x)=\cc_x(Q^s_Y(x)\cap Z)$.
Since $Q^s_Y$ is a stable decomposition we have that 
$\cc_x(Q^s_Y(x)\cap Z)$ is a stable continuum.
Given that $\cc_x(Q^s_Y(x)\cap Z)\subset Z$ we have that 
$\cc_x(Q^s_Y(x)\cap Z)\subset Q^s_Z(x)$ and $y\in Q^s_Z(x)$.
Then $Q^s_Y|^m_Z(x)\subset Q^s_Z(x)$.
\end{proof}

\begin{lem}
\label{lemStCover}
If $\U^\sigma$ covers $X$, $\sigma=s,u$, then: 
\begin{enumerate}
 \item $\A^\sigma$ is a stable atlas over the complete basis $\U^\sigma$, 
 \item every stable (unstable) continuum is contained in an open $\A^\sigma$-plaque,
 \item $Q^\sigma_X(x)$ is the $\A^\sigma$-leaf of $x$ for all $x\in X$, 
 \item every $\A^\sigma$-leaf is a countable union of open $\A^\sigma$-plaques.
\end{enumerate}
\end{lem}

\begin{proof}
Assume that $\U^s$ covers $X$.
By Lemma \ref{lemInvStDec} we know that $\U^s$ is a complete basis. 
  Consequently, $\A^s$ is a stable atlas.
Let $C\subset X$ be a stable continuum.
Since $\U^s$ covers $X$
we can take $n\geq 0$ and $U\in \U$ such that $f^n(C)\subset U\in \U^s$. 
Then $C\subset Q^s_{f^{-n}(U)}(x)$ for all $x\in C$. 
Since $\U^s$ is $f$-invariant, $Q^s_{f^{-n}(U)}(x)$ is a stable open plaque containing $C$. 
It is clear that the leaf of $x$ is contained in $Q^s_X(x)$. The converse inclusion follows by the previous item.

Let us show that each leaf is a countable union of open plaques. 
Take $\delta>0$ such that every open subset of diameter smaller than $\delta$ belongs to $\U^s$. 
Then $B_{\delta/2}(x)\in\U^s$ for all $p\in X$ 
and we can consider the open plaque $Q^s_{B_{\delta/2}(x)}(x)$.
For $x\in X$ and $n\geq 0$ define 
\[
P_n=f^{-n}(Q^s_{B_{\delta/2}(f^n(x))}(f^n(x))). 
\]
We have that $P_n$ is a stable open plaque. 
We have proved that the leaf of $x$ is $Q^s_X(x)$. 
Then $P_n\subset Q^s_X(x)$ for all $n\geq 0$. 
Given $y\in Q^s_X(x)$ let us show that there is $k\geq 0$ such that $y\in P_k$. 
If $y\in Q^s_X(x)$ then there is a stable continuum $C$ containing $x$ and $y$. 
Since $C$ is stable, there is $k\geq 0$ such that $\diam(f^k(C))<\delta/2$. 
Given that $f^k(x)\in f^k(C)$ and that $f^k(C)$ is a stable continuum 
we have that $f^k(C)\subset Q^s_{B_{\delta/2}(f^k(x))}(f^k(x))$.
Then $y\in P_k$. This proves that $Q^s_X(x)=\cup_{n\geq 0}P_n$, a countable union of plaques.
\end{proof}

\begin{lem}
 \label{lemCwExpBasCovPre}
 If $f$ is a homeomorphism of a compact metric space $X$ and there are $\epsilon,\delta>0$ such that 
 $\continua^\sigma\cap\continua_\delta\subset\continua^\sigma_\epsilon$ and 
 $\continua^\sigma_{2\epsilon}\subset\continua^\sigma$ 
 for $\sigma=s,u$ then $\U^\sigma$ covers $X$. 
\end{lem}

\begin{proof}
We give the proof for $\sigma=s$.
Let $Y\subset X$ be a compact subset with $\diam(Y)\leq\delta$.
We will show that $Q^s_Y$ is a stable decomposition of $Y$.
Take $x\in Y$. 
By definition we have that $Q^s_Y(x)$ is connected for all $x\in Y$. 
We have that $Q^s_Y(x)$ is the union of all the stable 
continua $C$ contained in $Y$ and containing $x$. 
Since $\continua^s\cap \continua_\delta\subset\continua^s_\epsilon$ and 
$\diam(Y)\leq\delta$ we have that each $C\in \continua^s_\epsilon$.
This implies that 
\begin{equation}
 \label{ecuPlacaEstable}
 \diam(f^k(Q^s_Y(x))\leq2\epsilon\text{ for all }k\geq 0.
\end{equation}
Since $\continua^s_{2\epsilon}\subset\continua^s$ 
we have that $Q^s_Y(x)$ is a stable continuum. 
Take $x_n\to x$ in $Y$ with $Q^s_Y(x_n)\to C$.
From (\ref{ecuPlacaEstable}) we know that $\diam(f^k(C))\leq2\epsilon$ for all $k\geq 0$. 
Therefore, we conclude that $C$ is stable. 
Then, $C\subset Q^s_Y(x)$ and the proof ends.
\end{proof}

%

\begin{proof}[Proof of Theorem \ref{thmCovCwExpFol}]
By Proposition \ref{propExConNT} 
we know that $f$ satisfies the hypothesis of Lemma \ref{lemCwExpBasCovPre}. 
Then $\U^\sigma$ covers $X$. 
Applying Lemma \ref{lemStCover}
we conclude that $\U^\sigma$ is a complete basis, that $\A^\sigma$ is an atlas 
and that every $\A^\sigma$-leaf is a countable union of open $\A^\sigma$-plaques. 
By Theorem \ref{thmExConNT} we know that $\continua^\sigma_\epsilon(x)\setminus\continua_\delta(x)\neq\emptyset$.
Then, every leaf has diameter greater than $\delta$.
 The cw-expansivity and the condition $\continua^\sigma_\epsilon(x)\setminus\continua_\delta(x)\neq\emptyset$ 
 imply that each stable plaque has empty interior. 
 Therefore, stable and unstable leaves are meagre sets.
This proves that $\W^\sigma$ are cw-foliations.

Suppose that $L$ is a plaque-compact leaf of $\A^s$. 
Let $\{P_i\}_{i\in I}$ be a cover of $L$ 
by stable plaques. 
Since $L$ is plaque-compact we can assume that $I$ is finite. 
Then $\diam(f^n(L))\to 0$ as $n\to+\infty$. 
This gives us stable leaves of arbitrarily small diameter, a contradiction to Theorem \ref{thmExConNT}.
\end{proof}

%

%

\begin{rmk}
 The quasi-Anosov diffeomorphism \S \ref{secQAnosov}
 is expansive and Theorem \ref{thmCovCwExpFol} can be applied. 
 In this example we see that the plaques of $\W^s$ and $\W^u$ may not have constant dimension. 
 There are stable plaques of dimension 1 and 2.
\end{rmk}

\begin{exa}[One-dimensional expanding attractor]
\label{exaSolenoide}
 Let $g\colon T^2\to T^2$ be a derived from Anosov diffeomorphism of the two-dimensional torus
 and denote by $X\subset T^2$ the non-trivial basic set (an expanding attractor). 
 Let $f\colon X\to X$ be the restriction of $g$. 
 We have that $X$ is locally a product of a Cantor set and an arc and every leaf is dense in $X$. 
 The dynamics of $f$ on $X$ expands the length of each arc contained in $X$ 
 and every proper subcontinuum is an unstable arc. 
 That is, given an open set $U\subset X$, such that $\clos{U}\neq X$, 
 the unstable plaque of $x\in U$ is, simply, its component in $U$ (an arc). 
 Then, the unstable cw-foliation is a Peano foliation.
 It is remarkable that in this case, 
 the unstable atlas is compatible with the trivial decomposition $Q\colon X\to\continua(X)$ given by 
 $Q(x)=X$ for all $x\in X$. However, $X$, as a plaque of this decomposition $Q$, is neither unstable nor locally connected. 
 Then, a maximal atlas extending a Peano unstable atlas may not be neither unstable nor Peano.
\end{exa}

\begin{prob}
For a cw-expansive homeomorphism of a Peano continuum, is it true that every continuum contained on a stable leaf is stable?  
 \end{prob}

 \begin{prob}Can a stable leaf of a cw-expansive homeomorphism of a Peano continuum 
be compact in the relative topology?
By Theorem \ref{thmCovCwExpFol} we only know that it cannot be 
plaque-compact.
\end{prob}

\subsection{Generating pairs}

If $\U^1$ and $\U^2$ are complete bases then $\U^1\cap \U^2$ is a complete basis. 
If $\A^i=\{Q^i_U:U\in\U^i\}$ is an atlas over the complete basis $\U^i$, $i=1,2$, 
then both atlases are defined in a common complete basis, namely $\U=\U^1\cap \U^2$. 

The following concept is the key to prove in \S \ref{teoEquivPANo1prong}
that cw1-expansivity implies expansivity on a compact surface.

\begin{df}
\label{dfGenPair}
Given two atlases $\A^1$ and $\A^2$ of the continuum $X$, we say that 
they \emph{generate (the topology of $X$)} if for all $B_\epsilon(x)\in\U$ there is 
$\delta>0$ such that if $y\in B_\delta(x)$ then 
$Q^i_{B_\epsilon(x)}(x)\cap Q^j_{B_\epsilon(x)}(y)\neq\emptyset$ 
for $\{i,j\}=\{1,2\}$.
\end{df}

\begin{exa} 
A pair of transverse foliations of a smooth manifold (with the standard meaning of these concepts) 
form a generating pair of cw-foliations. 
Another example is formed by 
the stable and the unstable singular foliations of a pseudo-Anosov map of a compact surface, 
even with singularities and 1-prongs.\footnote{We were tempted to say \emph{transversal} instead of \emph{generating pair} in Definition \ref{dfGenPair}, 
but at 1-prongs the foliations look closer to a tangency than to a transversal point, at least from author's viewpoint.} 
\end{exa}

\begin{rmk}
\label{rmkNeighSU}
For $r>0$ and $x\in X$ define 
\[
  su_r(x)=\{y\in X:\exists A^s\in \continua^s_r(y),A^u\in \continua^u_r(x)\text{ with } A^u\cap A^s\neq\emptyset\}
\]
Also define $y\in us_r(x)$ if $x\in su_r(y)$.
If $f\colon X\to X$ is cw-expansive then the cw-foliations $\A^s$ and $\A^u$ generate if and only if
$su_\epsilon(x)\cap us_\epsilon(x)$ is a neighborhood of $x$, for all $x\in X$ and for all $\epsilon>0$.  
The proof follows by the definitions.
\end{rmk}

\begin{rmk}
Notice that for pseudo-Anosov singular foliations 
the size of the ball covered by $su_r(x)$ is not uniform.
For $x$ close to a singularity, the maximal ball contained in $su_r(x)$ is small. 
\end{rmk}

\begin{exa}[Quadratic tangencies]
 Let $\W^1$ be the foliation of $\R^2$ whose leaves are horizontal lines $y=c$ for $c\in\R$. 
 Define $\W^2$ by $y=x^2+c$ for $c\in\R$. 
 It is easy to prove that these foliations do not generate at tangency points of the foliations. 
\azul{These kind of tangencies are present in the wandering set of the 
 Q{$^r$}-Anosov diffeomorphisms of \S \ref{secQrAnosov}.}
\end{exa}

\begin{exa}
Consider the vertical cw-foliation $\W^1$ of the plane and a cw-foliation $\W^2$ as in Figure~\ref{figFolNonGen}. 
In this case, for a point $x$ as in the figure, it holds that 
$\W^2|^m_{B_\epsilon(x)}(x)\cap \W^1|^m_{B_\epsilon(x)}(y)\neq\emptyset$ 
for all $y$ close to $x$. 
But, exchanging the foliations, we can find points $y$ arbitrarily close to $x$ such that
$\W^1|^m_{B_\epsilon(x)}(x)\cap \W^2|^m_{B_\epsilon(x)}(y)=\emptyset$.
This kind of non-generating points appear: 1) at the non-wandering set of the examples in \S \ref{secQrAnosov}
and 2) at wandering points in \azul{the example of \S \ref{secAnomalous}}. 

\end{exa}

 \begin{figure}
 \center
 \includegraphics[scale=1]{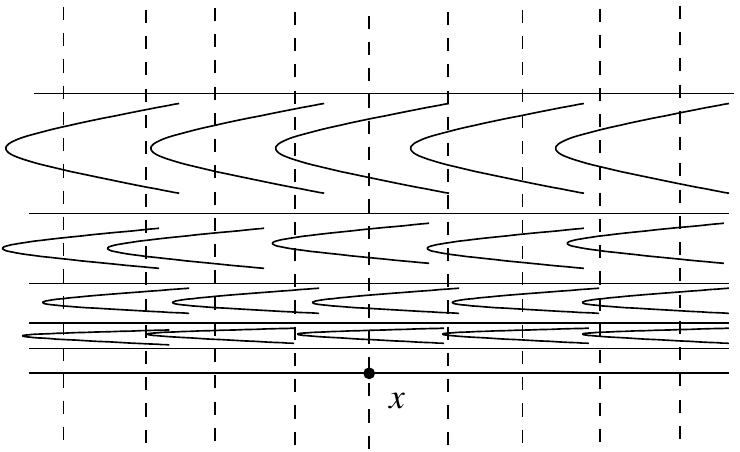}
 \caption{A non generating pair of plane cw-foliations}
 \label{figFolNonGen}
 \end{figure}

\begin{prop}
\label{propsurloccon}
  If $su_r(x)$ is a neighborhood of $x$, for all $x\in X$ and for all $r>0$ then $X$ is locally connected. 
\end{prop}

\begin{proof}
  Given $x\in X$ and a neighborhood $U$ of $x$ take $\rho>0$ such that $B_\rho(x)\subset U$. 
  Take $r\in (0,\rho/2)$. 
  We know that $su_r(x)$ is a neighborhood of $x$. 
  For each $y\in su_r(x)$ there are $C^s\in \continua^s_r$ and $C^u\in \continua^u_r$ 
  such that $x\in C^u$, $y\in C^s$ and $C^u\cap C^s\neq\emptyset$. 
  Then, $\diam(C^s\cup C^u)\leq 2r$ and $\dist(x,y)<\rho$. 
  Then $su_r(x)\subset B_\rho(x)$. 
  Since $su_r(x)$ is a union of connected sets $C^s\cup C^u$, with a common point $x$, 
  we conclude that $su_r(x)$ is connected. This proves the local connection of $X$ at $x$. 
  Since $x$ is arbitrary, $X$ is locally connected.
\end{proof}


\subsection{Cw1-expansivity versus expansivity}
\label{secCw1VsExp}
We now investigate the following question. 
\begin{prob}
Does cw1-expansivity imply the expansivity of a homeomorphism $f$ on a Peano continuum $X$? 
\end{prob}

The answer is affirmative in the examples known by the author and in 
Proposition \ref{propCwnNexpExpMod} we gave some information on this problem.
In Theorem \ref{teoEquivPANo1prong} we will prove it for $X$ a compact surface.
On a Cantor set $X$ the identity is (trivially) cw1-expansive but not expansive.
Even assuming the connection of $X$, the local connection of $X$ is needed in the question because of the following example. 

\begin{exa}
\label{exaCw1VsExp}
Let $f_1\colon T^2\to T^2$ be an Anosov diffeomorphism with a fixed point $p$. 
Denote by $X_2$ the closure of $\{1/n:n\in \Z^+\}$ in the usual topology of $\R$. 
Define $Y=T^2\times X_2$ and $g\colon Y\to Y$ as $g(x,y)=(f_1(x),y)$. 
Now, collapsing the fixed points $\{(p,0),(p,1),(p,1/2),\dots\}$ to a single point we obtain a 
continuum $X$, see Figure~\ref{figCw1}. 
Also, we have a homeomorphism $f\colon X\to X$ induced by $g$. 
We have that $f$ is not expansive because the distance $\dist(f^n(x,1/k),f^n(x,1/l))$ is small for 
all $n\in\Z$ whenever $k$ and $l$ are sufficiently large.
By construction, $f$ is cw1-expansive. 
Notice that $X$ is a non-locally connected continuum. 
\end{exa}

\begin{figure}[ht]
\center
  \includegraphics[scale=.5]{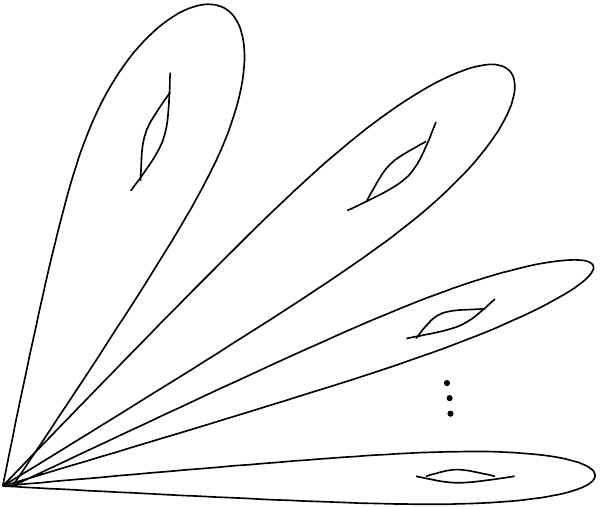}
  \caption{A non-expansive but cw1-expansive homeomorphism on a continuum that is not locally connected. 
  An Anosov diffeomorphism acts on each torus.}
  \label{figCw1}
\end{figure}

Under certain hypothesis, see Theorem \ref{thmCw1exp},
we can prove that cw1-expansivity implies expansivity. 
The following lemma is a part of its proof,
it is separated in order to simplify the notation. 
We assume that $f$ is not expansive because we will argue by contradiction. 

\begin{lem}
\label{lemCw1}
If $f$ is not expansive 
and $us_r(x)$ is a neighborhood of $x$ for all $x\in X$ and for all $r>0$,
then for all $\epsilon>0$ there are 
$x,y\in X$ such that 
$x\neq y$, $\dist(f^n(x),f^n(y))<\epsilon$ 
for all $n\in\Z$ and $x,y\in A$ with 
$A\in \continua^s_\epsilon\cup \continua^u_\epsilon$.
\end{lem}

\begin{proof}
For $\epsilon>0$ given, 
consider $\delta_1>0$ and a finite set $Z\subset X$ such that 
$B_{\delta_1}(z)\subset us_{\epsilon/3}(z)$ for each $z\in Z$ and 
$X=\cup_{z\in Z}B_{\delta_1}(z)$.
Take $\delta_2\in(0,\epsilon/3)$ such that if $v,w\in X$ and $\dist(v,w)<\delta_2$ then $v,w\in B_{\delta_1}(z)$ 
for some $z\in Z$.
Since $f$ is not expansive there are $v,w\in X$ such that $\dist(f^n(v),f^n(w))<\delta_2$ for all $n\in\Z$ 
and $v\neq w$. 
Take $z\in Z$ such that $v,w\in B_{\delta_1}(z)$. 
Then $v,w\in su_{\epsilon/3}(z)$.
By the definitions, we can take 
$A_v^s, A_w^s\in \continua^s_{\epsilon/3}$ and 
$A_v^u,A^u_w\in \continua^u_{\epsilon/3}$
such that $z\in A_v^s\cap A_w^s$, $A_v^s\cap A_v^u\neq\emptyset$, $A_w^s\cap A_w^u\neq\emptyset$, 
$v\in A_v^u$ and $w\in A_w^u$. 

Take $v_1\in A^s_v\cap A^u_v$ and $w_1\in A^s_w\cap A^u_w$. 
Since $A^s_v,A^s_w\in \continua^s_{\epsilon/3}$ with a common point, namely $z$, 
their union is a $2\epsilon/3$-stable continuum. 
Since $v_1,w_1\in A^s_v\cup A^s_w$ we have that 
\begin{equation}
  \label{ineqLem}
  \dist(f^n(v_1),f^n(w_1))<2\epsilon/3\text{ for all }n\geq 0.
\end{equation}
Also $\dist(f^n(v_1),f^n(w_1))$ 
is smaller or equal than $$\dist(f^n(v_1),f^n(v))+\dist(f^n(v),f^n(w))+\dist(f^n(w),f^n(w_1)).$$
For $n\leq 0$ we have that 
\[
 \begin{array}{l}
\dist(f^n(v_1),f^n(v))\leq \epsilon/3,\\
\dist(f^n(v),f^n(w))<\delta_2<\epsilon/3\text{ and }\\
\dist(f^n(w),f^n(w_1))\leq\epsilon/3.  
 \end{array}
\]
Therefore $\dist(f^n(v_1),f^n(w_1))<\epsilon$ for all $n\leq 0$. 
Recalling (\ref{ineqLem}) we conclude that 
$$\dist(f^n(v_1),f^n(w_1))<\epsilon$$
for all $n\in\Z.$

Recall that $A^s_v\cup A^s_w$ is a $2\epsilon/3$-stable continuum containing $v_1$ and $w_1$. 
If $v_1\neq w_1$ then we take $x=v_1$, $y=w_1$ and the $\epsilon$-stable continuum $A=A^s_v\cup A^s_w$.
If $v_1=w_1$ then we take $x=v$, $y=w$ and the $2\epsilon/3$-unstable continuum $A=A^u_v\cup A^u_w$.
\end{proof}

\azul{Notice that by Remark \ref{rmkNeighSU} and Proposition \ref{propsurloccon} the hypothesis of the next result imply that $X$ is locally connected.}

\begin{thm}
\label{thmCw1exp}
If $f\colon X\to X$ is a cw1-expansive homeomorphism of a continuum $X$
and $\W^s,\W^u$ is a generating pair 
then $f$ is expansive. 
\end{thm}

\begin{proof}
Take $\gamma>0$ such that 
\begin{equation}
  \text{ if }A^s\in \continua^s_\gamma \text{ and }A^u\in \continua^u_\gamma \text{ then }\card(A^s\cap A^u)\leq 1.
  \label{ecu1lem} 
\end{equation}
Consider $\delta\in(0,\gamma/3)$ such that
\begin{equation}
  \text{ if }\diam(A)\leq 2\delta\text{ then }
  \diam(f^{\pm 1}(A))<\gamma/3
  \label{ecu2lem}
\end{equation}
for every $A\subset X$.
Consider $\epsilon\in (0,\delta)$ such that 
\begin{equation}
  \text{ if }\dist(a,b)<\epsilon\text{ then there is }z\in X\text{ such that }a,b\in su_{\delta}(z)\cap us_{\delta}(z)
  \label{ecu3lem}
\end{equation}
for any $a,b\in X$.
Assume by contradiction that $f$ is not expansive.
  From Lemma \ref{lemCw1} there are $x,y\in X$ and 
  $A^s_0\in \continua^s_\epsilon$ such that $x\neq y$, $\dist(f^n(x),f^n(y))<\epsilon$ for all $n\in\Z$
  and $x,y\in A^s_0$. 
  If the lemma gives an $\epsilon$-unstable continuum then change $f$ by $f^{-1}$ in what follows. 

In this paragraph we will show that there is $n\geq 0$ such that $f^{-n}(x)$ and $f^{-n}(y)$ are not in a common 
$2\delta$-stable continuum. 
Arguing by contradiction, 
assume that for each $n\geq 0$ there is $A^s_n\in \continua^s_{2\delta}$ containing $f^{-n}(x)$ and $f^{-n}(y)$. 
Since $f$ is cw-expansive and $\delta$ is small, we have that $\diam(f^n(A^s_n))\to 0$ as $n\to\infty$. 
Since $x,y\in f^n(A^s_n)$ we conclude that $\dist(x,y)=0$, contradicting that $x\neq y$.

Let $n_0$ be the first positive number satisfying the condition proved in the previous paragraph. 
Then $f^{-n_0+1}(x)$ and $f^{-n_0+1}(y)$ are in a common $2\delta$-stable continuum $A^s_{n_0-1}$. 
From (\ref{ecu2lem}) we have that $\diam(f^{-1}(A^s_{n_0-1}))<\gamma/3$. 
Define $a=f^{-n_0}(x)$, $b=f^{-n_0}(y)$ and $A^s_*=f^{-1}(A^s_{n_0-1})$. 
We have that $A^s_*$ is a $\gamma/3$-stable continuum containing $a$ and $b$. 
Also, $a$ and $b$ are not in a common $2\delta$-stable continuum.

Since $\dist(a,b)<\epsilon$  
we can apply condition (\ref{ecu3lem}) to conclude that there is $z\in X$ such that $a,b\in su_{\delta}(z)$. 
Then, there are 
$A^s_a,A^s_b\in \continua^s_\delta$ and 
$A^u_a,A^u_b\in \continua^u_\delta$ such that 
$z\in A^u_a\cap A^u_b$, $a\in A^s_a$, $b\in A^s_b$, 
$A^s_a\cap A^u_a\neq\emptyset$ and
$A^s_b\cap A^u_b\neq\emptyset$.
We have that $A^s_a\cap A^s_b=\emptyset$ because $a$ and $b$ are not in a common $2\delta$-stable continuum. 
Therefore, we can take $p\in A^s_a\cap A^u_a$ and
$q\in A^s_b\cap A^u_b$ such that $p\neq q$. 
Since $\delta<\gamma/3$ we have that $p$ and $q$ are in a $\gamma$-stable continuum, namely, $A^s_a\cup A^s_*\cup A^s_b$. 
Also, $p$ and $q$ are in the $2\delta$-unstable continuum $A^u_a\cup A^u_b$. 
This contradicts (\ref{ecu1lem}) and proves that $f$ is expansive.
\end{proof}

\begin{rmk} 
\azul{The quasi-Anosov diffeomorphism of \S \ref{secQAnosov} is expansive but
it does not satisfy 
the hypothesis 
of Theorem \ref{thmCw1exp}.}
In this case, the set $su_\delta(x)\cap us_\delta(x)$ is not a neighborhood of a wandering point $x$. 
This means, if it is true that cw1-expansivity implies expansivity on a Peano continuum, 
then new arguments will be needed for its proof.
\end{rmk}

\section{Dendritations of surfaces}
\label{secDendritations}

In this section we will study cw-foliations of compact surfaces. 
Applications to cw-expansive surface homeomorphisms will be given.

\subsection{Cw-decompositions of discs}
\label{CwFCwExp}

Let $D$ be a metric space homeomorphic to the Euclidean disc 
$$\{(x,y)\in\R^2:x^2+y^2\leq 1\}.$$
As usual, $\partial D$ \azul{is the boundary of the disc}. 

\begin{df}
 A decomposition of $D$ is a \emph{cw-decomposition} 
 if it is codendritic and hereditarily unicoherent. 
 That is, the quotient space is a dendrite and each plaque is hereditarily unicoherent.
\end{df}

We say that $Q$ is $n$-\emph{dimensional} if $\dim Q(x)=n$ for all $x\in D\setminus\partial D$. 
\azul{As in \S \ref{secPartExp}, $\dim$ denotes the topological dimension}. 

\begin{prop}
\label{propHerUnico}
If $Q$ is a decomposition of $D$ then the following statements are equivalent:
\begin{enumerate}
 \item $Q$ is a cw-decomposition,
 \item $Q(x)\cap\partial D\neq\emptyset$ and $Q$ is 1-dimensional for all $x\in D$.
\end{enumerate}
\end{prop}

\begin{proof}
($1\to 2$). 
Arguing by contradiction, suppose that $Q(x)\cap\partial D=\emptyset$ for some $x\in D$.
Since $Q$ is codendritic there is $y\in D$ close to $x$ such that the plaque $Q(y)$ separates $D$ and 
$Q(y)\cap\partial D=\emptyset$. 
By Janisewski's Theorem \ref{teoJanis} $Q(y)$ is not unicoherent, contradicting that $Q$ is a cw-decomposition.
To prove that $Q$ is 1-dimensional, note that each plaque $Q(x)$ has dimension 0, 1 or 2. 
If some plaque has dimension 2 then it has interior points and then it is not hereditarily unicoherent. 
For an interior point $x$ of the disc it holds that $\dim(Q(x))\neq 0$, because it is a continuum meeting $\partial D$.

($2\to 1$). Let us first show that $Q$ is hereditarily unicoherent. Arguing by contradiction suppose that 
$Q(x)$ is not hereditarily unicoherent. 
Then there are two continua $A,B$ such that $Q(x)=A\cup B$ and $A\cap B$ is disconnected. 
By Janisewski's Theorem $Q(x)$ separates $D$. Let $U$ be a component of $D\setminus Q(x)$ 
disjoint from $\partial D$.
This implies that $U\subset Q(x)$, contradicting that $Q(x)$ is one-dimensional. 

Now we prove that $Q$ is codendritic.
First we show that if $x,y\in D$  
and $Q(x)\cap Q(y)=\emptyset$ then there is $z\in D$ such that $Q(z)$ separates $Q(x)$ from $Q(y)$.   
Let $I,J$ be two arcs contained in $\partial D$ such that $Q(x)\cap\partial D\subset I$ 
and $Q(y)\cap\partial D\subset J$. 
Suppose that $I,J$ are minimal with this property.
Since $Q(x)$ and $Q(y)$ are disjoint we have that $I\cap J=\emptyset$. 
Denote by $\alpha$ and $\beta$ the \azul{open} arcs in $\partial D\setminus (I\cup J)$.
Take an arc $\Gamma\subset D$ from $a\in\alpha$ to $b\in\beta$ separating $Q(x)$ and $Q(y)$ 
as in Figure~\ref{figCoDend}.
\begin{figure}[ht]
 \center
 \includegraphics[scale=1]{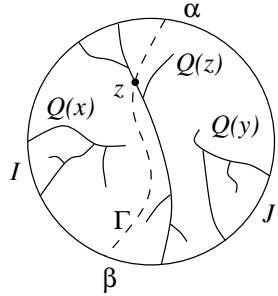}
 \caption{The plaque $Q(z)$ separates $Q(x)$ from $Q(y)$.}
 \label{figCoDend}
\end{figure}

Define the sets
\[
 \begin{array}{l}
\Gamma_\alpha=\{u\in \Gamma: Q(u)\cap \alpha\neq\emptyset\},\\
\Gamma_\beta=\{u\in \Gamma: Q(u)\cap \beta\neq\emptyset\}. 
 \end{array}
\]
We have that $a\in \Gamma_\alpha$ and $b\in\Gamma_\beta$, which implies that they are non-empty. 
\azul{Since $\Gamma$ is disjoint from $Q(x)$ 
and $Q(y)$, the upper semicontinuity of $Q$ implies that $\Gamma_\alpha$ and $\Gamma_\beta$ 
are closed, recall Remark \ref{rmkUpSCont}.} 
As every plaque intersects $\partial D$ we have that $\Gamma=\Gamma_\alpha\cup \Gamma_\beta$. 
Given that $\Gamma$ is connected, we find $z\in \Gamma_\alpha\cap \Gamma_\beta$. 
Therefore, $Q(z)$ separates $Q(x)$ from $Q(y)$.

By Proposition \ref{propCoc1} we know that $D/Q$ is a continuum
and from 
Theorem \ref{teoCharDend} 
we conclude that $Q$ is codendritic.
\end{proof}


\begin{prop}
\label{propSepara3}
 If $Q$ is a cw-decomposition and $x_1,x_2,x_3\in D$ are such that 
 $Q(x_i)$ does not separate $Q(x_j)$ from $Q(x_k)$ for $\{i,j,k\}=\{1,2,3\}$ 
 then there is $x\in D$ such that $Q(x_1),Q(x_2),Q(x_3)$ are in different components of $D\setminus Q(x)$. 
\end{prop}

\begin{proof}
Knowing that $D/Q$ is a dendrite the result follows by general properties of such spaces.
\end{proof}

\begin{rmk}
 On a dendrite there can be four points that cannot be separated simultaneously with a single point. 
 Consider for example a dendrite with the shape of the letter $H$ and try to separate with one single point the four ends. 
 It is easy to construct a decomposition $Q$ such that $D/Q$ has this topology.
 See Figure~\ref{figH}.
 \begin{figure}[ht]
 \center
 \includegraphics[scale=1]{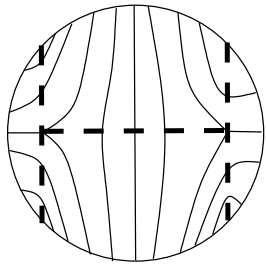}
 \caption{The quotient dendrite has the shape of the letter $H$.}
 \label{figH}
\end{figure}
\end{rmk}

\begin{rmk}
 It seems that given any tree $T$ (i.e., a dendrite with a finite number of ramification points) 
 there is a decomposition $Q$ of $D$ such that $D/Q$ is homeomorphic to $T$. 
 It also seems that this is still true for an arbitrary dendrite $T$, but a proof is not so clear for the author.
\end{rmk}

\begin{rmk}
In the study of \emph{complex dynamics}, several definitions of laminations can be found. 
Following \cite{McM}, a \emph{lamination} is a closed equivalence relation on $\partial D$ 
such that the convex hull of different equivalence classes are disjoint. 
The convex hull is considered with respect to (a differentiable structure and) a Riemannian metric on $D$ that makes it isometric with an Euclidean disc. 
There is a strong link between cw-decompositions  and laminations:
if $Q$ is a cw-decomposition of $D$ then the restriction of $Q$ to the boundary of $D$ is a lamination.
\end{rmk}

\begin{prop}
\label{propCwFolDec}
 If $\A$ is a cw-foliation over the complete basis $\U$ of a closed surface $S$ then
 there is $\delta>0$ such that $Q_{\clos U}$ is a cw-decomposition for all $U\in \U$ with $\diam(U)<\delta$. 
\end{prop}

\begin{proof}
It follows by 
Propositions \ref{propCwFolDec1} and \ref{propHerUnico}.
\end{proof}


%

%
%

\subsection{Dendritations and generic leaves}

The following is the main definition of the paper.
It combines the properties of Peano foliations, 
cw-foliations and the plane topology. 

\begin{df}
A \emph{dendritation} is a Peano cw-foliation of a compact surface.
\end{df}

By Remark \ref{rmkNoboundary} we know that compact surfaces with non-empty boundary admits no cw-expansive homeomorphisms. 
Then, we will assume that the surface has empty boundary.
Recall from (\ref{ecuUdeltaa}) the definition of the complete basis $\U_\delta$.

\begin{prop}
\label{propLocChartDendritation}
 If $\A=\{Q_U:U\in\U\}$ is an atlas of a dendritation, $\delta>0$ is sufficiently small, $U\in\U_\delta$ and 
 $D=\clos U$ is a disc then 
 each plaque $Q_D(x)$ is a dendrite and 
 the quotient $D /Q_D$ is a dendrite.
\end{prop}

\begin{proof}
By definition, we know that $\A$ is a cw-atlas. 
Take $\delta$ from Proposition \ref{propCwFolDec} such that $Q_D$ is a cw-decomposition. 
By the definition of cw-decomposition we know that $D/Q_D$ is a dendrite 
and that each plaque is a hereditarily unicoherent continuum. 
By hypothesis, the plaques are locally connected.
Applying Theorem \ref{teoCharDend} we conclude the
plaques are dendrites. 
\end{proof}

\begin{rmk} Knowing that each plaque is a dendrite 
it is natural to look for more topological information given that they are contained in the plane. 
Unfortunately, we must deal with the worst kind of dendrites because every dendrite can be embedded in the plane. 
Recall the Wazewski's universal dendrite, Theorem \ref{teoWaz}.
However, generic dendrites in a dendritation are nice, as we will now explain.
\end{rmk}

We will study the properties of the plaque of a point in a residual subset of the disc. 
A property is \emph{generic} if it holds in a residual subset.
For this purpose we recall that
in \cite{Mo28} Moore proved  that every family of mutually disjoint triods of the plane is countable. 
An alternative proof can be found in \cite{Pittman}. 
Recall that a \emph{triod} is a union of three arcs, 
homeomorphic to $([-1,1]\times\{0\})\cup (\{0\}\times [0,1])$ in the plane (three segments with a common end point).


\begin{prop}
\label{propGenDend}
If $Q$ is a cw-decomposition of the disk $D$ then there is a residual set $G\subset D$ such that 
$\card(Q(x)\cap\partial D)\leq 2$
for all $x\in G$.
\end{prop}

\begin{proof}
First we show that there is a residual set $G_1\subset D$ such that $Q(x)\cap\partial D$ is totally disconnected 
for all $x\in G_1$.
Let $A\subset \partial D$ be a countable dense subset. 
If $Q(x)\cap\partial D$ contains a non-trivial arc then 
$A\cap Q(x)\neq\emptyset$. 
Therefore, at most a countable family of plaques intersects $\partial D$ in a non totally disconnected set. 
The complement of the union of such plaques gives our residual set $G_1$.

By Proposition \ref{propHerUnico} we know that $Q$ is codendritic. 
Then, by Proposition \ref{propCoDendGen} there is a residual set $G_2\subset D$ such that 
$D\setminus Q(x)$ has 1 or 2 components for all $x\in G_2$.
By Janisewski's Theorem this implies that $Q(x)\cap\partial D$ has 1 or 2 components. 
For $x\in G=G_1\cap G_2$ we have that $Q(x)\cap\partial D$ has 1 or 2 points.
\end{proof}

In the examples known to the author of decompositions defined from stable sets of cw-expansive homeomorphisms 
it is true that 
$\card(Q(x)\cap\partial D)=2$ on a residual set. 
In general, for arbitrary decompositions, this is not true as the next example shows.

\begin{exa}
\label{exaEspinaGenerica}
We will construct a cw-decomposition $Q$ of a disc $D$ 
such that for every residual set $G\subset D$ there is $x\in G$ whose dendrite $Q(x)$ does not separate $D$. 
 Consider a pseudo-Anosov diffeomorphism on a hyperbolic surface. 
 Denote by $U=\{(x,y)\in\R^2:x^2+y^2<1\}$ the universal cover of the surface and denote 
 by $\tilde F^s$ the lifting to $U$ of the stable singular foliation of the pseudo-Anosov diffeomorphism. 
 On the closed disc $D'=\clos{U}$ (the closure in the usual topology of $\R^2$) 
 consider the dendritic decomposition where
 $Q'(p)$ is the closure of $\tilde F^s(p)$. 
 Consider the annulus $A=\{(x,y)\in\R^2:1\leq x^2+ y^2\leq 2\}$ and the disc $D=U\cup A$. 
 We will extend $Q'$ to $D$. 
 For $p\in\partial U$ define $A_p=\{\lambda p:1\leq\lambda\leq 2\}$ a line segment from $p$ to $\partial D$ contained in the line through $p$ and the origin. 
 For $p\in D'$ define 
 \[
  Q(p)=Q'(p)\cup\bigcup_{q\in Q'(p)\cap\partial D'} A_q.
 \]
 It is easy to prove that $Q$ defines a dendritic decomposition on $D$.

 Let us show that for a generic point $p\in A$ we have that $Q(p)$ is a non-separating arc. 
 The proof follows the ideas in \cite{AHNO}*{Lemma 3.1}. 
 Let $D_n$ be the closed disc of radius $1-1/n$ centered at the origin. 
 Define $T_n=Q^{-1}(Q(D_n))$. 
 Notice that $$\bigcup^\infty T_n=\{p\in D:Q(p)\text{ separates } D\}.$$
 Also $A\setminus T_n$ is open and dense in $A$. 
 Define $G=A\setminus(\cup^\infty T_n)$.
 Therefore, $G$ is residual in $A$ and for all $p\in A$ the plaque $Q(p)$ is an arc of the form $A_q$ (containing $p$), 
 and this arc does not separate $D$.
\end{exa}

A decomposition is \emph{arc-connected} if each plaque is arc-connected.

\begin{prop}
\label{propGenDend2}
 If $Q$ is an arc-connected cw-decomposition on $D$ then there is a residual set $G\subset D$ 
 such that $Q(x)$ is an arc for all $x\in G$.
\end{prop}

\begin{proof}
If a plaque is not an arc then it has a ramification point and contains a triod. 
\azul{By Moore's Theorem \cite{Mo28} at most a countable number of disjoint triods can be embedded in the plane.}
Then, the set of plaques with ramification points is countable. 
Denote by $Q_1, Q_2,\dots$ such plaques.
Since each $Q_n$ is closed and has empty interior, its complement in the disc $D$, $U_n=D\setminus Q_n$ is an open and dense subset of $D$. 
Then, $G=\cap_{n\geq 1} U_n$ is a residual set of points $x$ such that $Q(x)$ is an arc.
\end{proof}

\azul{If $\W$ is a dendritation of a surface we say that $y\in\W(x)$ is a \emph{ramification point} if 
$\W(x)\setminus \{y\}$ has at least 3 components in the plaque topology.}

\begin{thm}
\label{thmGenCwfSup}
If $\W$ is a dendritation of a compact surface $S$ then there is a residual set $G\subset S$ such that $\W(x)$ has no ramification for all $x\in G$.
For all $x\in G$ the leaf $\W(x)$ with the plaque topology is a one-dimensional manifold: $\R$, $[0,+\infty)$, $[0,1]$ or the circle.
\end{thm}

\begin{proof}
By Proposition \ref{propTauAMetric} and Theorems \ref{teoBaseNumLeaf}, \ref{teoUnaPLC2}
we know that: $\tau_\W$ is a metric topology, the leaves are arc-connected in $\tau_\W$ and 
each leaf has a countable basis with respect to $\tau_\W$.
Arguing as in the proof of Proposition \ref{propGenDend2} we have that the number of 
leaves with ramification points is at most countable. 
Since each leaf is a meagre set (by the definition of cw-foliation), the set of points in leaves without ramifications is a residual set $G$.
Then, the result follows by the classification of one-dimensional manifolds, see for example \cite{Lee}.
\end{proof}

\begin{rmk}
On a residual set, as in Theorem \ref{thmGenCwfSup}, the one-dimensional manifolds can occur. 
For example, a torus can be foliated by circles. 
Also, a minimal flow on the torus has all the leaves homeomorphic to $\R$. 
Using Example \ref{exaEspinaGenerica} we can define a cw-foliation of the sphere (identifying two discs by the boundary) 
with a residual set of 
leaves homeomorphic to one-dimensional manifolds with boundary.
\end{rmk}

\begin{rmk}
It would be interesting to extend Theorem \ref{thmGenCwfSup}
without assuming the local connection of the plaques.
\end{rmk}

In the next result we will use the stable and unstable 
cw-foliations $\W^s,\W^u$ that were constructed in \S \ref{secSUCwfols} for an arbitrary 
cw-expansive homeomorphism of a
Peano continuum (as our compact surface $S$).

\begin{thm}
\label{thmCwExpGenDen}
If $f\colon S\to S$ is cw-expansive homeomorphism and $\W^s,\W^u$ are dendritations then 
 there is a residual set $G\subset S$ such that for all $x\in G$ 
 it holds that $\W^s(x)$ and $\W^u(x)$ are homeomorphic to $\R$ or $[0,+\infty)$.
\end{thm}

\begin{proof}
Take from Theorem \ref{thmGenCwfSup} two residual sets $G^s$ and $G^u$ corresponding to $\W^s$ and $\W^u$ respectively 
and define $G=G^s\cap G^u$. 
For $x\in G$ we know that the stable and the unstable leaf of $x$ is a one-dimensional manifold. 
In Theorem \ref{thmCovCwExpFol} we proved that no leaf $\W^\sigma(x)$ is plaque-compact.
By Theorem 
\ref{thmPCompIffPeano}, this implies that no leaf $\W^\sigma(x)$ is a Peano continuum in the relative topology of $\tau$ (the topology of the surface).
Since $[0,1]$ and the circle are Peano continua 
we know from Theorem \ref{thmCovCwExpFol} that $\W^s(x)$ and $\W^u(x)$ is homeomorphic to $\R$ or $[0,+\infty)$ for each $x\in G$.
\end{proof}

\begin{rmk}
 In the examples known to the author, for every cw-expansive homeomorphism of a compact surface there is
 a residual set of points whose stable and unstable leaves are homeomorphic to $\R$ 
 and just a countable number of leaves are homeomorphic to $[0,+\infty)$.
\end{rmk}

\subsection{Graph like continua and decompositions}
\label{secGraphLike}
In this section we develop a technique to construct examples of decompositions.

Let $I, J\subset \R$ be compact intervals. 
Consider a continuum with empty interior $C\subset I\times J$ 
disjoint from $I\times\partial J$ and 
such that for each $x \in I$ the set $C\cap [\{x\}\times J]$ is connected and non-empty. 
In this case we say that $C$ is a \emph{graph like continuum}. 
For such continuum $C$ define 
\[
\begin{array}{l}
 U^+=\{(x,y)\in I\times J: \text{ if } (x,y')\in C\text{ then } y>y'\},\\
 U^-=\{(x,y)\in I\times J: \text{ if } (x,y')\in C\text{ then } y<y'\}.
 \end{array}
\]
We have that $I\times J$ is the disjoint union of $C, U^+$ and $U^-$.

We will define a decomposition of $I\times J$ associated to a graph like continuum $C$. 
Define $v\colon I\times \R\to\R^2$ by $v(p)=(0,\dist(p,C))$,
where $\dist(p,C)=\min\{\dist(p,q):q\in C\}$. 
\azul{Since $v$ is Lipschitz, it} defines a continuous flow $\phi$ on $I\times\R$ with equilibrium points at every point of $C$. 
For $p\in I\times J$ define $Q_C(p)=C$ if $p\in C$ and 
\[
 Q_C(p)=\{q\in I\times J:\exists t\in\R\text{ such that }\phi_t(p),\phi_t(q)\in I\times\partial J\}
\]
if $p\notin C$.
\azul{Recall from \S \ref{secPartandMonRes} that 
the continuity of a decomposition is with respect to the Hausdorff metric.}

\begin{prop}
 If $C\subset I\times J$ is a graph like continuum then 
 $Q_C$ is a cw-decomposition and: 
 \begin{enumerate}
  \item $Q_C(p)$ is the graph of a continuous map $I\to J$ for all $p\notin C$,
  \item $Q_C$ is continuous at every $p\notin C$, 
  \item $Q_C$ is continuous (at every point in $C$) if and only if $C\cap\partial U^+=C\cap\partial U^-$,
  \item $[I\times J]/Q_C$ is an arc.
 \end{enumerate}
\end{prop}

\begin{proof}
 Fix $p\in U^+$. 
 Define $L^{\pm}$ as the components of $I\times\partial J$ contained in $\partial U^\pm$ respectively.
 Take $T(p)\in\R$ such that $\phi_{T(p)}(p)\in L^+$. 
 Since $L^\pm$ are transverse to $\phi$ we have that $T\colon U^+\to\R$ is continuous. 
 From the definition of $Q_C$ we have that $Q_C(p)=\phi_{-T(p)}(L^+)$, given that $p\in U^+$. 
 Therefore, $Q_C(p)$ is the graph of a continuous map for all $p\in U^+$. 
 In a similar way this conclusion holds for all $p\in U^-$. 
 This implies that each $Q_C(p)$ is a continuum meeting $\partial (I\times J)$. 
 Also, no $Q_C(p)$ has interior points. 
 
 In order to conclude that $Q_C$ is a decomposition it only remains to prove 
 the upper semicontinuity. 
 The continuity of the flow $\phi$ implies that the functions whose graphs are $Q_C(p)$, for $p\notin C$, 
 varies with continuity. This proves the continuity of $Q_C$ at the points $p\notin C$. 
 Now consider $p\in C$ and $p_n\to p$ and assume without loss of generality that $p_n\in U^+$. 
 We have that $T(p_n)\to-\infty$. Take $q_n\in Q_C(p_n)$ with $q_n\to q$. 
 By the continuity of $T$ we have that $q\in C$. This proves the upper semicontinuity of $Q_C$ at $p\in C$. 
 Therefore $Q_C$ is a decomposition. 
 
 With the previous notation, we have that $Q_C(p_n)\to C\cap\partial U^+$. 
 Therefore, $Q_C$ is continuous at points in $C$ if and only if $C\cap\partial U^+=C\cap\partial U^-$.
 In order to conclude that the quotient is an arc note that except $L^\pm$, every plaque separates the rectangle $I\times J$. 
 This also implies that $Q_C$ is codendritic and a cw-decomposition.
 \end{proof}

\begin{exa}
\label{exaCocArc}
In the square $[-1,1]\times [-1,1]$ consider the graph like continuum
\[
 C=([-1,1]\times\{0\})\cup (\{0\}\times [0,1/2]).
\]
The decomposition $Q_C$ induced by $C$ is shown in Figure~\ref{figCocArc}. 
We see that it can be continuous at some $x$ while not being continuous at some $y\in Q_C(x)$.
\begin{figure}[ht]
 \center
  \includegraphics[scale=.65]{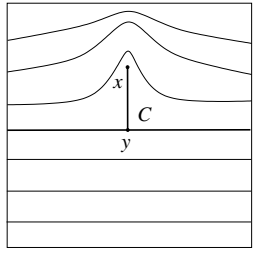}
  \caption{The decomposition is continuous at $x$ and discontinuous at $y$ (approximate $y$ from below).}
  \label{figCocArc}
 \end{figure}
\end{exa}

\begin{exa}
\label{exaDecSen}
In the square $[-2,2]\times [-2,2]$ consider the graph like continuum
\[
 C=(\{0\}\times[-1,1])\cup\{(x,\sin(1/x)):0<|x|\leq 2\}.
\]
We have that $Q_C$ is a continuous decomposition. 
It is illustrated in Figure~\ref{figDecSen}.
\begin{figure}[ht]
 \center
  \includegraphics[scale=.7]{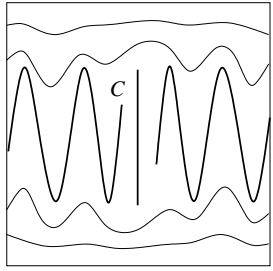}
  \caption{A continuous decomposition with a plaque that is not arc-connected.}
  \label{figDecSen}
 \end{figure}
\end{exa}

\begin{exa}
 Consider $K\subset [0,1]$ a Cantor set and define the graph like continuum
 $$C=(K\times[1/3,2/3])\cup([0,1]\times\{1/3\}).$$ 
Consider the decomposition $Q_C$ in the square $[0,1]\times[0,1]$.  
See Figure~\ref{figCantorCuadrado}.
If we think of $Q_C$ as a cw-foliation of this square and we consider 
a complete basis of small open sets then we see that $C$ is a leaf 
that is not a countable union of plaques. 
\begin{figure}[ht]
 \center
  \includegraphics[scale=.9]{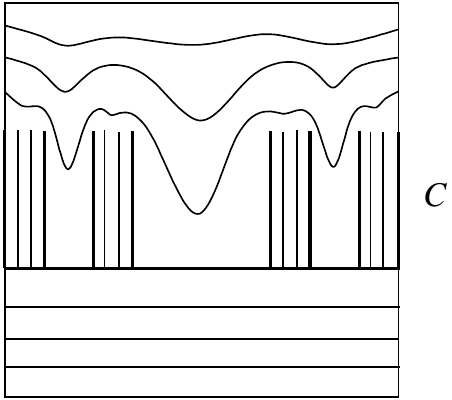}
  \caption{The (non-locally connected) leaf $C$ is not a countable union of small plaques.}
  \label{figCantorCuadrado}
\end{figure}
\end{exa}

\subsection{The foliated box}
\label{secFolBox}

The purpose of the present section is to give a characterization of the standard foliated two-dimensional box (a product structure of a rectangle).

\begin{prop}
\label{propContDend2}
Let $Q\colon D\to\continua(D)$ be a cw-decomposition. 
If $P$ is an arc-connected plaque that separates $D$ and $Q$ is continuous at every $x\in P$ 
then $P$ is an arc and $P\cap\partial D$ has exactly two points.
\end{prop}

\begin{proof}
By Janisewski's Theorem \ref{teoJanis}\footnote{In this case the Jordan's closed curve Theorem could also be applied.}, we have that $P\cap\partial D$ is disconnected because $P$ is a continuum separating $D$. 
Take $a,b\in\partial D$ in different components of $P\cap \partial D$. 
Since $P$ is arc-connected there is an arc $\gamma\subset P$ from $a$ to $b$. 
We can assume that $\gamma\cap\partial D=\{a,b\}$.

We will show that $\gamma= P$. 
Arguing by contradiction assume that there is $z\in P\setminus\gamma$.
Denote by $D_1,D_2$ the components of $D\setminus\gamma$ and suppose that $z\in D_1$.
Take $y\in\gamma$ and $x_n\in D_2$ with $x_n\to y$. 
Since $Q$ is continuous at every point of $P$, it is continuous at $y$. 
But $z$ is not in the limit of $Q(y_n)$. 
This contradiction proves that $P=\gamma$.
\end{proof}

\begin{rmk} 
We wish to remark the necessity of the hypothesis in Proposition \ref{propContDend2}.
The arc connection is needed by Example \ref{exaDecSen}, where $C$ is separating and continuous but not arc-connected. 
In Example \ref{exaCocArc} we see that $Q$ must be continuous at every point of $C$ in order to conclude that it is an arc, even being arc-connected and separating.
We need to assume that $C$ is separating in order to conclude that it is an arc. An example is illustrated in Figure~\ref{figEndDend}.
\end{rmk}
\begin{figure}[ht]
\center
\includegraphics[scale=.5]{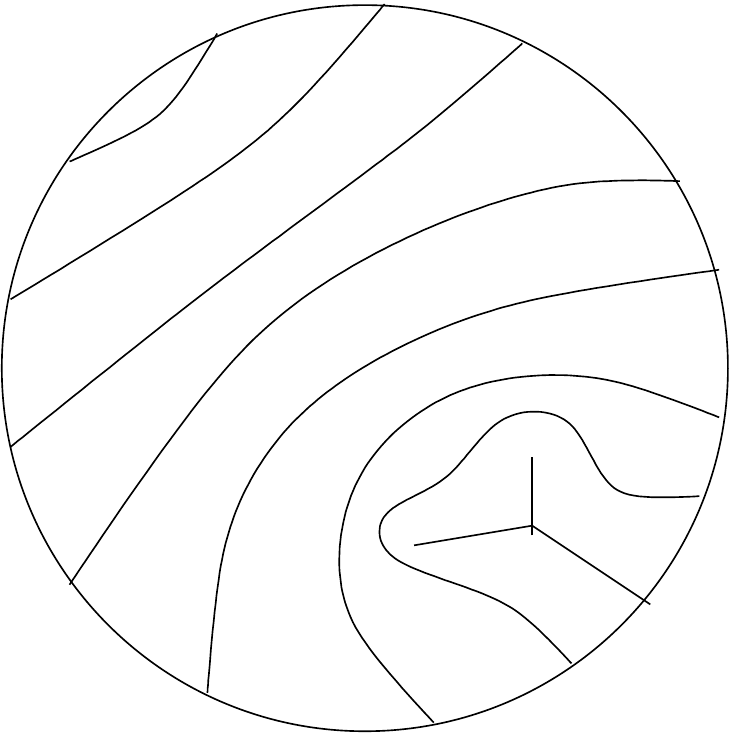}
\caption{A non-separating plaque not being an arc in a continuous arc-connected decomposition of a disc.}
\label{figEndDend}
\end{figure}

\begin{prop}
\label{propContDend}
If $Q\colon D\to\continua(D)$ is a continuous arc-connected cw-decomposition 
then $D/Q$ is an arc. 
\end{prop}

\begin{proof}
By definition of cw-decomposition we know that $D/Q$ is a dendrite.
In order to prove that it is an arc it is sufficient to prove that 
if $Q(x)$ \azul{separates} $D/Q$ then it separates in two components and this is a direct consequence of Proposition \ref{propContDend2}.
\end{proof}

\begin{exa} Proposition \ref{propContDend} is not true if we do not assume the arc-connection of each $Q(x)$. 
 An example can be constructed based on Wada's lakes. 
 See Figure~\ref{figWada}.
 A similar decomposition was previously considered in \cite{ChaCha}*{Remark 4.11}. 
 It is a continuous decomposition $Q$ of a disc $D$ such that $Q/D$ is a triod.
\begin{figure}[ht]
\center
\includegraphics[scale=.3]{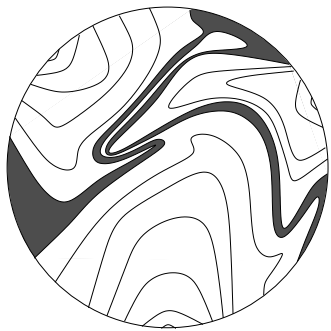}
\caption{Wada's lakes}
\label{figWada}
\end{figure}
\end{exa}

For the following proof we recall that a \emph{size function} is a continuous function $\mu\colon \continua(D)\to\R$ 
such that: $\mu(C)\geq 0$ for all $C\in\continua(D)$ with equality if and only if $C$ is a singleton; 
if $C\subset C'$ and $C\neq C'$ then $\mu(C)<\mu(C')$. 
Size functions exist \cite{Nadler} and in fact were introduced by Whitney in \cite{Whitney33} 
for the study of regular families of curves (very similar to our case).
The following result characterizes standard foliated boxes 
(\azul{or product structure as defined in \S \ref{secDecomp}}).
Its proof uses techniques from \cite{Whitney33}*{Theorem 17A}.

\begin{thm}
\label{thmCajaCaract}
If $Q\colon D\to\continua(D)$ is a $\continua$-smooth, continuous
cw-decomposition 
with two non-trivial plaques contained in $\partial D$
then it is a product structure.
\end{thm}

\begin{proof}
\azul{In Proposition \ref{propHuacImpDend} we proved that 
 every $\continua$-smooth decomposition is dendritic (in particular, arc-connected). 
Then, we can apply Proposition \ref{propContDend} to conclude that the quotient $D/Q$ is an arc.} 
Therefore, it has two non-separating plaques.
Let $E_1, E_2\subset\partial D$ be the non-separating plaques of $Q$. 
Take two arcs $A_1,A_2\subset \partial D$ intersecting $E_1$ and $E_2$ in one point to each one, as shown in Figure~\ref{figArcosAyE}.
\begin{figure}[ht]
\center
\includegraphics[scale=.9]{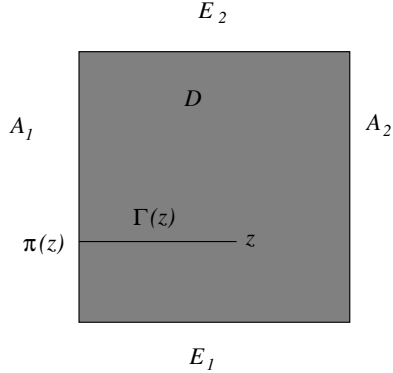}
\caption{The construction of suitable coordinates}
\label{figArcosAyE}
\end{figure}
Since $D/Q$ is an arc, $Q(z)$ separates $E_1$ from $E_2$ for all $z\in D$.
Then $Q(z)\cap A_i\neq\emptyset$ for all $z\in D$ and $i=1,2$.
By Proposition \ref{propContDend2} we have that $Q(z)\cap A_i$ is a singleton.
Define $\pi\colon D\to A_1$ by $\{\pi(z)\}=A_1\cap Q(z)$. 
Proposition \ref{propContDend2} also gives us that 
each $Q(z)$ is an arc. 
Define $\Gamma\colon D\to \continua(D)$ such that $\Gamma(z)\subset Q(z)$ is the arc from $z$ to $\pi(z)$. 

Let $\mu\colon \continua(D)\to\R$ be a size function 
and define $h\colon D\to \azul{A_1}\times [0,1]$ by 
\begin{equation}
 \label{ecuDefConj}
h(z)=(\pi(z),\mu(\Gamma(z))/\mu(Q(z))). 
\end{equation}
Since $Q$ has not trivial plaques we have that $\mu(Q(z))\neq 0$ for all $z\in D$. 
Let us show that $h$ is continuous. 
As $Q(z)\cap A_1$ is a singleton for all $z\in D$ we have that $\pi$ is continuous.
Given that $Q$ is $\continua$-smooth, we conclude that $\Gamma$ is continuous.
Since $\mu$ is continuous we conclude that $h$ is continuous. 

To prove that $h$ is injective suppose that $h(z)=h(z')$. 
Then $\pi(z)=\pi(z')$ and $\mu(\Gamma(z))/\mu(Q(z))=\mu(\Gamma(z'))/\mu(Q(z'))$. 
Therefore $z,z'$ are in the same plaque and $\mu(Q(z))=\mu(Q(z'))$, 
which implies that $\mu(\Gamma(z))=\mu(\Gamma(z'))$. 
Since every plaque is an arc we conclude that $z=z'$ and the injectivity of $h$. 
Similar arguments prove the surjectivity of $h$
and we conclude that $h$ is a homeomorphism.
Denote by $\tilde Q$ the decomposition of the rectangle $A_1\times [0,1]$ in horizontal lines $\{a\}\times [0,1]$ for $a\in A$.
Since it holds that $h(Q(z))=\tilde Q(h(z))$ 
the proof ends.
\end{proof}

It is natural to ask if a 
continuous arc-connected cw-decomposition $Q\colon D\to\continua(D)$
must be a product structure. In the next example we show that this is not always the case.

\begin{exa} 
\label{exaCuasiBox}
The cw-decomposition illustrated in Figure~\ref{figX3} 
is not a product structure because it is not $\continua$-smooth.
\begin{figure}[ht]
\center
\includegraphics[scale=.7]{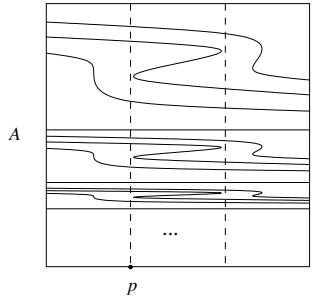}
\caption{Every plaque is an arc from the left side to the right side but $Q$ is not $\continua$-smooth. 
Also, the quotient space is an arc.}
\label{figX3}
\end{figure}
\end{exa}

%
%

\subsection{Smooth dendritations}
Let $\W$ be a cw-foliation on a closed surface $S$ with atlas $\A=\{Q_{\clos U}:U\in\U\}$. 

\begin{df}
We say that cw-foliation $\W$ is \emph{$\continua$-smooth} if 
each $Q_{\clos U}$ is a $\continua$-smooth decomposition.
\end{df}
By Proposition \ref{propRestSmooth} we know that monotone restrictions of a $\continua$-smooth cw-decomposition are $\continua$-smooth. 
Then, a $\continua$-smooth cw-foliations is defined ever a complete basis.

\begin{thm}
\label{thmSmoothFol}
Every $\continua$-smooth cw-foliation of a closed surface $S$ is 
a standard foliation with a finite set of prong-singularities and no 1-prongs.
\end{thm}

\begin{proof}
Consider a cw-atlas $\A=\{Q_{\clos U}:U\in\U\}$ of the 
$\continua$-smooth cw-foliation $\W$. 
By Proposition \ref{propHuacImpDend} we have that $\W$ is a dendritation.
Take $U\in\U$ such that $D=\clos{U}$ is a disc and define $Q=Q_D$.

In this paragraph we will show that 
if $z$ is an end point of the dendrite $Q(z)$ then $z\in\partial D$.
Arguing by contradiction suppose that 
there is $z\in D\setminus\partial D$ that is an end point of its plaque. 
Take $z_n\in D\setminus Q(z)$ such that $z_n\to z$.
Since $Q$ is codendritic, for each $n\geq 1$ there is a plaque $P_n$ separating $Q(z)$ from $z_n$.
Consider an arc $A_n\subset P_n$ separating $Q(z)$ from $z_n$ such that 
$A_n\cap\partial D=\{x_n,y_n\}$.
Consider a Riemannian metric on $D$ and for each $n\geq 1$ a geodesic arc $\gamma_n$ from $z$ to $z_n$. 
Since $A_n$ separates $z$ from $z_n$, we can take $z'_n\in \gamma_n\cap A_n$. 
Since $\gamma_n$ is a geodesic arc we have that $z'_n\to z$. 
Taking subsequences we can assume that $x_n\to x$ and $y_n\to y$ with $x,y\in Q(z)\cap\partial D$. 
Given that $Q$ is $\continua$-smooth we have that $A_n\to A$ where $A$ is an arc from $x$ to $y$.
Since $z'_n\in A_n$ and $z_n\to z$ we have that $z\in A$. 
The upper semicontinuity of $Q$ gives us that $A\subset Q(z)$.
Since $z\notin\partial D$ we have that $z$ is not an end point of $Q(z)$. 
This contradiction proves that every end point of each plaque is in the boundary of $D$.

Now we will show that
 for all compact set $K\subset\interior(D)$ 
 the set of points $x\in K$ such that $x$ is a ramification point of $Q(x)$ is finite.
The result of the previous paragraph implies that for all $x\in D$
every component of $Q(x)\setminus\{x\}$ meets $\partial D$. 
Now suppose that there is a sequence $y_n\to y$ such that 
$y_n\in K$ and each $y_n$ is a ramification point of $Q(y_n)$. 
Denote by $A_n=Q(y_n)\cap\partial D$. 
We have that each $A_n$ has at least three points. 
We can take $a_n,b_n\in A_n$ such that $\dist(a_n,b_n)\to 0$ 
and $y_n\in Q(a_n,b_n)$\footnote{For two points $a,b$ in a common plaque, $Q(a,b)$ is the arc in the plaque from $a$ to $b$.}. 
This contradicts that $Q$ is $\continua$-smooth 
and proves that the set of ramification points can only accumulate on $\partial D$.
Then, there is a finite number of ramification points in the whole surface.

Take $x\in S$ and a small disc $D$ around $x$. 
Assume that in $D\setminus\{x\}$ there are no ramification points. 
We know that $Q_D(x)$ separates $D$ in at least two components. 
Some of these components may be far from $x$ and only a finite number of them contain $x$ in its boundary, see Figure~\ref{figSmoothChart}. 
Denote by $D_1,\dots,D_n$ the closures of such components. 

\begin{figure}[ht]
\center
  \includegraphics[scale=1]{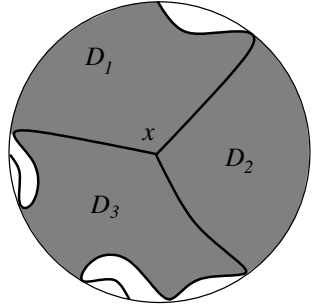}
  \caption{A finite number of components of $D\setminus Q(x)$ are discs with $x$ in its boundary.}
  \label{figSmoothChart}
\end{figure}
Since the only possible ramification point in $D$ is $x$, we know that in each $D_i$ every plaque is an arc.
As the end points of the plaques are in the boundary, each plaque is an arc starting and ending at the boundary. 
Given $y$ in the interior of $D_i$, the arcs $Q_{D_i}(x)$ and $Q_{D_i}(y)$ determine two arcs $A, B$ contained in the boundary of $D_i$ 
as shown in Figure~\ref{figSmoothChart2}.
\begin{figure}[ht]
\center
  \includegraphics[scale=1]{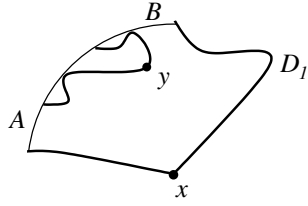}
  \caption{The arcs $A$ and $B$ cuts $Q_{D_1}(x)$ and $Q_{D_1}(y)$ in one point.}
  \label{figSmoothChart2}
\end{figure}

Now we will prove that there is a neighborhood $V$ of 
$x$ in $D_i$ such that every $z\in V$ 
separates $Q_{D_i}(z)$ in two arcs $\gamma_A(z)$ and $\gamma_B(z)$ such that 
$\gamma_A(z)$ cuts $A$ and is disjoint from $B$ and 
$\gamma_B(z)$ cuts $B$ and is disjoint from $A$.
Arguing by contradiction assume first that there is $z_n\to x$ such that $Q_{D_i}(z_n)\cap B=\emptyset$. 
We can take two points $a_n,b_n\in A\cap Q_{D_i}(z_n)$ such that $\dist(a_n,b_n)\to 0$. 
This contradicts that $Q_{D_i}$ is $\continua$-smooth. 
Therefore, if $z$ is close to $x$ then $Q_{D_i}(z)$ cuts $A$ and $B$.
A similar argument proves that if $z$ is sufficiently close to $x$ then 
one component of $Q_{D_i}(z)\setminus\{z\}$ cuts $A$ and not $B$ while the other component 
cuts $B$ and not $A$.

Define $K$ as the union of all the arcs starting and ending in $A$ and contained in a plaque of $Q_{D_i}$. 
Considering that the arcs defining $K$ may be trivial, we have that $A\subset K$. 
In this way $K$ is a continuum.
For $u,v\in D_i$ in a common plaque define $\arc(u,v)$ as the arc from $u$ to $v$ inside the plaque. 
For $u\in V$ define $\pi(u)$ as the first point of $\gamma_A(u)$ in $A$. 
Define $\Gamma\colon V\to \continua(D_i)$ as 
\[
 \Gamma(u)=\arc(u,\pi(u))\cup K.
\]

In this paragraph we show that $\Gamma$ is continuous. 
Take $u\in V$ and $u_n\to u$. 
We can assume that $\Gamma(u_n)$ converges and $\pi(u_n)\to v\in \gamma_A(u)\cap A$. 
Since $Q$ is $\continua$-smooth we have that $\arc(u_n,\pi(u_n))\to \arc(u,v)$. 
It is clear that $\arc(u,\pi(u))\subset\arc(u,v)$. Consequently, $\Gamma(u)\subset \lim\Gamma(u_n)$. 
Since $\arc(\pi(u),v)\subset K$ and 
$$\arc(u,v)=\arc(u,\pi(u))\cup\arc(\pi(u),v)$$ we have that 
$\lim\Gamma(u_n)\subset\Gamma(u)$. This proves the continuity of $\Gamma$.

Consider a size function $\mu\colon \continua(D_i)\to\R$ 
and define $\varphi\colon V\to\R$ as $\varphi(u)=\mu(\Gamma(u))$. 
The function $\varphi$ is continuous and increases along a plaque from $A$ to $B$. 
The level sets of $\varphi$ and the plaques give coordinates and define a product structure 
in a neighborhood of $x$ in $D_i$ (contained in $V$). 

Arguing in the same way in the sectors $D_1,\dots,D_n$ an $n$-prong structure is defined around $x$. 
As we proved, $n\neq 1$ and the set of $n$-prongs with $n\geq 3$ is finite. This finishes the proof.
\end{proof}

\subsection{Continuous dendritations}

%
%


Let $\A=\{Q_{\clos U}:U\in\U\}$ be an atlas of a dendritation $\W$ over the complete basis $\U$ of a surface $S$. 

\begin{df} 
\label{dfContDend}
 We say that a dendritation $\W$ is \emph{continuous} 
 if for all $x\in U\in\U$ there is a closed disc $D\subset U$ 
 such that $x$ is in the interior of $D$ and $Q_D$ is continuous.
\end{df}

This definition of continuity is stronger
than the one given in \S \ref{secConAtl}
because we require the continuity on discs instead of arbitrary neighborhoods.

Denote by $Q_1$ the decomposition of $D=[-1,1]\times [-1,1]$ in horizontal plaques (a foliated box). 
Consider the equivalence relation on $D$ generated by $(x,1)\sim(-x,1)$ for all $x\in [-1,1]$. 
Denote by $Q_2$ the induced decomposition on the quotient space $\tilde D=D/\sim$. 
The point $\{(0,1)\}$ in $\tilde D$ is called \emph{1-prong} of $Q_2$.

\begin{thm}
\label{teoClassRegCwFol}
Every continuous dendritation is a foliation possibly with a finite number of 1-prongs.
\end{thm}

\begin{proof}
By Proposition \ref{propContDend2} there are no ramification points, that is, 
every plaque is an arc. 
Take $x\in S$ and a disc $D$ around $x$ with $Q_D$ continuous. 
From Proposition \ref{propContDend} we know that $D/Q$ is an arc. 
Then there are exactly two non-separating plaques that will be called $E_1, E_2$.
From Proposition \ref{propContDend2}, every separating plaque cuts the boundary of $D$ 
at exactly two points.
Reducing $D$ with two separating plaques (if needed) we can assume that $E_1$ and $E_2$ are non-trivial (are not singletons).

In this paragraph we will show that if $E_1,E_2\subset \partial D$ then $Q$ is a product structure.
For this purpose we will show that $Q$ is $\continua$-smooth and apply Theorem \ref{thmCajaCaract}.
Arguing by contradiction assume that $Q(x_n,y_n)$ does not converge to $Q(x,y)$ 
assuming that $x_n\to x$, $y_n\to y$ and $y_n\in Q(x_n)$. 
Taking a subsequence we can assume that $Q(x_n,y_n)\to C$ for a continuum $C\subset D$. 
Since $Q(x_n,y_n)\subset Q(x_n)$ and $Q$ is continuous we have that 
$C\subset Q(x)$. 
Therefore, $C$ is an arc because $Q(x)$ is an arc. 
Moreover, $x,y\in C$.
Since $C\neq Q(x,y)$, we have that $y$ (or $x$) is an interior point of the arc $C$.
Let $z\in C$ be the end point of $C$ that is not in $Q(x,y)$. 
Consider $z_n\in Q(x_n,y_n)$, an interior point, 
such that $z_n\to z$. 
Since the dendritation is continuous
there is a disc $D'$, a neighborhood of $z$, 
such that $Q'=Q|^m_{D'}$ is continuous. 
We can assume that $D'$ is so small that $y,y_n\notin D'$ for all $n\geq 1$. 
Therefore, $Q'(z_n)$ is a sequence of arcs separating $D'$. 
We can find $n_1,n_2,n_3$ such that
$z$ and $Q'(z_{n_i})$ are in the same component of $D'\setminus Q'(z_{n_j})$ if $i\neq j$, $i,j\in\{1,2,3\}$. 
Then $Q'$ has at least three non-separating plaques.
This contradicts Proposition \ref{propContDend} in the disc $D'$
and proves that $Q$ is $\continua$-smooth.
Since $E_1,E_2$ are non-trivial we can apply 
Theorem \ref{thmCajaCaract} to conclude that $Q$ is a product structure.

If $x\in E_1$ (i.e., $E_1$ is not contained in the boundary of $D$) 
we can cut $D$ along $E_1$ and reduce the proof to the previous case ($E_1,E_2\subset \partial D$). 
In this case we obtain that $Q$ is a 1-prong decomposition around $x$. 
This also proves that the number of 1-prongs is finite.
\end{proof}

In the following result the word foliation means a one-dimensional $C^0$ foliation in the standard sense (without singular points).

\begin{cor}
\label{corCharFol}
 A dendritation is a foliation if and only if it is continuous and $\continua$-smooth.
\end{cor}

\begin{proof}
It follows by Theorems \ref{thmSmoothFol} and \ref{teoClassRegCwFol}.
\end{proof}

\subsection{Cw$_{\bf F}$-expansivity and dendritations}


\azul{If two arcs $\alpha,\beta$ in a surface meet at $x$ we say that they are \emph{topologically transverse} at $x$
if the components of $\alpha\setminus\beta$ are in 
different components of $D\setminus\beta$, where $D$ is a disc separated by $\beta$ and containing $\alpha$.
See Figure~\ref{figDenSU}.}
\begin{figure}[ht]
\center
\includegraphics[scale=.9]{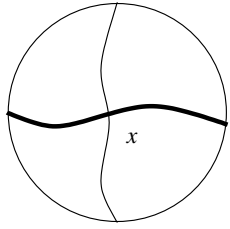}
\caption{Two arcs topologically transverse at the point $x$.}
\label{figDenSU}
\end{figure}

The following result gives some ideas of how the stable and the unstable cw-foliations of a \cwfin-expansive 
homeomorphism are distributed in the surface.

\begin{thm}
\label{thmCwfSuperficies}
 If $f$ is a \cwfin-expansive homeomorphism of a compact surface $S$ then: 
 \begin{enumerate}
  \item $\W^s$ and $\W^u$ are $f$-invariant dendritations and no leaf is a Peano continuum, 
  \item there is a residual set $G\subset S$ such that $\W^s(x)$ and $\W^u(x)$ are non-compact 
  one-dimensional manifolds for all $x\in G$, 
  \item there is a subset $H\subset G$ that is dense in $S$ such that $\W^s(x)$ and $\W^u(x)$ are topologically transverse at $x$ for all $x\in H$.
 \end{enumerate}
\end{thm}


\begin{proof}
By Theorem \ref{thmCovCwExpFol} we know that $\W^s$ and $\W^u$ are $f$-invariant cw-foliations 
without plaque-compact leaves. 
We start proving the local connection of stable and unstable plaques to conclude that they are dendritations.
Arguing by contradiction, let $P\subset S$ be a non-locally connected stable continuum. 
By Theorem \ref{teoLocConConvCont}
we can consider $P_n$ a sequence of subcontinua of $P$ such that 
$P_i\cap P_j=\emptyset$ if $i\neq j$ and $\dist(P_n,P_*)\to 0$ for some non-trivial continuum $P_*\subset P$. 
Take $x\in P_*$ and a small disc $D$ around $x$ such that every $P_n$ separates $D$ 
for every $n\geq n_0$, for some $n_0$. 
We can assume that $D=B_r(x)$ for some $r>0$.
Denote by $Q^s$ and $Q^u$ the stable and the unstable cw-decompositions of $D$, respectively. 
Consider $x_n\in P_n\cap D$ such that $x_n\to x$ and each $Q^s(x_n)$ separates $D$. 
Taking a subsequence if needed we can assume that $Q^s(x_n)$ converges to a stable continuum $C\subset Q^s(x)$. 
Moreover, we can assume that the sequence of plaques $Q^s(x_n)$ is monotonous (i.e., if $i<j<k$ then $Q^s(x_j)$ separates 
$Q^s(x_i)$ from $Q^s(x_k)$).
For the value of $r>0$ fixed above ($D=B_r(x)$) 
take $\epsilon>0$ given by Theorem \ref{teoCap}.
Denote by $G_n$ the component of $D$ between $C$ and $Q^s(x_n)$. 
Since $Q^s(x_n)\to C$ in the Hausdorff metric, there is $n_1>0$ such that 
$Q^s(x_n)\subset B_\epsilon(C)$ for all $n\geq n_1$. 
Consequently, $G_n\subset B_\epsilon(C)$ for all $n\geq n_1$. 
If $\epsilon$ is small, we can take $y_n\in \partial B_{r/2}(x)\cap Q^s(x_n)$ 
and an arc $\gamma_n\subset B_{r/2}(x)$ from $x_n$ to $y_n$.
\azul{Taking a subarc of $\gamma_n$ we can assume that $\gamma_n\subset \clos{G_n}\cap B_{r/2}(x)$ 
meeting $C$ and $Q^s(x_n)$.}
This proves that $(C,G_n,Q^s(x_n))$ is an $(r,\epsilon,x)$-capacitor for all $n\geq n_1$. 
Applying Theorem \ref{teoCap} we can take an unstable plaque $Q^u(z)$ cutting $C$ and $Q^s(x_n)$. 
This unstable plaque cuts $Q^s(x_m)$ for all $m\geq n$. 
Since the plaques $Q^s(x_n)$ are contained in the stable plaque $P$ (considered at the beginning of the proof) 
we have a contradiction with the \cwfin-expansivity of $f$.
This proves that $\W^s$ and $\W^u$ are dendritations. 

Since \cwfin-expansivity implies cw-expansivity we can now apply Theorem \ref{thmCwExpGenDen} 
to obtain the residual set $G$.

By Theorem \ref{thmCovCwExpFol} no stable leaf is plaque compact and 
by Theorem \ref{thmPCompIffPeano}, a leaf is plaque compact if and only if it is a Peano continuum. 
Therefore, no stable leaf is a Peano continuum.

Let $D\subset S$ be a small disc and take an interior point $x\in D$. 
We can take a sequence $x_n\in G$ such that $x_n\to x$. 
Consider a stable arc $\gamma^s_n$ from $x_n$ to $y_n\in\partial D$. 
We can assume that $y_n\to y$ and $\gamma_n\to C$ where $C$ is a dendrite containing $x$ and $y$. 
Taking a subsequence we can suppose that $y_n$ is monotone in $\partial D$. 
Consider four consecutive arcs $\gamma^s_{n+1},\dots,\gamma^s_{n+4}$. 
As before, for $n$ large, we can apply Theorem \ref{teoCap} to the capacitor determined by 
$\gamma^s_{n+1}$ and $\gamma^s_{n+4}$ to find a point $p\in G$ between $\gamma^s_{n+2}$ and $\gamma^s_{n+3}$ 
whose unstable arc cuts $\gamma^s_{n+1}$ or $\gamma^s_{n+4}$. 
Suppose that it cuts $\gamma^s_{n+1}$. 
Then, it cuts $\gamma^s_{n+2}$. 
\azul{By the \cwfin-expansivity of $f$,
the unstable arc of $p$ intersects $\gamma^s_{n+2}$ in finitely many points.
By the construction, at least one of this cuts is topologically transverse.} 
\end{proof}

\subsection{Cw1-expansivity}
\label{secPairDec}

In this section we show that every cw1-expansive homeomorphism of a compact surface is expansive.

\begin{lem}
\label{lemArcoSepara}
Let $Q^1,Q^2$ be two cw-decompositions of $D$ such that 
$Q^1(x)\cap Q^2(x)=\{x\}$ for all $x\in D$, 
$\partial D$ is the union of two arcs $\alpha$ and $\beta$ with extreme points $p,q$ 
and $\alpha\subset Q^1(a)$ for some $a\in \partial D$\azul{. Then}
for all $x\in \alpha$, $x\notin\{p,q\}$, 
it holds that $Q^2(x)\cap \beta\neq\emptyset$.
\end{lem}

\begin{proof}
Suppose that for some $x\in \alpha$, $x\neq p,q$, we have that 
$Q^2(x)\cap\beta=\emptyset$. 
We know that $D/Q^2$ is a dendrite.
Since $Q^2(x)\cap Q^1(x)=\{x\}$ we have that $Q^2(x)$ is an end of $D/Q^2$.
In the quotient dendrite we can find a point say $Q^2(y)$ that separates $D/Q^2$ 
and is arbitrarily close to $Q^2(x)$. 
Then $Q^2(y)$ cuts $Q^1(a)$ in at least two points, contradicting our hypothesis.
\end{proof}

\begin{df}
Given two cw-decompositions $Q^1,Q^2$ of a disc $D$ we say that $\partial D$ is a $(Q^1,Q^2)$-\emph{rectangle}  
if there are $x_1,y_1,x_2,y_2\in\partial D$
such that $\partial D$ is contained in the ordered union $Q^1(x_1)\cup Q^2(y_1)\cup Q^1(x_2)\cup Q^2(y_2)$.
\end{df}

\begin{prop}
\label{propCajaProd}
 Let $Q^1,Q^2$ be two cw-decompositions of a disc $D$ such that $Q^1(x)\cap Q^2(x)=\{x\}$ for all $x\in D$ and 
 $\partial D$ is a $(Q^1,Q^2)$-rectangle\azul{. Then} there is 
 a homeomorphism $h\colon D\to [0,1]\times[0,1]$ sending the plaques of $Q^1$ and $Q^2$ \azul{onto} horizontal and vertical 
 segments, respectively.
\end{prop}

\begin{proof}
We know that $D/Q^1$ and $D/Q^2$ dendrites. 
Let us show that they are arcs. 
For this purpose we will show that they have no ramification point. 
Suppose by contradiction that $Q^1(x)$ (similarly for $Q^2$) is a ramification of the dendrite $D/Q^1$ for some $x\in D$.
Then $Q^1(x)$ separates $D$ into at least three components. 
Consequently, $Q^1(x)\cap \partial D$ has at least three components.
Since $\partial D$ is a $(Q^1,Q^2)$-rectangle, we have that $Q^1(x)$ cuts in at least two points 
to a plaque of $Q^2$ in the boundary of $D$. This contradiction proves that $D/Q^1$ is an arc. 

Let $h\colon D\to [D/Q^1]\times [D/Q^2]$ as $h(x)=(Q^1(x),Q^2(x))$.
We have that $h$ is continuous and injective. 
In order to conclude that it is a homeomorphism it is sufficient to prove that 
it is surjective. 
The surjectivity in this case means that $Q^1(x)$ cuts $Q^2(y)$ for all $x,y\in D$, 
and this follows by Lemma \ref{lemArcoSepara}.
\end{proof}

\begin{prop}
\label{propCw1Gen}
If $Q^1,Q^2$ is a pair of dendritic cw-decompositions of $D$ and 
$Q^1(x)\cap Q^2(x)=\{x\}$ for all $x\in D$ then $(Q^1,Q^2)$ is a generating pair.
\end{prop}

\begin{proof}
Fix a point $x\in D$. 
\azul{Taking a small} subdisc around $x$ we can assume that the intersection of each component of 
$Q^i(x)\setminus\{x\}$ with $\partial D$ is a \azul{single point}, $i=1,2$.
In this way $D\setminus Q^i(x)$ has a finite number of components.
By Lemma \ref{lemArcoSepara} we know that each component of $D\setminus Q^1(x)$ is separated by a component of $Q^2(x)\setminus\{x\}$. 
Therefore, the components of $Q^1(x)\setminus\{x\}$ and $Q^2(x)\setminus\{x\}$ are alternated in the disc. 
Denote by $D'$ a disc contained in $D$ bounded by an arc $\alpha^1\subset Q^1(x)$, an arc $\alpha^2\subset Q^2(x)\setminus\{x\}$ and 
an arc $\gamma\subset\partial D$. 
Assume that $\alpha^1$ and $\alpha^2$ are consecutive. 

By Lemma \ref{lemArcoSepara} we have that for each interior point 
$y\in \alpha^1$ there is an arc $\beta^2_y$ (of $Q^2$) from $y$ to $\gamma$. 
Also, for each interior point 
$z\in \alpha^2$ there is a stable arc $\beta^u_z$ from $z$ to $\gamma$. 
Since $\alpha^s$ and $\alpha^u$ are consecutive, if $z$ and $y$ are close to $x$ 
we have that $\beta^s_y$ cuts $\beta^u_z$. 
In this way we obtain a $(Q^1,Q^2)$-rectangle and we can apply Proposition \ref{propCajaProd} to prove that both decompositions 
generate a neighborhood of $x$ in $D'$.

Repeating this argument in each sector we conclude 
that $Q^1$ and $Q^2$ generate a (full) neighborhood of $x$. 
\end{proof}

\begin{teo}
\label{teoEquivPANo1prong}
Every cw1-expansive homeomorphism of a compact surface is expansive.
\end{teo}

\begin{proof}
If $f$ is cw1-expansive then by Proposition \ref{propCw1Gen} we have that $\W^s$ and $\W^u$ generate.
Then, the result follows by Theorem \ref{thmCw1exp}. 
\end{proof}

\azul{As usual, $\Omega(f)$ denotes the non-wandering set of $f$.}

\begin{cor}
For a homeomorphism $f\colon S\to S$ of a compact surface the following statements are equivalent: 
 \begin{enumerate}
  \item $f$ is cw1-expansive,
  \item $f$ is expansive,
  \item $f$ is 2-expansive and $\Omega(f)=S$,
  \item $f$ is conjugate to a pseudo-Anosov diffeomorphism without 1-prongs.
 \end{enumerate}
\end{cor}

\begin{proof}
We have that $2\to 1$ on arbitrary metric spaces. 
The equivalence of 2 and 4 was shown in \cites{Hi,L}. 
The equivalence of 2 and 3 follows by \cite{APV}.
We have that $1 \to 2$ by Theorem \ref{teoEquivPANo1prong}. 
\end{proof}

\noindent Departamento de Matemática y Estadística del Litoral, Universidad de la República,\\
Gral. Rivera 1350, Salto-Uruguay\\
E-mail: artigue@unorte.edu.uy

\begin{bibdiv}
\begin{biblist}

\bib{AHNO}{article}{
author={G. Acosta}, 
author={R. Hernández-Gutiérrez}, 
author={I. Naghmouchi}, 
author={P. Oprocha},
title={Periodic points and transitivity on dendrites},
journal={Ergodic Theory and Dynamical Systems},
year={2016}}

\bib{AH}{book}{
author={N. Aoki},
author={K. Hiraide},
title={Topological theory of dynamical systems},
publisher={North-Holland},
year={1994}}

%
%
%
\bib{ArKinExp}{article}{
author = {A. Artigue},
title = {Kinematic expansive flows},
journal = {Ergodic Theory and Dynamical Systems},
volume = {\azul{36}},
year = {2016},
pages = {\azul{390--421}}}
%
%

\bib{ArRobNexp}{article}{
author={A. Artigue},
title={Robustly N-expansive surface diffeomorphisms},
volume={36},
year={2016},
journal={Discrete and Continuous Dynamical Systems},
pages={2367--2376}}

\bib{ArAnomalous}{article}{
author={A. Artigue},
title={Anomalous cw-expansive surface homeomorphisms},
journal={Discrete and Continuous Dynamical Systems},
pages={\azul{3511--3518}},
volume={36},
year={2016}}

\bib{APV}{article}{
author={A. Artigue},
author={M.J. Pacífico},
author={J.L. Vieitez},
title={N-expansive homeomorphisms on surfaces},
year={to appear},
journal={Communications in Contemporary Mathematics}}

\bib{BM}{article}{
author={M. Barge},
author={B.F. Martensen},
title={\azul{Classification of expansive attractors on surfaces}},
journal={Ergodic Theory and Dynamical Systems},
volume={31},
year={2011}, 
pages={1619--1639}}

%
%
%
\bib{BW}{article}{
author={R. Bowen and P. Walters}, 
title={Expansive one-parameter flows}, 
journal={J. Diff. Eq.}, year={1972}, pages={180--193},
volume={12}}

\bib{BS}{book}{
author={M. Brin},
author={G. Stuck},
title={Introduction to Dynamical Systems},
publisher={Cambridge University Press},
year={2003}}

\bib{CaNe}{book}{
author={C. Camacho},
author={A.L. Neto},
title={Geometric Theory of Foliations},
publisher={Birkhäuser Boston, Inc.},
year={1985}}

\bib{CaCo}{book}{
author={A. Candel},
author={L. Conlon},
title={Foliations I},
publisher={Am. Math. Soc.},
year={2000},
series={Graduate Studies in Math.},
volume={23}}

\bib{CaPa}{article}{
title={Monotone quotients of surface diffeomorphisms},
author={A. de Carvalho},
author={M. Paternain},
journal={Math. Research Letters},
volume={10},
year={2003},
pages={603--619}}


\bib{ChaCha}{article}{
author={J.J. Charatonik},
author={W.J. Charatonik},
title={Connected subsets of dendrites and separators of the plane},
journal={Topology and its Applications},
volume={36}, 
year={1990}, 
pages={233--245}}

\bib{Da}{book}{
author={R.J. Daverman}, 
title={Decompositions of manifolds},
publisher={Academic Press. Inc},
year={1986},
series={Pure and applied mathematics}}
%

\bib{Engel}{book}{
title={General Topology},
author={R. Engelking},
publisher={Heldermann Verlag},
series={Sigma series in pure math.},
volume={6},
year={1989}}

%
%
\bib{FR}{article}{
author={J. Franks},
author={C. Robinson},
title={A quasi-Anosov diffeomorphism that is not Anosov},
journal={Trans. of the AMS},
volume={223},
year={1976},
pages={267--278}}

%
%
%

\bib{Hi89}{article}{
author={K. Hiraide},
title={Expansive homeomorphisms with the pseudo-orbit tracing property of $n$-tori},
journal={J. Math. Soc. Japan}, 
volume={41},
year={1989}, 
pages={357--389}}

\bib{Hi}{article}{
author={K. Hiraide},
title={Expansive homeomorphisms of compact surfaces are pseudo-Anosov},
journal={Osaka J. Math.}, 
volume={27},
year={1990}, 
pages={117--162}}

%
\bib{HPS}{book}{
author={M. Hirsch}, 
author={C. Pugh},
author={M. Shub},
title={Invariant manifolds},
series={Lecture notes in mathematics}, 
publisher={Springer-Verlag}, 
volume={583}, 
year={1977}}
%
\bib{HY}{book}{
author={J.G. Hocking},
author={G.S. Young},
title={Topology},
publisher={Addison-Wesley Publishing Company, Inc.},
year={1961}}
%
%
\bib{HW}{book}{
author={W. Hurewicz},
author={H. Wallman},
title={Dimension Theory},
publisher={Princeton University Press},
year={1948},
edition={Revised edition}}

\bib{IN}{book}{
author={A. Illanes},
author={A.B. Nadler},
title={Hyperspaces, Fundamentals and recent advances},
publisher={Marcel Dekker, Inc},
year={1999}}

\bib{Jo}{article}{
journal={Bull. Amer. Math. Soc.},
volume={68},
year={1962}, 
pages={117--119},
title={On the existence of a small connected open set with a connected boundary},
author={F.B. Jones}}

%
%
\bib{Ka93}{article}{
author={H. Kato},
title={Continuum-wise expansive homeomorphisms},
journal={Can. J. Math.},
volume={45},
number={3},
year={1993},
pages={576--598}}
%
\bib{Ka2}{article}{
author={H. Kato},
title={Concerning continuum-wise fully expansive homeomorphisms of continua},
journal={Topology and its Applications},
volume={53},
year={1993},
pages={239--258}}

\bib{KTT}{article}{
author={K. Kawamura},
author={H.M. Tuncali}, 
author={E.D. Tymchatyn},
journal={Houston Journal of Mathemtaics},
volume={21}, 
year={1995},
title={Expansive homeomorphisms on Peano curves},
pages={573--583}}


%
%
\bib{Ko84}{article}{
author={M. Komuro}, 
title={Expansive properties of Lorenz
attractors}, journal={The Theory of dynamical systems and its
applications to nonlinear problems}, year={1984}, place={Kyoto},
pages={4--26}, publisher={World Sci. Singapure}}
%
\bib{Kur1}{book}{
author={K. Kuratowski},
title={Introduction to set theory and topology},
publisher={Pergamon Press},
year={1961}}

\bib{Kur}{book}{
author={K. Kuratowski},
title={Topology},
volume={II},
publisher={Academic Press, New York and London},
year={1968}}

\bib{Lee}{book}{
author={J.M. Lee},
title={Introduction to Topological Manifolds},
series={Graduate texts in mathematics},
volume={202},
publisher={Springer-Verlag},
year={2000}}


%
%
\bib{L}{article}{
author={J. Lewowicz},
title={Expansive homeomorphisms of surfaces},
journal={Bol. Soc. Bras. Mat.},
volume={20},
pages={113-133},
year={1989}}
%
%
\bib{Ma75}{incollection}{
author={R. Mañé},
title={Expansive diffeomorphisms},
year={1975},
booktitle={Dynamical Systems - Warwick 1974},
volume={468},
series={Lecture Notes in Mathematics},
editor={Manning, Anthony},
publisher={Springer Berlin Heidelberg},
pages={162-174}}
 
\bib{Ma}{article}{
author={R. Mañé},
title={Expansive homeomorphisms and topological dimension},
journal={Trans. of the AMS}, 
volume={252}, 
pages={313--319}, 
year={1979}}
 

\bib{McM}{book}{
author={C.T. McMullen}, 
title={Complex dynamics and renormalization},
series={Annals of Mathematical Studies},
volume={135},
publisher={Princeton University Press, Princeton, NJ},
year={1994}}


\bib{Mo25}{article}{
author={R.L. Moore}, 
title={Concerning upper-semicontinuous collections of continua}, 
journal={Transactions of the AMS},
volume={4},
year={1925},
pages={416--428}}

\bib{Mo32}{book}{
author={R.L. Moore}, 
title={Foundations of point set theory}, 
publisher={Amer. Math. Soc. Colloquium Publications}, 
volume={13}, 
place={New York}, 
year={1932}}

\bib{Mo28}{article}{
author={R.L. Moore},
journal={Proc. Natl. Acad. Sci. USA},
year={1928},
volume={14},
pages={85--88},
title={Concerning Triods in the Plane and the Junction Points of Plane Continua}}

\bib{Mo12}{article}{
author={C.A. Morales},
title={A generalization of expansivity},
journal={Disc. and Cont. Dyn. Sys.},
volume={32},
year={2012}, 
pages={293--301}}
%
%
\bib{Nadler}{book}{
author={S. Nadler Jr.},
title={Hyperspaces of Sets},
publisher={Marcel Dekker Inc. New York and Basel},
year={1978}}

\bib{Nadler2}{book}{
author={S. Nadler Jr.}, 
title={Continuum Theory}, 
series={Pure and Applied Mathematics}, 
volume={158}, 
publisher={Marcel Dekker, New York},
year={1992}} 
 
\bib{PaVi}{article}{
author={M.J. Pacifico},
author={J.L. Vieitez},
title={\azul{Entropy expansiveness and domination for surface diffeomorphisms}},
journal={Rev. Mat. Complut.},
volume={21},
number={2},
year={2008},
pages={293--317}}

\bib{PaPuVi}{article}{
title={\azul{Robustly expansive homoclinic classes}},
author={M.J. Pacifico},
author={E.R. Pujals},
author={J.L. Vieitez},
journal={Ergodic Theory and Dynamical Systems},
year={2005}, 
volume={25}, 
pages={271--300}}


\bib{PX}{article}{
year={2014},
journal={Mathematische Zeitschrift},
volume={278},
title={A classification of minimal sets for surface homeomorphisms},
author={A. Passeggi},
author={J. Xavier},
pages={1153--1177}}

%
%

\bib{Pittman}{article}{
author={C.R. Pittman}, 
title={An elementary proof of the triod theorem}, 
journal={Proc. Amer. Math. Soc.},
volume={25},
year={1970}, 
pages={919}}


\bib{Ro}{article}{
author={J.H. Roberts}, 
title={There does not exist an upper semi-continuous decomposition of E into arcs},
journal={Duke Math. J.},
volume={2}, 
year={1936}, 
pages={10--19}}

\bib{RoSt}{article}{
title={Monotone Transformations of Two-Dimensional Manifolds},
author={J.H. Roberts}, 
author={N.E. Steenrod},
journal={Annals of Mathematics}, 
volume={39}, 
year={1938},
pages={851--862}}

%
%
%

\bib{Samba}{book}{
author={M. Sambarino},
title={Estructura local de conjuntos estables e inestables de homeomorfismos en superficies},
year={1993},
publisher={Universidad de la República, Uruguay, Monograph}}


%
%

\bib{Sm}{article}{
author={M. Smith},
title={A theorem on continuous decompositions of the plane into nonseparating continua},
journal={Proceedings of the AMS},
volume={55},
pages={221--222},
year={1976}}

%
%
\bib{Vi2002}{article}{
author={J.L. Vieitez},
title={Lyapunov functions and expansive diffeomorphisms on 3D-manifolds},
journal={Ergodic Theory and Dynamical Systems},
year={2002},
volume={22},
pages={601--632}}

%

\bib{WaltersET}{book}{
author={P. Walters},
title={\azul{An introduction to ergodic theory}},
publisher={Springer-Verlag New York, Inc.},
year={1982}}


\bib{Whitney33}{article}{
author={H. Whitney},
title={Regular families of curves},
journal={Ann. of Math.}, 
number={34},
year={1933}, 
pages={244--270}}

\end{biblist}
\end{bibdiv}
\end{document}